\documentclass[11pt]{article}
\usepackage[T1]{fontenc}
\usepackage{lmodern}
\usepackage[utf8]{inputenc}
\usepackage[english]{babel}
\usepackage[a4paper,margin=27mm]{geometry}
\usepackage{amsmath,amsfonts,amssymb,amsthm,mathtools,bm,mathrsfs,cite}
\usepackage{graphicx,cases,booktabs}

\usepackage[dvipsnames]{xcolor}
\usepackage{tikz}
\usetikzlibrary{positioning,arrows.meta,decorations.pathreplacing,calc}
\usepackage[prependcaption,textsize=tiny,textwidth=3cm]{todonotes}
\usepackage{microtype}
\makeatletter
\def\thmhead@plain#1#2#3{%
  \thmname{#1}\thmnumber{\@ifnotempty{#1}{ }\@upn{#2}}%
  \thmnote{ {\the\thm@notefont#3}}}
\let\thmhead\thmhead@plain
\makeatother

\usepackage[hidelinks]{hyperref}
\hypersetup{colorlinks = true,
linkcolor = MidnightBlue, anchorcolor =red,
citecolor = ForestGreen,
filecolor = red,urlcolor = red,
            pdftitle={Periodic Fabry-Perot resonances}}
\usepackage{cleveref}

\newtheorem{theorem}{Theorem}[section]
\newtheorem{lemma}[theorem]{Lemma}
\newtheorem{proposition}[theorem]{Proposition}
\newtheorem{corollary}[theorem]{Corollary}

\numberwithin{equation}{section}

\theoremstyle{remark}
\newtheorem{remark}[theorem]{Remark}

\theoremstyle{definition}
\newtheorem{definition}[theorem]{Definition}
\newtheorem{example}[theorem]{Example}

\newcommand{\mr}[1]{{\color{red} #1}}

\newcommand{\mc}[1]{\mathcal{#1}}

\def\bb{\begin{equation}
  \left\{\ 
   \begin{aligned} }
\def\ee{   \end{aligned}
  \right.
  \end{equation}}

\def\mm{ \left[
 \begin{matrix}}
\def\nn{\end{matrix} \right] }

\def \dd{\mathrm{d}}

\def\d{\delta}

\def\s{\sigma}

\def \diag{\operatorname{diag}}

\def\re{\mathfrak{Re}}
\def\im{\mathfrak{Im}}

\def \sdiv {\sdiv_S}

\def \b{\overline}

\def \R{\mathbb{R}}

\def \C{\mathbb{C}}

\def \Z{\mathbb{Z}}

\def \e{\mathrm{e}}
\def \i{\mathrm{i}}

\def \L{\mathcal{L}}

\def \O{\mathcal{O}}

\begin{document}

\title{Frequency-dependent capacitance matrix formulation for Fabry--P\'erot resonances. Part II: One-dimensional periodic media}

\author{Habib Ammari\footnote{
  ETH Z\"urich, Department of Mathematics, Rämistrasse 101, 8092 Z\"urich, Switzerland, (habib.ammari@math.ethz.ch).} \and Bowen Li\footnote{Department of Mathematics,  City University of Hong Kong, Kowloon Tong, Hong Kong SAR, (bowen.li@cityu.edu.hk).}
\and Ping Liu\footnote{
School of Mathematical Sciences, Zhejiang University, Hangzhou, 310027, China, (pingliu@zju.edu.cn, shaoyingjie323@zju.edu.cn).}  \thanks{
Institute of
Fundamental and Transdisciplinary Research, Zhejiang University, Hangzhou, 310027, China.}
\and Yingjie Shao\footnotemark[3]
}

\date{}
\maketitle
\begin{abstract}
We study Bloch resonances of one-dimensional periodic systems of high-contrast resonators beyond the subwavelength regime. An exact transfer-matrix formulation identifies the limiting wavenumbers with the union of the interior Neumann frequencies and the Dirichlet frequencies of the separating intervals. At each fixed reference wavenumber, we introduce a periodic frequency-dependent capacitance matrix whose positive eigenvalues determine all leading square-root splittings in the contrast parameter. A direct cyclic matching expansion proves the approximation, including multiple eigenvalues and small complex contrasts. A spectral exclusion result outside the limiting set yields a direct bandgap statement. At a common resonance of all intervals, we also derive uniform eigenmode approximations and distinguish the zero-frequency mode from nonzero double Bloch eigenvalues. For positive contrast, exact linear crossings are characterized
by the monodromy matrix being proportional to the identity.
A complete local classification is obtained for two resonators
per cell and, through band folding, for fully symmetric arrays
with any number of resonators per cell. We further formulate the weighted Zak phase for isolated bands and its quantization under reflection symmetry. Explicit two-resonator formulas and numerical illustrations complement the analysis.
\end{abstract}

\section{Introduction}\label{sec:introduction}
High-contrast resonator systems provide a setting in which the
spectral properties of a continuous wave equation can be related to
those of a finite-dimensional interaction matrix. In the
subwavelength regime, the generalized capacitance matrix describes
the leading resonant frequencies and the amplitudes of the
corresponding modes. This approach has led to rigorous descriptions
of bandgap opening, localization, and interface modes; see
\cite{cbms136,discrete2024,ammari2017subwavelength,ammari2023edge,
topological2020,ammari2024exponentially}. In one dimension, the
geometry also allows an exact propagation-matrix description,
which complements the asymptotic reduction
\cite{Subwavelength_1D_High-Contrast,pm1}.

The present paper continues the analysis of finite one-dimensional
systems in~\cite{fabry1} by considering an infinite periodic chain.
Replacing the outgoing condition by a Bloch condition changes both
the spectral problem and the coupling matrix. For positive material
parameters and real quasimomentum, the Bloch problem is
self-adjoint and its wavenumbers are real. The additional spacing
across the boundary of a period cell closes the interaction graph
into a cycle and introduces a Bloch phase into the corner entries
of the reduced matrix. These features are responsible for the
dispersion and exact band crossings studied below.

Our results are related to the frequency-dependent capacitance
formulation for two- and three-dimensional systems
in~\cite{fabry23}. That work develops reductions near fixed
nonzero reference frequencies, proves periodic bandgap opening under
a uniform exterior non-resonance condition, and establishes generic
Dirac cones for a honeycomb scaling family. In one dimension, the
separating intervals are bounded components of the exterior.
Their Dirichlet frequencies must therefore be included in the
limiting spectrum. The explicit interval geometry allows us to
resolve coincidences between these frequencies and the interior
Neumann frequencies, which produce square-root shifts. The
one-dimensional linear crossings considered here are closed gaps
of a scalar periodic problem, rather than two-dimensional Dirac cones.

For a cell containing $N$ resonators, we first derive an exact
dispersion equation $\operatorname{tr}A(k;\delta)=2\cos(\alpha L)$,
where $A$ is the transfer matrix for one period. A suitable
normalization has a limit equal to a product of $2N$ sine
functions. Thus, at a fixed limiting wavenumber $k_0$, the number
of nearby zeros, counted in the wavenumber plane, is the number
$n(k_0)$ of resonant intervals in the cell. This count includes
resonators and spacings, whose optical lengths need not be equal.

The leading splitting is governed by a cyclic matrix
$\mathcal C^\gamma(k_0)$. Its weights are determined by adjacent
resonant intervals, and its effective Bloch phase includes the signs
accumulated by standing waves at $k_0$. If $\lambda_j>0$ is an
eigenvalue of this matrix, two wavenumbers have the expansions
\[
 k_j^\pm(\delta)=k_0\pm\sqrt{\lambda_j\delta}+\O(\delta).
\]
All remaining branches lie within $\O(\delta)$ of $k_0$.
The proof uses a direct expansion over matchings of the cyclic
graph. This gives the coefficients of the first rescaled polynomial
without requiring higher Newton-polygon edges to be attained. The result is local at each fixed $k_0$ and fixed quasimomentum. We also prove a spectral exclusion result: on any compact set disjoint from the limiting set, no Bloch wavenumber exists for sufficiently small positive contrast, which yields a direct bandgap statement.

When every interval is resonant at a nonzero $k_0$, the leading
eigenmodes are described by coupled resonator cosine amplitudes and
spacing sine amplitudes. Their different contrast scalings are
essential. At the phase for which the reduced matrix has a zero
eigenvalue, $k_0$ remains an exact double Bloch wavenumber. The same
double zero in the wavenumber variable at $k_0=0$ instead corresponds
to the single constant mode. This distinction is necessary when
interpreting band crossings.

For positive contrast, we prove that an exact
Dirac point occurs precisely when the monodromy is $+I$ or $-I$
at a nonzero wavenumber. Positivity of a Gram matrix yields a
nonzero crossing slope. We then classify every nearby Dirac point
for $N=2$, including both points fixed at their reference frequency
and contrast-dependent points.
For completely symmetric configurations, all Dirac points
are further classified for arbitrary \(N\) through band folding.
For a cell consisting of repeated identical resonator--spacing
pairs, the latter are explained by folding the bands of a smaller
period. Their interpretation must therefore retain the chosen
period cell; see the general Floquet framework
in~\cite{kuchment2016overview} and the one-dimensional analysis
in~\cite{Lin2022Topological}.

Finally, we discuss the Zak phase of simple isolated bands
\cite{zak1989}, motivated by the SSH model \cite{ssh1979}
and related studies of subwavelength resonator chains
\cite{topological2020}.
We specify the material-weighted inner product,
the reciprocal-lattice endpoint identification, and the phase conventions for quasiperiodic and periodic
Bloch functions. Under reflection symmetry, we prove
that the Zak phase is quantized and determined by endpoint
reflection eigenvalues. Using asymptotic eigenmode profiles
to identify these eigenvalues, we classify the Zak phases
for symmetric two-resonator configurations.

Section~2 fixes the geometry and notation.
Section~\ref{sec:transfer} gives the exact dispersion relation.
Section~\ref{sec:capacitance} proves the capacitance approximation,
the spectral-exclusion result, and the eigenmode estimates.
Section~\ref{sec:dirac} characterizes and classifies exact crossings, and 
Section~\ref{sec: Zak phase} determines the Zak phases under reflection symmetry. 
Section~\ref{sec:conclusion} discusses the scope of the results and further questions. Appendix~\ref{app:asymptotics N2} collects explicit resonance
and eigenmode expansions for two resonators per cell.

\section{Problem formulation}
Let $N\in\Z_+$. A period cell contains disjoint resonators
$D_j=(x_j^-,x_j^+)$, $j=1,\ldots,N$, ordered from left to right.
Set $x_1^-=0,x_{N+1}^-=x_1^-+L$, and extend all interfaces and material
coefficients periodically. Write
\[
 \ell_j=x_j^+-x_j^->0,\qquad s_j=x_{j+1}^--x_j^+>0,
 \qquad L=\sum_{j=1}^N(\ell_j+s_j).
\]
The resonator set in the line is
$D=\bigcup_{n\in\mathbb Z}\bigcup_{j=1}^N(D_j+nL)$.
The exterior density and the bulk modulus are $\rho,\kappa>0$;
the corresponding interior parameters are $\rho_b,\kappa_j>0$.
We use
\begin{equation*}\label{eq:contrast}
 v_j=\sqrt{\kappa_j/\rho_b},\quad v=\sqrt{\kappa/\rho},\quad
 k=\omega/v,\quad k_j=r_jk,\quad
 r_j=v/v_j>0,\quad \delta=\rho_b/\rho.
\end{equation*}
The lengths and the ratios $r_j$ remain fixed as $\delta\to0$.
We also allow analytic continuation to complex $\delta\ne0$;
positivity is explicitly imposed for self-adjoint and Dirac-point results.
% The one-resonator case is recorded separately in Remark \ref{rem: N=1}.

For real quasimomentum $\alpha\in Y^*=(-\pi/L,\pi/L]$, the Bloch problem is
\begin{equation}\label{equ: Bloch problem}
\begin{cases}
 u''+k^2u=0&\text{in }(x_j^+,x_{j+1}^-),\quad j=1,\ldots,N,\\
 u''+r_j^2k^2u=0&\text{in }(x_j^-,x_j^+),\quad j=1,\ldots,N,\\
 u|_+=u|_-&\text{at each interface},\\
 u'_+(x_j^-)=\delta u'_-(x_j^-),\quad
 u'_-(x_j^+)=\delta u'_+(x_j^+)&j=1,\ldots,N,\\
 u(x_{N+1}^-)=\e^{\i\alpha L}u(x_1^-),&\\
 u'_-(x_{N+1}^-)=\e^{\i\alpha L}u'_-(x_1^-).&
\end{cases}
\end{equation}
Here, $w|_\pm(x)=\lim_{h\downarrow0}w(x\pm h)$ denote spatial
one-sided traces. They are not interior/exterior labels.
The final derivative condition uses the exterior trace at both ends.
The interface conditions then give the same Bloch condition for the
interior traces. Figure~\ref{fig:geometry} illustrates this convention.
\begin{figure}[htbp]
\centering
\begin{tikzpicture}[x=0.91cm,y=0.9cm,>=Stealth]
\draw[->] (-.5,0)--(13,0) node[right] {$x$};
\foreach \a/\b/\j in {0/1.7/1,3/4.5/2,8.2/10/N}{
 \filldraw[fill=blue!12,draw=blue!65!black] (\a,0) rectangle (\b,.75);
 \node at ({(\a+\b)/2},.38) {$D_{\j}$};
 \node[below] at (\a,0) {$x_{\j}^-$};
 \node[below] at (\b,0) {$x_{\j}^+$};
 \draw[<->] (\a,1.1)--node[above] {$\ell_{\j}$}(\b,1.1);
}
\draw[<->] (1.7,1.1)--node[above] {$s_1$}(3,1.1);
\draw[<->] (10,1.1)--node[above] {$s_N$}(12,1.1);
\node at (6.3,.4) {$\cdots$};
\draw[dashed] (12,-.1)--(12,.8);
\node[below] at (12,0) {$x_{N+1}^-$};
\draw[<->] (0,-1)--node[fill=white] {$L$, Bloch multiplier $\e^{\i\alpha L}$}(12,-1);
\end{tikzpicture}
\caption{A period cell containing $N$ resonators
$D_j=(x_j^-,x_j^+)$ of lengths $\ell_j$, separated by
spacings of lengths $s_j$. The last spacing connects
$D_N$ to the first resonator in the next cell.
The cell has length $L$ and Bloch multiplier $\e^{\i\alpha L}$.}
\label{fig:geometry}
\end{figure}
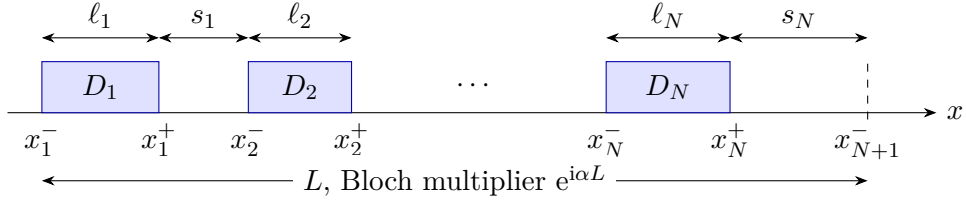
\begin{definition}
We call \(\omega\) a \textit{resonant frequency} (or \textit{resonance}) if (\ref{equ: Bloch problem}) admits a nontrivial solution \(u\), which we refer to as an \textit{eigenmode}. By a slight abuse of notation, we also call the corresponding wavenumber \(k=\omega/v\) a \textit{resonance} when $\omega$  is a resonant frequency. The dependence of these resonances on the Bloch parameter \(\alpha\in Y^*\), expressed as \(\alpha\mapsto k(\alpha)\), defines the \textit{Bloch dispersion relation}.    
\end{definition}

When $u$ is an eigenmode of (\ref{equ: Bloch problem}), it extends uniquely to an
\(\alpha\)-quasiperiodic function on \(\R\) by setting
\[
u(x+nL):=\e^{\i n\alpha L}u(x),
\qquad x\in[x_1^-,x_{N+1}^-),\quad n\in\Z.
\]
The Bloch boundary conditions ensure that this extension
satisfies the same differential equations and transmission
conditions throughout the periodically extended medium.
We use the same notation \(u\) for the extended eigenmode.

\section{Characterization of Bloch resonances}\label{sec:transfer}

We characterize the resonances of \((\ref{equ: Bloch problem})\). To express the transmission
conditions of
\((\ref{equ: Bloch problem})\) as continuity conditions, we define
\begin{align}\label{equ: V(x) def}
V(x):=
\begin{cases}
\bigl(u(x),\,\delta^{-1/2}u'(x)\bigr)^{\top},
&x\in(x_j^-,x_j^+),\\
\bigl(u(x),\,\delta^{1/2}u'(x)\bigr)^{\top},
&x\in(x_j^+,x_{j+1}^-),
\end{cases}
\qquad j=1,\ldots,N,
\end{align}
with a fixed branch of \(\delta^{1/2}\). The transmission conditions ensure that
the left and right limits of \(V\) agree at each resonator interface
\(x_j^\pm\). We assign \(V\) this common limiting value and denote
it by
\begin{align}\label{equ: V_j^pm def}
V_j^\pm:=V(x_j^\pm)
=V|_-(x_j^\pm)
=V|_+(x_j^\pm).
\end{align}
Thus $V\in C\bigl([x_1^-,x_{N+1}^-];\mathbb C^2\bigr)$. The differential equations in the first two lines of
\((\ref{equ: Bloch problem})\) give the propagation relations
\begin{align}\label{equ: V(x) propagation}
\begin{cases}
V(x+a)=T_j^{\mathrm{Res}}(a)V(x),
&x,x+a\in[x_j^-,x_j^+],\\
V(x+a)=T^{\mathrm{Sp}}(a)V(x),
&x,x+a\in[x_j^+,x_{j+1}^-],
\end{cases}
\qquad j=1,\ldots,N,
\end{align}
where, for \(k\ne0\),
\[
T_j^{\mathrm{Res}}(a)=
\begin{pmatrix}
\cos(r_jka) &
\dfrac{\delta^{1/2}}{r_jk}\sin(r_jka)\\
-\dfrac{r_jk}{\delta^{1/2}}\sin(r_jka) &
\cos(r_jka)
\end{pmatrix},
\]
\[
T^{\mathrm{Sp}}(a)=
\begin{pmatrix}
\cos(ka) &
\dfrac{1}{k\delta^{1/2}}\sin(ka)\\[2mm]
-k\delta^{1/2}\sin(ka) &
\cos(ka)
\end{pmatrix}.
\]
The apparent singularities at \(k=0\) are removable. At \(k=0\),
these matrices are defined by their limits:
\[
T_j^{\mathrm{Res}}(a)\big|_{k=0}
=
\begin{pmatrix}
1&\delta^{1/2}a\\
0&1
\end{pmatrix},
\qquad
T^{\mathrm{Sp}}(a)\big|_{k=0}
=
\begin{pmatrix}
1&\delta^{-1/2}a\\
0&1
\end{pmatrix}.
\]
These matrices also follow directly from the differential equations
at \(k=0\). Hence, \((\ref{equ: V(x) propagation})\) holds for every
\(k\in\mathbb C\). In particular,
\begin{align}\label{equ: V_j^pm propagation}
V_j^+=T_j^{\mathrm{Res}}V_j^-,
\qquad
V_{j+1}^-=T_j^{\mathrm{Sp}}V_j^+,
\qquad j=1,\ldots,N,
\end{align}
where
\begin{align}\label{equ: T_j Res Sp def}
T_j^{\mathrm{Res}}:=T_j^{\mathrm{Res}}(\ell_j),
\qquad
T_j^{\mathrm{Sp}}:=T^{\mathrm{Sp}}(s_j),
\qquad j=1,\ldots,N.
\end{align}
These propagation relations reduce the resonance problem to a
condition on the transfer matrix over one cell.

\begin{theorem}\label{thm: resonant conition original}
For \(k\in\mathbb C\), define
\begin{align}\label{equ: A(k;delta) def}
A(k;\delta)=
T_N^{\mathrm{Sp}}T_N^{\mathrm{Res}}\cdots
T_2^{\mathrm{Sp}}T_2^{\mathrm{Res}}
T_1^{\mathrm{Sp}}T_1^{\mathrm{Res}},\qquad
D(k;\delta)=\operatorname{tr}A(k;\delta).
\end{align}
Then \(k\) is a resonance of
\((\ref{equ: Bloch problem})\) if and only if $D(k;\delta)=2\cos(\alpha L)$.
\end{theorem}

\begin{proof}
By \((\ref{equ: V_j^pm propagation})\), $V_{N+1}^-=A(k;\delta)V_1^-$. The Bloch conditions in the last line of
\((\ref{equ: Bloch problem})\) give $V_{N+1}^-=\e^{\i\alpha L}V_1^-$. Conversely, any nonzero initial vector \(V_1^-\) that satisfies both
relations determines, by propagation through the cell, a nontrivial
solution of (\ref{equ: Bloch problem}). Thus, \(k\) is a resonance if and only if
\(\e^{\i\alpha L}\) is an eigenvalue of \(A(k;\delta)\).

Since each $T_j^{\text{Res}}$ and $T_j^{\text{Sp}}$ has determinant one, so does $A(k;\delta)$. Thus, \(\e^{\i\alpha L}\) is an eigenvalue of \(A(k;\delta)\) if and only if 
\[
\det\!\left(A(k;\delta)-\e^{\i\alpha L}I\right)
=
\e^{2\i\alpha L}
-D(k;\delta)\e^{\i\alpha L}+1
\]
vanishes. This determinant vanishes if and only if $D(k;\delta)=2\cos(\alpha L)$.
\end{proof}
\begin{example}\label{ex: k=0}
Using Theorem \ref{thm: resonant conition original}, the zero-frequency case can be described explicitly. When $\textstyle\delta\ne-(\sum_{j=1}^N s_j)/(\sum_{j=1}^N\ell_j)$, the propagation
matrices at \(k=0\) give
\[
A(0;\delta)
=
\begin{pmatrix}
1&S_\delta\\
0&1
\end{pmatrix},
\qquad
S_\delta
=
\sum_{j=1}^{N}
\bigl(\delta^{1/2}\ell_j+\delta^{-1/2}s_j\bigr)\ne0.
\]
Consequently, using \Cref{thm: resonant conition original}, \(k=0\) is a resonance if and only if
\(\e^{\i\alpha L}=1\). In that case,
\[
\operatorname{Ker}\bigl(A(0;\delta)-I\bigr)
=
\operatorname{span}\left\{(1,0)^\top\right\},
\]
so its eigenmodes are precisely the nonzero constant functions.    
\end{example}

The preceding theorem characterizes the resonances through the
propagation matrices. For a real Bloch parameter, positivity of the
coefficients also implies that every resonance is real.

\begin{proposition}\label{thm: k real}
Assume that \(\delta>0\) and
\(\alpha\in Y^*\). Then every resonance \(k\) of
\((\ref{equ: Bloch problem})\) is real.
\end{proposition}

\begin{proof}
Define the piecewise constant functions \(a,b\) on the cell by
\[
a(x)=
\begin{cases}
\delta^{1/2}, & x\in(x_j^-,x_j^+),\\
\delta^{-1/2}, & x\in(x_j^+,x_{j+1}^-),
\end{cases}
\qquad
b(x)=
\begin{cases}
r_j^2\delta^{-1/2}, & x\in(x_j^-,x_j^+),\\
\delta^{1/2}, & x\in(x_j^+,x_{j+1}^-),
\end{cases}
\qquad j=1,\ldots,N.
\]
Thus, \(a,b>0\), and the continuous vector \(V\) satisfies
\begin{align}\label{equ: V(x) diff equ}
V'=K(x,k)V,
\qquad
K(x,k)=
\begin{pmatrix}
0&a(x)\\
-k^2b(x)&0
\end{pmatrix}
\end{align}
on each subinterval. Write \(V=(v_1,v_2)^\top\). Then $v_1'=av_2,v_2'=-k^2bv_1$ with $a,b>0$. Hence,
\[
\frac{\mathrm d}{\mathrm dx}
\bigl(\overline{v_1}v_2\bigr)
=
a|v_2|^2-k^2b|v_1|^2.
\]
Integrating from $x_1^-$ to $x_{N+1}^-$ and using
the Bloch condition gives
\[
k^2\int_{x_1^-}^{x_{N+1}^-}b|v_1|^2\,\mathrm dx
=
\int_{x_1^-}^{x_{N+1}^-}a|v_2|^2\,\mathrm dx.
\]
The integral on the left without \(k^2\) is strictly
positive. Otherwise, \(b>0\) would imply \(v_1\equiv0\),
and \(v_1'=av_2\), with \(a>0\), would then imply
\(v_2\equiv0\), contradicting the nontriviality of \(V\).
Since the integral on the right is real and nonnegative,
\(k^2\) is also real and nonnegative. Therefore
\(k\in\mathbb R\).
\end{proof}

\begin{remark}
The reality of $k$ in Proposition \ref{thm: k real} is a consequence of the Bloch condition, which causes the boundary term to vanish. In contrast, for finite scattering problems with an outgoing radiation
condition, the corresponding boundary term is nonzero and represents energy leakage. Hence, the
resonances are generally complex and typically have a negative imaginary part with the time convention $\e^{-\i\omega t}$, see \cite[Theorem
2.2(2)]{pm1}.
\end{remark}
\section{Capacitance matrix approximation}\label{sec:capacitance}
In this section, we derive a capacitance matrix approximation
for the Bloch resonances of (\ref{equ: Bloch problem}) as \(\delta\to0\). Starting from the exact dispersion relation in Theorem \ref{thm: resonant conition original},
we obtain a normalized equation and show that the leading
local resonance splittings are governed by the periodic frequency-dependent capacitance matrix \(\mathcal C^\gamma(k_0)\). Section \ref{subsec: resonances} establishes the resonance expansions near
a fixed limiting wavenumber \(k_0\in E\), including the
square-root splittings and the remaining \(\mathcal O(\delta)\) branches. A spectral exclusion result on compact sets outside \(E\)
is also obtained. Section \ref{subsec: eigenmodes} describes the corresponding eigenmodes when all intervals are resonant at \(k_0\), that is,
\(n(k_0)=2N\).
\subsection{Resonances}\label{subsec: resonances}
We begin with the resonance asymptotics as $\delta\to0$. To facilitate the subsequent analysis and computation, we introduce the following auxiliary matrices. Let  
\begin{align} \label{def:rlmatrix}
    \sigma_i=\frac{\d}{r_i},1\leq i\leq N,\quad R(z) := \frac{1}{2}\begin{pmatrix}
        {1+ z} & {1-z}\\
        {1- z} & {1+z}
    \end{pmatrix},
    & \quad
    L(z) := \begin{pmatrix}
         \e^{\i z} & 0\\
        0 & \e^{-\i z}
    \end{pmatrix}.
\end{align}
Define the vector  
\begin{equation} \label{def:vectort}
    \mathfrak{\bm t}:=(r_1\ell_1,\ s_{1},\ r_2\ell_{2},\ s_{2},\ \cdots,\ r_{N-1}\ell_{N-1},\ s_{N-1},\ r_N\ell_N,s_N)^\top\in\R^{2N}_{>0},
\end{equation}
whose \(j\)-th component is denoted by \(\mathfrak t_j\), and set  
\[
M_{\mathrm{tot}}(k;\delta)
=
\prod_{i=N}^1
L(\mathfrak t_{2i}k)R(\sigma_i^{-1})
L(\mathfrak t_{2i-1}k)R(\sigma_i),
\]
with the product ordered as \(i=N,\dots,1\). For \(k\ne0\), let \(P\) have columns
\((\delta^{-1/2},\i k)^\top\) and
\((\delta^{-1/2},-\i k)^\top\).
Recalling \eqref{equ: A(k;delta) def}, a direct computation gives
\[
P^{-1}T_i^{\text{Sp}}T_i^{\text{Res}}P
=L(\mathfrak{t}_{2i}k)R(\sigma_i^{-1})
 L(\mathfrak{t}_{2i-1}k)R(\sigma_i),
\]
and thus, $A(k;\delta)$ and $M_{\text{tot}}(k;\delta)$ are similar. Consequently, they have the same trace:
\[
D(k;\delta)=
\operatorname{tr}M_{\mathrm{tot}}(k;\delta),\quad k\ne0.
\]
Both sides are analytic functions of \(k\) since the propagation matrices
on the left are understood at \(k=0\) by their removable extensions.
Thus, the above equation also holds at $k=0$.
Let
\begin{align}\label{eq: f(z; sigma) expansion}
f(k;\delta)=\operatorname{tr}\left(\prod_{i=1}^N\frac{4\sigma_i}{(1+\sigma_i)^{2}}M_{\text{tot}}(k;\delta)\right)-\prod_{i=1}^N\frac{4\sigma_i}{(1+\sigma_i)^{2}}\cdot2\cos(\alpha L)
\end{align}
be analytic in $k$. In this section, we restrict $|\delta|<\min_j r_j$,
so the normalizing factors have no poles; for $\delta\ne0$ they are nonzero.
Theorem \ref{thm: resonant conition original} implies that $k\in\C$ is a resonance of (\ref{equ: Bloch problem}) if and only if  $k$ is a zero of $f(k;\delta)$.
This equivalent formulation will be used in the subsequent analysis. From definition (\ref{def:rlmatrix}) we obtain
\begin{align*}
\lim_{\sigma\to0}\frac{2}{1+\s}R(\s)=R_+:=\begin{pmatrix}
    1&1\\1&1
\end{pmatrix},\quad \lim_{\sigma\to0}\frac{2\s}{1+\s}R\left(\frac{1}{\s}\right)=R_-:=\begin{pmatrix}
    1&-1\\-1&1
\end{pmatrix}.
\end{align*}
Thus, as $\delta\to0$, we have $\sigma_i\to0,1\leq i\leq N$. Therefore, the matrix $\prod_{i=1}^N\frac{4\sigma_i}{(1+\sigma_i)^{2}}M_{\text{tot}}(k;\delta)$ uniformly converges to
\begin{align*}
&L(\mathfrak{t}_{2N}k)R_-L(\mathfrak{t}_{2N-1}k)R_+L(\mathfrak{t}_{2N-2}k)R_-\cdots R_-L(\mathfrak{t}_1k)R_+\\=&(2\i)^{2N-1}\prod_{j=1}^{2N-1}\sin (\mathfrak{t}_jk)\cdot\begin{pmatrix}
    \e^{\i \mathfrak{t}_{2N}k}&\e^{\i \mathfrak{t}_{2N}k}\\-\e^{-\i \mathfrak{t}_{2N}k}&-\e^{-\i \mathfrak{t}_{2N}k}.
\end{pmatrix},
\end{align*}
on any compact set $K\subset\C$. 
We then have the following lemma.
\begin{lemma}\label{thm: f(k;0) form}
When $\delta \to 0$, the analytic function $f(k;\delta)$ in \eqref{eq: f(z; sigma) expansion} uniformly converges to
\begin{align*}
    f(k;0)=(2\i)^{2N}\prod_{j=1}^{2N}\sin (\mathfrak{t}_j k).
\end{align*}
on any compact set $K\subset\C$. Then $k_0\in \C$ is a zero of $f(\cdot,0)$ if and only if $k_0\in E: = \cup_{j=1}^{2N}(\pi \Z/\mathfrak{t}_j)$. When $k_0\in E$, $k_0$ is a zero of order $n$, where
\begin{align}\label{equ: I def}
n=n(k_0):=\# I,\quad I=I(k_0):=\left\{j\middle|\mathfrak{t}_j k_0\in \pi\Z,1\leq j\leq2N\right\}.
\end{align}
In particular, all zeros of $f(\cdot;0)$ are real and $0$ is a zero of order $2N$.
\end{lemma}
Since $f(k;\delta)$ uniformly converges to $f(k;0)$ on any compact set as $\delta\to0$, Rouch\'e's Theorem implies that there are exactly $n(k)$ zeros of $f(\cdot;\delta)$ in a small neighborhood of $k$ for every $k\in E$ when $|\delta|$ is small enough. This shows that (\ref{equ: Bloch problem}) has exactly $n(k)$ resonances near $k$ for every $k\in E$ when $\d\to0$. Fix a $k_0\in E$, our next target is to derive the asymptotic form of the $n(k_0)$ zeros of $f(k;\delta)$ near $k_0\in E$ when $\delta\to 0$. 

According to (\ref{equ: I def}), $j \in I$ if and only if there exists $m_j \in \mathbb{Z}$ such that 
\begin{equation}\label{equ:defmj}
\mathfrak{t}_j k_0 = m_j \pi.
\end{equation}
Similar to the setting in \cite{fabry1}, we set
\begin{equation} \label{def:vector bm t}
    {\bm t}:=(r_1^2\ell_1,\ s_{1},\ r_2^2\ell_{2},\ s_{2},\ \cdots,\ r_{N-1}^2\ell_{N-1},\ s_{N-1},\ r_N^2\ell_N,s_N)^\top\in\R^{2N}_{>0}
\end{equation}
and use $t_j$ to denote the $j$\textsuperscript{th} component of ${\bm t}$. We then define a wavenumber-dependent parameter 
\begin{align}\label{equ: t_j(k) def}
t_j(k_0) = 
\begin{cases}
    t_j, & \text{if } \mathfrak{t}_j k_0\in\pi\mathbb Z\text{ (or equivalently $j\in I$)}, \\
    \infty, & \text{otherwise},
\end{cases}
\end{align}
with the convention that $1/\infty = 0$. Based on this definition, we introduce the coupling coefficients
\begin{equation} \label{eq:coupling}
     \theta_j = \frac{1}{t_j t_{j+1}}, \quad \theta_j(k_0) = \frac{1}{t_j(k_0) t_{j+1}(k_0)}, \ 1 \leq j \leq 2N. 
\end{equation}
with the convention that $t_{2N+1}=t_{1},t_{2N+1}(k_0)=t_{1}(k_0)$.
For real $\gamma$, define the periodic frequency-dependent capacitance
matrix by its action:
\begin{equation}\label{eq:mathcalCkgamma_def}
 (\mathcal C^\gamma(k_0)q)_i
 =\theta_{2i-2}(k_0)(q_i-q_{i-1})
  +\theta_{2i-1}(k_0)(q_i-q_{i+1}),\qquad 1\leq i\leq N,
\end{equation}
where $q_0=\e^{-\i\gamma L}q_N$, $q_{N+1}=\e^{\i\gamma L}q_1$,
and $\theta_0=\theta_{2N}$. Contributions to the same entry are
\emph{added}; in particular, for $N=2$ the internal and wrap-around
off-diagonal terms both contribute. Equivalently,
\[
 \mathcal C^\gamma(k_0)=\mathcal C(k_0)
 -\theta_{2N}(k_0)\e^{-\i\gamma L}E_{1N}
 -\theta_{2N-1}(k_0)\e^{\i\gamma L}E_{N1},
\]
where $\mathcal C(k_0)$ has diagonal
$\theta_{2i-2}(k_0)+\theta_{2i-1}(k_0)$, upper diagonal
$-\theta_{2i-1}(k_0)$, and lower diagonal $-\theta_{2i}(k_0)$.
% \begin{remark}
% Although we write \(n = n(k_0),m = m(k_0),t_j(k_0),\theta_j(k_0)\) in the periodic system to emphasize their dependence on \(k_0 \in E\), all these quantities are in fact completely determined by \(I = I(k_0)\), see (\ref{equ: I def})(\ref{def:m_periodic})(\ref{equ: t_j(k) def}). Consequently, whenever convenient we may also denote them by \(n = n(I),m = m(I),t_j(I),\theta_j(I)\) without ambiguity. For the same reason, we may write $\tilde{n} = \tilde{n}(\tilde{I})$ and $\tilde{m} = \tilde{m}(\tilde{I})$ in the finite system.
% \end{remark}
\begin{remark}\label{rem:reference-matrix}
The dependence on $k_0$ in \eqref{eq:mathcalCkgamma_def}
is through the resonant set $I(k_0)$.
The matrix is associated with a fixed reference wavenumber
rather than an analytic function of a continuously varying
frequency. At $k_0=0$, it is the generalized quasiperiodic
capacitance matrix introduced in \cite{ammari2023edge}. At a nonzero coincidence of interior
Neumann and spacing Dirichlet frequencies, it describes
the leading singular coupling.
%All proofs below allow different $r_j$ and derive the
%matching expansion directly.
\end{remark}

Fix \(k_0\in E\). Regard \(I=I(k_0)\) as a subset of
\[
\Omega = \mathbb{Z}/(2N\mathbb{Z}) = \{1,2,\dots,2N\},
\]
where the representatives are chosen as the standard set $\{1,2,\dots,2N\}$. On $\Omega$ we define the cyclic distance $d(x,y) = \min_{k \in \mathbb{Z}} \bigl| (x - y) + 2Nk \bigr|$. It is straightforward to verify that $d$ is a metric, turning $(\Omega,d)$ into a finite metric space.
To describe the connected components of $I$ in the cyclic sense, on $I$ we define an equivalence relation $\sim$ by  
\[
x \sim y \;\Longleftrightarrow\; \exists s\geq0,\{v_i\}_{i=0}^s\subset I \text{ such that }v_0=x,v_s=y,d(v_i,v_{i-1})=1,1\leq i\leq s.
\]
The equivalence classes $\mathcal{I}_1, \dots, \mathcal{I}_p$ of $\sim$ are called the cyclic maximal connected components of $I$. Each $\mathcal{I}_j$ is a maximal subset of $I$ whose elements form an uninterrupted chain under the cyclic adjacency. Let $n_j = \#\mathcal{I}_j$ so that $\sum_{j=1}^p n_j = n = \#I$. Define
\begin{equation}\label{def:m_periodic}
m = m(k_0) := \sum_{j=1}^{p} \left\lfloor \frac{n_j}{2} \right\rfloor.
\end{equation}
For \(l\ge1\), set
\[
\Omega_l=
\left\{
(j_1,\dots,j_l)\in\Omega^l:
j_1<\cdots<j_l,\;
d(j_a,j_b)\ge2\ \text{for }a\ne b
\right\},
\]
where the ordering uses the chosen representatives.
Define
\begin{equation}\label{equ: C^per_I,l def}
C^{\text{per}}_{I,l}
=
\sum_{\substack{
(j_1,\dots,j_l)\in\Omega_l\\
\{j_i,j_i+1\}\subset I,\;1\le i\le l}}
\prod_{i=1}^l\theta_{j_i},
\qquad
C^{\text{per}}_{I,0}=1,\quad
C^{\text{per}}_{I,-1}=0.
\end{equation}
Here \(j_i+1\) is understood modulo \(2N\).
The selected pairs are disjoint edges within the cyclic
components of \(I\), so their number cannot exceed \(m\).
Consequently, \(C^{\text{per}}_{I,l}=0\) for \(l>m\),
with empty sums interpreted as zero. These coefficients yield the following spectral result
for the periodic capacitance matrix
\(\mc C^\gamma(k_0)\) defined in
\eqref{eq:mathcalCkgamma_def}.

\begin{lemma}[Spectrum of the periodic frequency-dependent capacitance matrix]\label{lem: spec of capacitance matrix}
 The characteristic polynomial of the $N\times N$ matrix $\mathcal{C}^{\gamma}(k_0)$ in (\ref{eq:mathcalCkgamma_def}) satisfies
\begin{align}\label{equ: C^gamma(k_0) characteristic}
\det\bigl(\mathcal{C}^{\gamma}(k_0)-\lambda I\bigr) = \sum_{l=0}^{m} C_{I,l}^{\mathrm{per}}\,(-\lambda)^{N-l}-\frac{2\cos(\gamma L)}{\prod_{j=1}^{2N}t_j(k_0)},
\end{align}
where $m=m(k_0)$ is defined in (\ref{def:m_periodic}). Consequently, when $n=n(k_0)<2N$, the eigenvalues of $\mathcal{C}^{\gamma}(k_0)$ are $\gamma$-independent. In particular, \(\mathcal{C}^{\gamma}(k_0)\) possesses exactly \(M\) nonzero eigenvalues counted with multiplicity, where
\[
M=\begin{cases}
    m,&n<2N\text{ or }\cos(\gamma L)<1,\\
    N-1,&n=2N\text{ and }\cos(\gamma L)=1.
\end{cases}
\]
Moreover, all eigenvalues of $\mc{C}^{\gamma}(k_0)$ are real and non-negative.
\end{lemma}
\begin{proof}
Let $A_\gamma$ be the incidence matrix acting by
$(A_\gamma q)_j=q_{j+1}-q_j$, where
$q_{N+1}=\e^{\i\gamma L}q_1$. Set
\begin{equation}\label{equ: A_gamma def}
 \mathcal L^{-1}=\operatorname{diag}_{1\leq j\leq N}
             \{t_{2j-1}(k_0)^{-1}\},\qquad
 \mathcal S^{-1}=\operatorname{diag}_{1\leq j\leq N}
             \{t_{2j}(k_0)^{-1}\}.
\end{equation}
These are finite nonnegative diagonal matrices; the superscript
$-1$ denotes the specified weights, even if some weights vanish.
When all intervals are resonant, $\mathcal L$ and $\mathcal S$
are their ordinary inverses. Directly from
\eqref{eq:mathcalCkgamma_def},
\[
 \mathcal C^\gamma=\mathcal L^{-1}A_\gamma^*
                         \mathcal S^{-1}A_\gamma.
\]
Here and below a star or the superscript $H$ denotes conjugate
transpose. Sylvester's determinant identity shows that this
matrix has the same characteristic polynomial as
\[
 H=\mathcal L^{-1/2}A_\gamma^*
                         \mathcal S^{-1}A_\gamma\mathcal L^{-1/2},
\]
which is Hermitian positive semidefinite. Therefore, all its
eigenvalues are real and nonnegative, including when weights vanish.

We compute the elementary symmetric coefficients of $H$ using
Cauchy--Binet. The coefficient of $(-\lambda)^{N-l}$ is
\begin{equation}\label{eq:cb-minors}
 \sum_{\substack{U,V\subset\{1,\ldots,N\}\\|U|=|V|=l}}
 \left(\prod_{j\in U}t_{2j-1}(k_0)^{-1}\right)
 \left(\prod_{j\in V}t_{2j}(k_0)^{-1}\right)
 |\det (A_\gamma)_{V,U}|^2.
\end{equation}
The bipartite graph of nonzero entries of $A_\gamma$ is the cycle
of $2N$ alternating resonator and spacing vertices. For $l<N$,
the subgraph induced by $U,V$ is a union of paths. A path has at
most one perfect matching: an endpoint, if matched, forces its
incident edge, and deletion of its endpoints proves the claim
inductively. Thus, a minor in \eqref{eq:cb-minors} vanishes unless
this induced graph has a perfect matching, in which case its
determinant is a single product of entries of modulus one.
Every size-$l$ matching of the full cycle determines exactly
one pair $U,V$ of its covered vertices, and its weight is
$\prod\theta_j$. Zero weights simply eliminate edges outside
the resonant set. It follows that \eqref{eq:cb-minors} equals
$C_{I,l}^{\mathrm{per}}$ for $l<N$.

For $l=N$, there is one full minor and
$|\det A_\gamma|^2=|1-\e^{\i\gamma L}|^2$.
The two perfect matchings of the full cycle each have weight
$\prod_{j=1}^{2N}t_j(k_0)^{-1}$, or zero if $I\ne\Omega$.
The final coefficient is therefore
\[
 C_{I,N}^{\mathrm{per}}
 -2\cos(\gamma L)\prod_{j=1}^{2N}t_j(k_0)^{-1}.
\]
This proves \eqref{equ: C^gamma(k_0) characteristic}.

If $n<2N$, the product vanishes, and the largest nonzero
matching coefficient is $C_{I,m}^{\mathrm{per}}>0$ (with the
empty matching giving $1$ when $m=0$). The zero eigenvalue has
multiplicity $N-m$. If $n=2N$, the constant term is
$2(1-\cos(\gamma L))/\prod_jt_j$. It is positive unless
$\e^{\i\gamma L}=1$. At that phase the constant term vanishes,
but $C_{\Omega,N-1}^{\mathrm{per}}>0$, giving a simple zero
eigenvalue and $N-1$ positive eigenvalues. All assertions follow.
\end{proof}

Following the notation of~\cite[Section~3.3]{pm1}, we introduce
\begin{equation*}\label{def:matrixrs}
    \nu_i := \frac{2\sigma_i}{1+\sigma_i} = \frac{2\delta}{\delta+r_i},1\leq i\leq   N, \quad
    R := \begin{pmatrix} -1 & 1 \\ -1 & 1 \end{pmatrix}, \quad
    S := \begin{pmatrix} 0 & -1 \\ 1 & 0 \end{pmatrix}.
\end{equation*}
We write $L_j(k) := L(\mathfrak{t}_j k)$ for the matrix associated 
with the $j$\textsuperscript{th} component $\mathfrak{t}_j$ of $\mathfrak{\bm{t}}$ defined in~\eqref{def:vectort}. Define
\begin{equation}\label{equ: G(k;delta) expand}
\begin{aligned}
    G(k;\delta) := L_{2N}(k)(R+\nu_N S)L_{2N-1}(k)(R+\nu_N S)L_{2N-2}(k)(R+\nu_{N-1} S)
    \cdots \\L_3(k)(R+\nu_2S)L_2(k)(R+\nu_1S)L_1(k)(R+\nu_1 S),
\end{aligned}
\end{equation}
and set
\begin{align*}\label{equ: g(k;nu) def}
g(k;\delta) :=\operatorname{tr} G(k;\delta)-\prod_{i=1}^N[\nu_i(2-\nu_i)]\cdot2\cos(\alpha L).
\end{align*}
A direct multiplication gives $\prod_{i=1}^N\frac{4\sigma_i}{(1+\sigma_i)^{2}}J_0 M_{\text{tot}}(k ; \delta) J_0=G(k ; \delta)$ with $J_0=\diag\{-1,1\}$ and thus, $f(k ; \delta)\equiv g(k ; \delta)$. Therefore, $k\in\C$ is a resonance of (\ref{equ: Bloch problem}) if and only if $k$ is a zero of $g(k;\delta)$. We expand
\begin{equation*}\label{equ:G_expession}
    G(k;\delta) = \sum_{l=0}^{\infty} G_l(k)\,\delta^l, \qquad
    g(k;\delta) = \sum_{l=0}^{\infty} g_l(k)\,\delta^l,
\end{equation*}
Then we have
\begin{align*}\label{equ: g_l_tr relation}
g_l(k)=\begin{cases}
    \operatorname{tr}G_l(k),&0\leq l\leq N-1,\\
    \operatorname{tr}G_l(k)-c_l\cdot\cos(\alpha L),&l\geq N,
\end{cases}
\end{align*}
with some non-zero constants $c_l$. In particular, $c_N=2^{2N+1}/{(r_1r_2\cdots r_N)}$.
We next prove directly the Taylor information needed for the
square-root scaling.

\begin{proposition}[Leading cyclic expansion]\label{thm: tr_G_order}
Let $z=k-k_0$ and write
$\operatorname{tr}G(k_0+z;\delta)=\sum_{q,l\geq0}d_{ql}z^q\delta^l$.
Then $d_{ql}=0$ if $q+2l<n$. For $0\leq l\leq m$,
\begin{equation}\label{equ: tr G_l expression}
 \operatorname{tr}G_l(k_0+z)
 =C_1(-2i)^{2N}(-1)^l C_{I,l}^{\mathrm{per}}z^{n-2l}
   +\O(z^{n-2l+1}),
\end{equation}
where
\begin{equation}\label{equ: C_1 def}
 C_1=(-1)^{\sum_{j\in I}m_j}
       \prod_{j\notin I}\sin(\mathfrak t_jk_0)
       \prod_{j\in I}\mathfrak t_j\ne0.
\end{equation}
There are no further coefficients on $q+2l=n$.
Thus, the part of the Newton diagram on this line is determined
exactly, without any assertion about a second edge of the diagram.
\end{proposition}
\begin{proof}
Write $R=uv^\top$ with $u=(1,1)^\top$ and $v=(-1,1)^\top$, and
put $\eta(X)=v^\top Xu$. For matrices separated by copies of $R$,
\[
 \operatorname{tr}(X_sR\cdots X_1R)=\prod_{j=1}^s\eta(X_j).
\]
In particular, $\eta(L_j)=-2\i\sin(\mathfrak t_jk),\eta(L_{j+1}SL_j)=2\cos((\mathfrak t_{j+1}-\mathfrak t_j)k)$. The indices in these identities are cyclic. Recall (\ref{equ: G(k;delta) expand}), expanding each
$R+\nu_iS$ selects a set of interfaces at which $R$ is replaced by
$S$. A term with $s$ selected interfaces has a factor $\O(\delta^s)$.
Every unselected interval, meaning an $L_j$ bordered by two
unreplaced copies of $R$, contributes its sine factor. At most
$2s$ intervals are touched by the selected interfaces; hence at
least $\max(n-2s,0)$ of the resonant sine factors remain.
This proves $q+2l\geq n$ for every Taylor monomial. The term with
all interfaces selected has $l\geq2N$ and satisfies the same
inequality automatically.

Equality $q+2l=n$ requires $s=l$, the linear term of every
$\nu_i$, and $2l$ distinct touched intervals, all in $I$.
Equivalently, the selected interfaces form a matching of the
resonant cyclic graph. For the interface between $L_j$ and
$L_{j+1}$, let $\rho_j=r_i$ if $j=2i-1$, and
$\rho_j=r_{i+1}$ if $j=2i$ (cyclically). Its coefficient is
$2\delta/\rho_j+\O(\delta^2)$. When $j,j+1\in I$, its
replacement of two sine factors changes the leading product by
\[
 \frac{(2\delta/\rho_j)\,2(-1)^{m_j+m_{j+1}}}
 {(-2i)^2(-1)^{m_j+m_{j+1}}
       \mathfrak t_j\mathfrak t_{j+1}z^2}
 =-\frac{\delta}{t_jt_{j+1}z^2}.
\]
Here $\rho_j\mathfrak t_j\mathfrak t_{j+1}=t_jt_{j+1}$
also when adjacent resonators have different speed ratios.
Disjoint replacements multiply. Summing over size-$l$ matchings
therefore gives \eqref{equ: tr G_l expression}. The maximum
matching size is $m$: each proper component is a path with
$n_j$ vertices, and the full cycle has maximum matching size $N$.
This proves the final assertion as well.
\end{proof}
\begin{figure}[htbp]
\centering
\includegraphics[width=.74\textwidth]{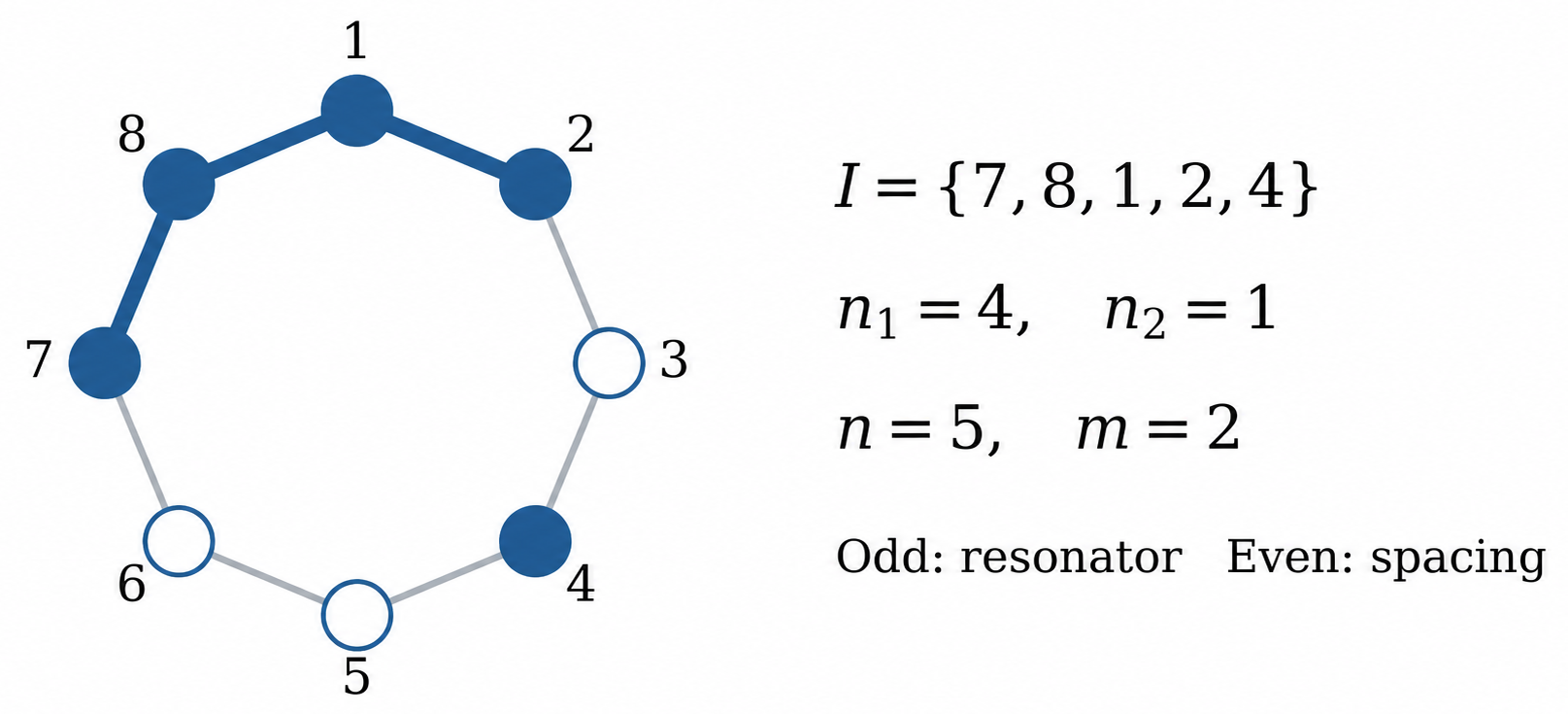}
\caption{The cyclic resonant set for $N=4$ and
$I=\{7,8,1,2,4\}$. The filled vertices form components of lengths
$4$ and $1$, so $n=5$ and $m=2$. Bold edges show one maximum
matching. A selected edge contributes its weight
$\theta_j=1/(t_jt_{j+1})$ to $C_{I,l}^{\mathrm{per}}$.
This is a graph of interval indices, not a spatial picture of the chain.}
\label{fig:matching}
\end{figure}
% \begin{lemma}\label{thm: max_m_n-m}
%  Let \(n(I)\) and \(m(I)\) be defined in (\ref{equ: I def}) and (\ref{def:m_periodic}), respectively. Then
% \[
% \max_{I\subset\Omega} m(I)=N,\qquad \max_{I\subset\Omega} (n(I)-m(I))=N.
% \]
% Moreover, $m(I)=N$ if and only if $I=\Omega$.
% \end{lemma}
% \begin{proof}
% For nonempty \(I\), let \(n_1,\dots,n_p\) be the lengths
% of its cyclic maximal connected components. Then
% \(n=\sum_j n_j\) and \(m=\sum_j\lfloor n_j/2\rfloor\),
% so \(m\le n/2\le N\). Equality \(m=N\) forces
% \(n=2N\), hence \(I=\Omega\); conversely,
% \(I=\Omega\) gives \(m=N\).

% If \(\varnothing\ne I\ne\Omega\), each component is
% followed by at least one index outside \(I\), so
% \(p\le 2N-n\). Consequently,
% \[
%  n-m=\sum_{j=1}^p\left\lceil\frac{n_j}{2}\right\rceil
%  \le\frac{n+p}{2}\le N.
% \]
% For \(I=\varnothing\), \(n-m=0\), whereas
% \(I=\Omega\) gives \(n-m=N\).
% Thus both maxima equal \(N\).
% \end{proof}

We now establish the main resonance theorem for the Bloch eigenvalue problem \eqref{equ: Bloch problem}.
\begin{theorem}\label{thm: k(delta) expand}
Fix a real $\alpha\in Y^*$. Let $\delta\in\C\setminus\{0\}$ tend to zero and fix \(k_{0}\in E\). Let \(I=I(k_0)\),\(n=n(k_{0})\), \(m=m(k_{0})\) be as in (\ref{equ: I def})(\ref{def:m_periodic}). 
Define
\begin{align}\label{equ: gamma def}
    \gamma:=\gamma(\alpha)=\alpha+\frac{k_0}{L}\sum_{j\in I}\mathfrak{t}_j.
\end{align}
According to Lemma \ref{lem: spec of capacitance matrix}, denote by \(\lambda_{1},\dots,\lambda_{M}\) the positive eigenvalues of the periodic frequency-dependent capacitance matrix \(\mathcal{C}^{\gamma}(k_{0})\) defined in (\ref{eq:mathcalCkgamma_def}), counted with multiplicity. Zeros are counted in the $k$-plane;
at $k=0$, multiplicity in $k$ need not equal eigenspace dimension.
Fix a branch $h=\sqrt\delta$ on a slit punctured neighborhood;
the conclusions also specify the roots as an unordered multiset.
Then as \(\delta\to0\), the Bloch problem (\ref{equ: Bloch problem}) admits exactly \(n(k_0)\) resonances \(k(\delta)\) in a neighbourhood of \(k_{0}\):
\begin{itemize}
    \item The first $2M$ branches of resonances have the following asymptotic expression
    \begin{align*}\label{equ:omega_j^pm expand}
    k_j^\pm(\delta)=k_0\pm \sqrt{{\lambda_j}\delta}+\mc{O}(\delta),\quad j=1,2,\cdots,M.
    \end{align*}
    \item The remaining $n-2M$ resonances have the form
    \[
    k_j(\delta)=k_0+\mc{O}(\delta),\quad j=1,2,\cdots,n-2M,
    \]
    Moreover, if $n=2N$ and $\cos(\gamma L)=1$, we have $n-2M=2$ and the two $\mc{O}(\delta)$-resonances are exactly $k_j(\delta)\equiv k_0,j=1,2$, counted with multiplicity.
\end{itemize}
\end{theorem}
\begin{proof}
Using Lemma \ref{thm: f(k;0) form} and $f(k;\delta)\equiv g(k;\delta)$, $g(\cdot;0)$ has a zero of order $n(k_0)$ at $k_0$. When $\delta\to0$, $g(\cdot;\delta)$ converges to $g(\cdot;0)$ uniformly on any compact set and thus, Rouch\'e's theorem implies that there are exactly $n$ zeros near $k_0$ when $|\delta|$ is sufficiently small. Put $z=k-k_0$ and expand
\[
g(k_0+z;\delta)=\sum_{q,l\geq0}c_{ql}z^q\delta^l.
\]
Set $h=\delta^{1/2}$ and define
\[
H(y,h):=h^{-n}g(k_0+hy;h^2)=\sum_{q,l\geq0}c_{ql}h^{q+2l-n}y^q.
\]
Proposition~\ref{thm: tr_G_order} gives $c_{ql}=0$ whenever
$q+2l<n$ for the trace part. The Bloch term starts at order
$\delta^N$, and $n\leq2N$, so the same conclusion holds for $g$.
It contributes to $q+2l=n$ only when $n=2N$.
Consequently, $H$ extends analytically across $h=0$. The coefficient
formula in Proposition~\ref{thm: tr_G_order} gives
\begin{align*}
H(y,0):=\sum_{q+2l=n}c_{ql}y^q&=C_1\sum_{l=0}^m 2^{2l}(-2\i)^{2N-2l}C_{I,l}^{\text{per}}y^{n-2l}-\mathbf{1}_{m=N}\cdot  \frac{2^{2N+1}}{r_1r_2\cdots r_N}\cos(\alpha L)
\\&=\begin{cases}
C_1(-1)^{N+m}2^{2N}y^{n-2m}\sum_{l=0}^mC_{I,l}^{\text{per}}(-y^2)^{m-l},&n<2N,\\
C_12^{2N}\left[\sum_{l=0}^N C_{I,l}^{\text{per}}(-y^2)^{N-l}-\frac{2\cos(\gamma L)}{\prod_{j=1}^{2N}t_j}\right],&n=2N.
\end{cases}    
\end{align*}
Compared with Lemma \ref{lem: spec of capacitance matrix}, all zeros of $H(\cdot,0)$, counted with multiplicity, are given by
\[
\underbrace{0,\ldots,0}_{n-2M},\quad
\pm\sqrt{\lambda_1},\pm\sqrt{\lambda_2},\ldots,\pm\sqrt{\lambda_M},
\]
where $\lambda_1,\lambda_2,\ldots,\lambda_M$ are the $M$ positive eigenvalues of $\mathcal{C}^\gamma(k_0)$, counted with their multiplicities. Let \(c\in\R\) be a root of \(H(\,\cdot\,,0)\) of multiplicity \(d\). Choose a small circle \(\Gamma\) around \(c\) that contains no other root and on which \(H(y,0)\ne0\). For sufficiently small \(|h|\), $H(y,h)\ne0$ on $\Gamma$, and precisely \(d\) roots lie inside it. Denote them by \(y_1(h),\ldots,y_d(h)\), counted with multiplicity. Then $y_j(0)=c,j=1,2,\ldots,d$.

First, take a real $h\ne0$. Then $\delta=h^2>0$. Each $k_0+hy_j(h)$ is a resonance, so Proposition \ref{thm: k real} makes it real. Hence, the $y_j(h)$'s are real. Consider
\[
Q(h)=\frac{1}{2\pi \i}\int_\Gamma(y-c)^2\frac{\partial_y H(y,h)}{H(y,h)}\dd y.
\]
The denominator stays nonzero on the fixed circle \(\Gamma\) for all small complex \(h\). The integrand is therefore analytic in \(h\), uniformly on \(\Gamma\), so \(Q\) is analytic near zero. Summing residues gives
\[
Q(h)=\sum_{j=1}^d\left(y_j(h)-c\right)^2.
\]
For real $h$, the right hand side is non-negative. Since $Q(0)=0$, the restriction of $Q$ to the real axis has a local minimum at zero. Thus, $Q'(0)=0$ and its Taylor expansion gives $Q(h)=\mc{O}(h^2),h\to  0,h\in\R$. Then for $j=1,2,\ldots,d$,
\begin{align}\label{equ: y_j(h) estimate real}
(y_j(h)-c)^2\leq Q(h)=\mc{O}(h^2)\Longrightarrow y_j(h)=c+\mc{O}(h),\quad h\to0,h\in\R.
\end{align}
To extend (\ref{equ: y_j(h) estimate real}) to complex $h$, consider the polynomial
\begin{align}\label{equ: prod w}
\prod_{j=1}^d\left[w-(y_j(h)-c)\right]:=w^d+a_1(h)w^{d-1}+\cdots+a_d(h).
\end{align}
For each $l=1,2,\ldots,d$, the contour integrals
\[
S_l(h)=\frac{1}{2\pi \i}\int_\Gamma(y-c)^l\frac{\partial_y H(y,h)}{H(y,h)}\dd y=\sum_{j=1}^d\left(y_j(h)-c\right)^l
\]
are analytic by the same argument used for $Q$. Newton's identities express $a_l$ recursively as a polynomial in $S_1,\cdots,S_l$, proving that $a_l(h)$ is analytic. For real $h\ne0$, the root--coefficient formulas together with (\ref{equ: y_j(h) estimate real}) give
\[
a_l(h)=(-1)^l\sum_{1\leq j_1<\cdots<j_l\leq d}\prod_{s=1}^l\left(y_{j_s}(h)-c\right)=\mc{O}(h^l).
\]
Since \(a_l\) is analytic for complex \(h\), this real-axis bound forces every Taylor coefficient of \(a_l\) of degree below \(l\) to vanish. Hence,
\[
a_l(h)=h^l b_l(h),
\]
where \(b_l\) is analytic and bounded near zero. For complex $h\ne0$, substituting $w=hv$ into (\ref{equ: prod w}) and dividing $h^d$ gives
\[
v^d+b_1(h)v^{d-1}+\cdots+b_d(h)=0.
\]
Because these coefficients are uniformly bounded, the roots \(v\) are uniformly bounded. Consequently
\begin{align}\label{equ: y_j(h) estimate complex}
y_j(h)=c+\mc{O}(h),\quad h\to0,h\in\C.
\end{align}
Applying (\ref{equ: y_j(h) estimate complex}) first to the roots $c=\pm\sqrt{\lambda_j}$, we obtain
\[
k_j^{\pm}(\delta)=k_0+h(\pm\sqrt{\lambda_j}+\mc{O}(h))=k_0\pm\sqrt{\lambda_j\delta}+\mc{O}(\delta),\quad j=1,2,\cdots,M.
\]
Applying (\ref{equ: y_j(h) estimate complex}) to $c=0$, whose multiplicity is $n-2M$, gives the remaining $n-2M$ resonances with the form  $k_0+\mc{O}(\delta)$.

Now consider the exceptional case $n=2N,\cos(\gamma L)=1$. In this case $n-2M=2$. If $k_0\ne0$, each segment transfer matrix equals its
resonant sign times $I$, and hence, $A(k_0;\delta)=\sigma I$ with $\sigma=(-1)^{\sum_{j=1}^{2N}m_j}=\e^{\i\alpha L}$. Moreover, the trigonometric formula for the trace shows that
$D(k_0+z;\delta)$ is even in $z$: equivalently, for $z\ne0$,
the segment matrices at $k_0+z$ are the corresponding signed
matrices at $z$, conjugated by $\operatorname{diag}(1,(k_0+z)/z)$.
If $k_0=0$, Example~\ref{ex: k=0} instead gives an upper triangular
matrix with trace $2$; it does not give $A(0;\delta)=I$.
Its trace is again even in $k$. In either case
$g(k_0;\delta)=\partial_kg(k_0;\delta)=0$.
To show that the multiplicity of $k_0$ is exactly two, we need to show that $\partial_{kk}g(k_0;\delta)\ne0$. To this end, using the coefficient in Proposition \ref{thm: tr_G_order} with \(n=2N\) and \(l=N-1\) gives
\[
\frac{1}{2}\partial_{kk}g(k_0;\delta)=-2^{2N}C_1C_{\Omega,N-1}^{\text{per}}\delta^{N-1}+\mc{O}(\delta^N),
\]
where $C_1$ defined in (\ref{equ: C_1 def}) is nonzero, and
\[
C_{\Omega,N-1}^{\text{per}}=\frac{\left(\sum_{j=1}^N t_{2j-1}\right)\left(\sum_{j=1}^N t_{2j}\right)}{\prod_{j=1}^{2N}t_j}>0,
\]
which proves that $\partial_{kk}g(k_0;\delta)\ne0$ for sufficiently small $|\delta|$.
\end{proof}
\begin{remark}\label{rem: N=1}
For \(N=1\), the discriminant is explicitly $D(k;\delta)=\Phi(k,\delta)/\delta$ with $\Phi(k,\delta)$ given in (\ref{Phi(k,delta) def}).
If \(k_0\ne0\) satisfies
\(r\ell k_0=m_1\pi\) and \(sk_0=m_2\pi\),
Taylor expansion of the dispersion equation, after
multiplication by \(\delta\), gives
\[
 (k-k_0)^2
 =\frac{2(1-\cos(\gamma L))}{r^2\ell s}\delta
 +\mc O\bigl(|k-k_0|^3
       +\delta|k-k_0|+\delta^2\bigr),
\]
Thus, the periodic capacitance matrix reduces to the
scalar \(2(1-\cos(\gamma L))/(r^2\ell s)\).
If only one interval is resonant at \(k_0\), the
dispersion equation multiplied by \(\delta\) has
a simple zero at \((k,\delta)=(k_0,0)\);
the implicit function theorem then gives
\(k-k_0=\mc O(\delta)\).
At \(k_0=0\), the scalar  capacitance gives the usual acoustic square-root scale, while at $\alpha=0$ the exact zero mode is constant.
\end{remark}
\begin{remark}[Fixed parameters and multiplicities]
The error bounds in Theorem~\ref{thm: k(delta) expand} are for fixed
$k_0$, geometry, and $\alpha$. When $n<2N$, the positive leading
coefficients do not depend on $\alpha$; this does not assert that
the exact bands are flat, since their higher-order terms can depend
on $\alpha$. For $n=2N$ and the exceptional phase, the two zeros
at $k_0=0$ are a double zero in $k$, with only a one-dimensional
eigenspace for positive contrast. At $k_0\ne0$ the eigenspace is
two-dimensional. Negative $k$ gives the opposite sign of frequency;
the physical positive-frequency bands are obtained by restricting
to $k\geq0$.
\end{remark}

\begin{corollary}[Spectral exclusion away from the limiting set]
\label{cor:gaps}
Let $K\subset\mathbb R$ be compact and disjoint from $E$. There
exists $\delta_K>0$ such that, for all $0<\delta<\delta_K$ and
all $\alpha\in Y^*$, no Bloch wavenumber belongs to $K$.
In particular, every closed subinterval between consecutive
positive points of $E$ that stays away from its endpoints lies
in a spectral gap for sufficiently small contrast.
\end{corollary}
\begin{proof}
The continuous function $|f(k;0)|$ has a strictly positive minimum
on $K$. The convergence in Lemma~\ref{thm: f(k;0) form} is uniform
also in real $\alpha$: the only $\alpha$-dependent term is a
bounded cosine multiplied by $\O(\delta^N)$. Hence, $f(k;\delta)$
does not vanish on $K$, uniformly over $Y^*$, for sufficiently
small $\delta>0$.
\end{proof}

\begin{figure}[htbp]
\centering
\includegraphics[width=\textwidth]{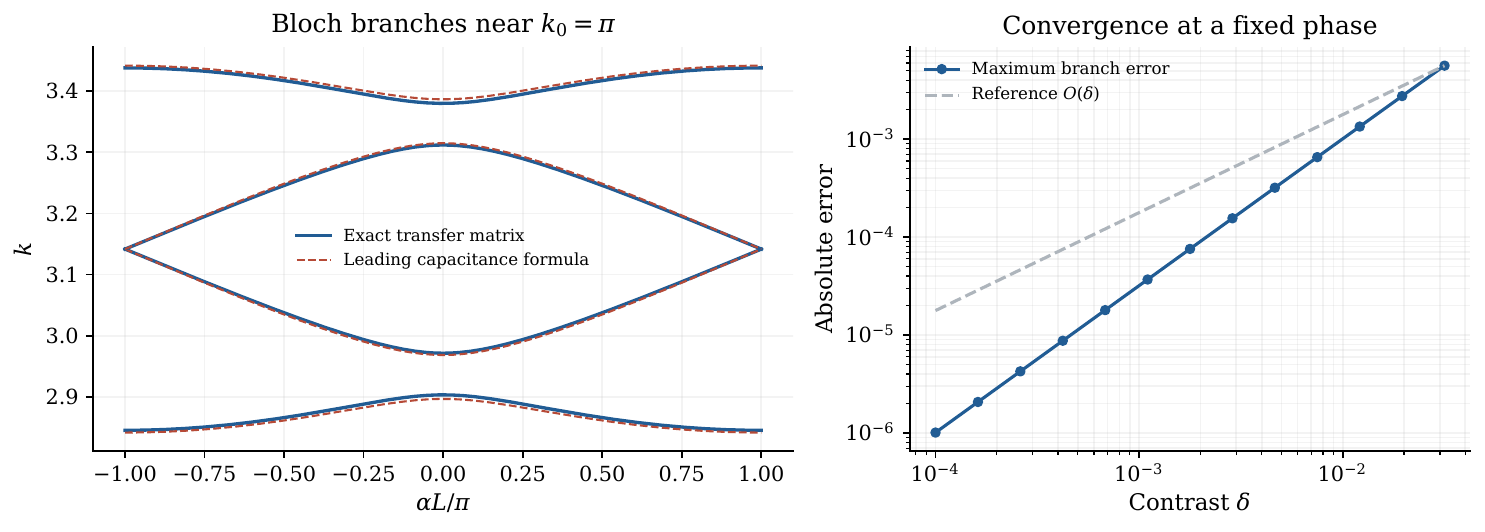}
\caption{Left: exact Bloch branches near $k_0=\pi$ (solid) and
the capacitance approximation $k_0\pm\sqrt{\delta\lambda_j}$
(dashed), for $N=2$, $r_1=r_2=1$, $\ell_1=\ell_2=1$,
$s_1=1$, $s_2=2$, and $\delta=0.03$. Here $L=5$ and
$\gamma L=\alpha L+5\pi$. Right: the maximum error over the four
branches at $\alpha L=0.37\pi$, compared with a reference of
order $\delta$. The observed faster decay in this fully resonant
example is compatible with the general $\O(\delta)$ bound; it
does not establish a sharper bound for arbitrary configurations.}
\label{fig:bands}
\end{figure}

\subsection{Eigenmodes}\label{subsec: eigenmodes}
We now study the eigenmode asymptotics for
\(k_0\in E\setminus\{0\}\) with \(n(k_0)=2N\).
This extends the full-resonance condition \(n(0)=2N\)
beyond the subwavelength regime, although the
zero-frequency eigenmode asymptotics differ substantially
from those at nonzero \(k_0\).
The former were established in
\cite[Lemma 4.10]{ammari2023edge}.
For completeness and their application to the Zak phase
in Section~\ref{sec: Zak phase}, we recall them for \(N=2\)
and simple eigenvalues of \(\mc C^\alpha(0)\) in
Proposition~\ref{prop:lowmodes}; here
\eqref{equ: gamma def} gives \(\gamma=\alpha\).

When \(n(k_0)=2N\), all resonators and spacings belong
to \(I(k_0)\), and the leading-order matching conditions
reduce to a finite-dimensional eigenvalue problem
governed by \(\mc C^\gamma(k_0)\).
When \(n(k_0)<2N\), the resonant set consists of cyclic
connected components separated by nonresonant intervals.
The different asymptotic orders of their propagation
matrices lead to more complicated matching conditions.
The case \(N=2\) is treated in
Appendix~\ref{app:asymptotics N2}, while the general
analysis is left to future work. Recall the notation in (\ref{equ: A_gamma def}), when $n=2N$, the matrices $\mc{L},\mc{S}$ are invertible and thus we have
\begin{align}\label{equ: C^gamma decomposition}
\mc{C}^\gamma(k_0)=\mc{L}^{-1}A_\gamma^H\mc{S}^{-1}A_\gamma.
\end{align}
Let the associated  matrix $\mc{D}^\gamma(k_0)$ be defined by
\begin{align}\label{equ: D^gamma decomposition}
\mc{D}^\gamma(k_0)=\mc{S}^{-1}A_\gamma \mc{L}^{-1}A_\gamma^H,
\end{align}
where $A_\gamma$ acts by
$(A_\gamma q)_j=q_{j+1}-q_j$, with $q_{N+1}=\e^{\i\gamma L}q_1$.
In particular, the definitions remain unambiguous for $N=2$.
The following lemma records their spectral properties.
\begin{lemma}
When $n=2N$, the following results about the spectra of $\mc{C}^\gamma(k_0)$ and $\mc{D}^\gamma(k_0)$ hold. 
\begin{itemize}
    \item The eigenvalues of \(\mathcal C^\gamma(k_0)\) and \(\mathcal D^\gamma(k_0)\) are non-negative and coincide, counted with multiplicity. Furthermore, both matrices are diagonalizable, so the algebraic and geometric multiplicities of every eigenvalue agree.
    \item Every eigenvalue of \(\mathcal C^\gamma(k_0)\) has multiplicity at most two. Moreover, a double eigenvalue can occur only if $
    \cos(\gamma L)\in\{1,-1\}$.
\end{itemize}
\end{lemma}
\begin{proof}
Set \(B=\mc S^{-1/2}A_\gamma\mc L^{-1/2}\).
By \eqref{equ: C^gamma decomposition} and
\eqref{equ: D^gamma decomposition},
\[
\mc L^{1/2}\mc C^\gamma(k_0)\mc L^{-1/2}=B^HB,\quad \mc S^{1/2}\mc D^\gamma(k_0)\mc S^{-1/2}=BB^H.
\]
These Hermitian positive semidefinite matrices have
the same eigenvalues, counted with multiplicity.
Thus both original matrices are diagonalizable
and have identical nonnegative spectra. For an eigenpair \((\lambda,\bm v)\) of
\(\mc C^\gamma(k_0)\), the eigenvalue equation becomes
\begin{equation*}\label{equ: eigen recurrence equations}
-\theta_{2i-2}v_{i-1}
+(\theta_{2i-2}+\theta_{2i-1})v_i
-\theta_{2i-1}v_{i+1}
=\lambda v_i,\qquad 1\le i\le N,
\end{equation*}
with \(v_0=\e^{-\i\gamma L}v_N\) and
\(v_{N+1}=\e^{\i\gamma L}v_1\).
Since \(\theta_{2i-1}>0\), the recurrence is uniquely
determined by \((v_0,v_1)\).
Hence each eigenspace has dimension at most two,
and diagonalizability gives the same bound on
algebraic multiplicity. Let \(M_\lambda=T_N(\lambda)\cdots T_1(\lambda)\),
where
\[
T_i(\lambda)=
\begin{pmatrix}
0&1\\
-\theta_{2i-2}/\theta_{2i-1}
&(\theta_{2i-2}+\theta_{2i-1}-\lambda)/\theta_{2i-1}
\end{pmatrix}.
\]
By \eqref{eq:coupling},
\(\det M_\lambda
=\prod_{i=1}^N\theta_{2i-2}/\theta_{2i-1}=1\).
The boundary conditions are equivalent to
\(M_\lambda(v_0,v_1)^\top
=\e^{\i\gamma L}(v_0,v_1)^\top\).
If the eigenspace is two-dimensional, its initial
vectors span \(\mathbb C^2\), so
\(M_\lambda=\e^{\i\gamma L}I\).
Taking determinants gives \(\e^{2\i\gamma L}=1\),
and therefore \(\cos(\gamma L)\in\{1,-1\}\).
\end{proof}
To describe the eigenmodes beyond the subwavelength regime, it is convenient to introduce the rescaled coordinate
\begin{align}\label{equ: Y(x) def}
Y(x):=\int_{x_1^-}^{x}\left(\sum_{i=1}^Nr_i\mathbf 1_{(x_i^-,x_i^+)}(s)+\mathbf 1_{\mathbb R\setminus\overline D}(s)\right)\,\dd s,
\qquad x\in[x_1^-,x_{N+1}^-].
\end{align}
Then \(Y\) is continuous and piecewise linear. Moreover,
\[
Y(x_{i}^+)-Y(x_i^-)=r_i\ell_i,\quad Y(x_{i+1}^-)-Y(x_i^+)=s_i,\qquad 1\le i\le N.
\]
In particular, \(Y\) rescales the physical length \(\ell_i\) of each resonator to \(r_i\ell_i\), while it leaves each spacing \(s_i\) unchanged. We will use \(Y\) to express the phase of the eigenmodes.
Now we can give the result on the eigenmodes.
\begin{theorem}[Asymptotic behavior of the eigenmodes when \(n=2N\)]\label{thm: eigenmode n=2N}
Assume that \(k_0\ne 0,n(k_0)=2N\), and let \(\lambda\) be an eigenvalue of \(\mathcal C^\gamma(k_0)\). By Theorem \ref{thm: k(delta) expand}, let $k(\delta)=k_0+c\delta^{1/2}+\mc{O}(\delta)$ with $c\in\left\{\sqrt{\lambda},-\sqrt{\lambda}\right\}$ be a resonance of (\ref{equ: Bloch problem}) and let $u_\delta(x)$ be the associated eigenmode.
\begin{itemize}
\item Suppose that \(\lambda>0\). For all $\delta$ with sufficiently small $|\delta|$, there exist eigenvectors
\[
\bm a_\delta=(a_{1,\delta},\ldots,a_{N,\delta})^\top,
\qquad
\bm b_\delta=(b_{1,\delta},\ldots,b_{N,\delta})^\top,
\]
of \(\mathcal C^\gamma(k_0)\) and \(\mathcal D^\gamma(k_0)\), respectively, with bounded norms and satisfying
\[
A_\gamma\bm a_\delta=c\mc{S}\bm b_\delta,\quad A_\gamma^H\bm b_\delta=c\mc{L}\bm a_\delta.
\]
After a suitable normalization, $u_\delta$ satisfies
\[
u_\delta(x)=\begin{cases}
b_{i,\delta}\sin\bigl(k_0Y(x)\bigr)
+
\mc{O}(\delta^{1/2}),&x\in[x_i^+,x_{i+1}^-],\\
\delta^{1/2}\,a_{i,\delta}\cos\bigl(k_0Y(x)\bigr)
+
\mc{O}(\delta),&x\in[x_i^-,x_i^+],
\end{cases}\qquad i=1,2,\ldots,N.
\]
Moreover, if \(\lambda\) is simple, which holds whenever $\cos(\gamma L)\ne\pm1$, one can take $\bm a_\delta\equiv\bm a,\bm b_\delta\equiv \bm b$, where $\bm a$ and $\bm b$ are fixed eigenvectors of $\mc{C}^\gamma(k_0)$ and $\mc{D}^\gamma(k_0)$, respectively.
\item Suppose that \(\lambda=0\), which occurs only when $\cos(\gamma L)=1$. In this case, by \Cref{thm: k(delta) expand}, the two corresponding resonances are exactly given by $k(\delta)\equiv k_0.$ Then for every $(A,B)\ne(0,0)$, $u_\delta$ can be chosen such that
\[
u_\delta(x)=\begin{cases}
    B\sin\big(k_0Y(x)\big)+\mc{O}(\delta^{1/2}),&x\in[x_i^+,x_{i+1}^-],\\
    \delta^{1/2}A\cos\bigl(k_0Y(x)\bigr)+\mc{O}(\delta),&x\in[x_i^-,x_i^+],
\end{cases}\qquad i=1,2,\ldots,N.
\]
\end{itemize}
All the above estimates are uniform on the indicated intervals.
\end{theorem}
\begin{proof}
Since $k(\delta)=k_0+c\delta^{1/2}+\mc{O}(\delta)$, (\ref{equ:defmj}) gives
\begin{align}\label{equ: cos sin expand}
\cos(k(\delta)\mathfrak{t}_j)=(-1)^{m_j}+\mc{O}(\delta),\sin(k(\delta)\mathfrak{t}_j)=(-1)^{m_j}c\delta^{1/2}\mathfrak{t}_j+\mc{O}(\delta),\quad 1\leq j\leq 2N.
\end{align}
All estimates below are uniform because $N$ is fixed and finite.

\noindent\textbf{Step 1: Interface matching.} Recall matrices $T_j^{\text{Res}}$ and $T_j^{\text{Sp}}$ in (\ref{equ: T_j Res Sp def}) and the definition  of $\bm t$ in (\ref{def:vector bm t}), for $1\leq j\leq N$, (\ref{equ: cos sin expand}) gives
\begin{align}\label{equ: T_j^res expand}
T_j^{\text{Res}}=(-1)^{m_{2j-1}}\begin{pmatrix}
    1+\mc{O}(\delta)&\mc{O}(\delta)\\
    -ck_0t_{2j-1}+\mc{O}(\delta^{1/2})&1+\mc{O}(\delta)
\end{pmatrix},
\end{align}
\begin{align}\label{equ: T_j^sp  expand}
T_j^{\text{Sp}}=(-1)^{m_{2j}}\begin{pmatrix}
    1+\mc{O}(\delta)&ct_{2j}/k_0+\mc{O}(\delta^{1/2})\\
    \O(\delta)&1+\mc{O}(\delta)
\end{pmatrix}.
\end{align}
Recall the interface vectors $V_j^\pm$ given in (\ref{equ: V_j^pm def}), we introduce the sign-adjusted interface vectors
\begin{equation}\label{equ:  W_i^pm def}
\begin{aligned}
W_i^-:=(-1)^{\sum_{j=1}^{2i-2}m_j}\delta^{-1/2}V_i^-=\begin{pmatrix}
    a_i^-(\delta)&k_0b_i^-(\delta)
\end{pmatrix}^\top,\\
W_i^+:=(-1)^{\sum_{j=1}^{2i-1}m_j}\delta^{-1/2}V_i^+=\begin{pmatrix}
    a_i^+(\delta)&k_0b_i^+(\delta)
\end{pmatrix}^\top.
\end{aligned}
\end{equation}
Normalize \(u_\delta\) so that  $\max_{1\leq i\leq N}\left\{||W_i^-||,||W_i^+||\right\}=1$. Since $k_0\ne0$, all $a_i^\pm,b_i^\pm$ are bounded. (\ref{equ: T_j^res expand})(\ref{equ: T_j^sp  expand}) together with the propagation relation (\ref{equ: V_j^pm propagation}) then yield
\begin{equation}\label{equ: ab pm relation}
\begin{aligned}
a_i^+(\delta)-a_i^-(\delta)=\mc{O}(\delta),\quad b_i^-(\delta)-b_i^+(\delta)=ct_{2i-1}a_i^-(\delta)+\mc{O}(\delta^{1/2}).\\
a_{i+1}^-(\delta)-a_i^+(\delta)=ct_{2i}b_i^+(\delta)+\mc{O}(\delta^{1/2}),\quad b_{i+1}^-(\delta)-b_i^+(\delta)=\mc{O}(\delta).
\end{aligned}
\end{equation}
By \((\ref{equ: gamma def})\), the Bloch condition \(V_{N+1}^-=\mathrm e^{\i\alpha L}V_1^-\) becomes \(W_{N+1}^-=\mathrm e^{\i\gamma L}W_1^-\). Thus, with the conventions
\[
a_{N+1}^-(\delta)=\mathrm e^{\i\gamma L}a_1^-(\delta),
\qquad b_0^+(\delta)=\mathrm e^{-\i\gamma L}b_N^+(\delta),
\]
equations (\ref{equ: ab pm relation}) imply, for \(1\le i\le N\),
\[
a_{i+1}^-(\delta)-a_i^-(\delta)=ct_{2i}b_i^+(\delta)+\mc{O}(\delta^{1/2}),\quad
b_{i-1}^+(\delta)-b_i^+(\delta)=ct_{2i-1}a_i^-(\delta)+\mc{O}(\delta^{1/2}).
\]
Indeed, the second relation at \(i=1\) follows from
\(b_1^-(\delta)=\mathrm e^{-\i\gamma L}b_N^+(\delta)+\mc{O}(\delta)\). Set
\[
\bm a(\delta):=(a_1^-(\delta),a_2^-(\delta),\cdots,a_N^-(\delta))^\top,\quad
\bm b(\delta):=(b_1^+(\delta),b_2^+(\delta),\cdots,b_N^+(\delta))^\top.
\]
Recalling the matrices $A_\gamma,\mc{L},\mc{S}$ in (\ref{equ: A_gamma def}), the preceding relations give
\begin{align}\label{equ: coupling system asymptotic}
A_\gamma\bm a(\delta)=c\mc{S}\bm b(\delta)+\mc{O}(\delta^{1/2}),\quad A_\gamma^H\bm b(\delta)=c\mc{L}\bm a(\delta)+\mc{O}(\delta^{1/2}).
\end{align}
Moreover, (\ref{equ: ab pm relation}) and the Bloch condition express every \(a_i^+(\delta)\) and \(b_i^-(\delta)\) in terms of the components of \(\bm a(\delta),\bm b(\delta)\), up to \(\mc{O}(\delta)\). Thus, the normalization implies
\[
1\le C_1\bigl(\|\bm a(\delta)\|+\|\bm b(\delta)\|\bigr)+C_2|\delta|,
\]
while also giving a uniform upper bound on all components. Hence, for sufficiently small \(|\delta|\),
\begin{align}\label{equ: a(delta) b(delta) bound}
0<c_0\le \|\bm a(\delta)\|+\|\bm b(\delta)\|\le C,
\end{align}
where \(c_0,C>0\) are independent of \(\delta\).

\medskip
\noindent\textbf{Step 2: The case $\lambda>0$.} Since $c^2=\lambda>0$, we have $c\ne0$ and the first equation in (\ref{equ: coupling system asymptotic}) gives
\begin{align}\label{equ: b(delta) a(delta) relation}
\bm b(\delta)=\frac{1}{c}\mc{S}^{-1}A_\gamma \bm a(\delta)+\mc{O}(\delta^{1/2}).
\end{align}
Substituting into the second equation in (\ref{equ: coupling system asymptotic}) and using
\((\ref{equ: C^gamma decomposition})\), we obtain
\begin{align}\label{equ: C^gamma-lambda I estimate}
\left(\mc{C}^\gamma(k_0)-\lambda I\right)\bm a(\delta)=\mc{O}(\delta^{1/2}).
\end{align}
Let $\Gamma_\lambda$ be a positively oriented circle enclosing $\lambda$ but no other eigenvalues of $\mc{C}^\gamma(k_0)$, and define
\[
P_\lambda:=\frac{1}{2\pi\i}\int_{\Gamma_\lambda}\left(zI-\mc{C}^\gamma(k_0)\right)^{-1}\dd z.
\]
Since $\mc{C}^\gamma(k_0)$ is diagonalizable, $P_\lambda$ is the projection onto $E_\lambda:=\operatorname{Ker}\left(\mc{C}^\gamma(k_0)-\lambda I\right)$ along the direct sum of the eigenspaces associated with the remaining eigenvalues. In particular, $P_\lambda$ and $\mc{C}^\gamma(k_0)$ commute and the restriction $\left(\mc{C}^\gamma(k_0)-\lambda I\right)|_{\operatorname{Ran}(I-P_\lambda)}$ has a bounded inverse. Then (\ref{equ: C^gamma-lambda I estimate}) gives
\[
\left(\mc{C}^\gamma(k_0)-\lambda I\right)(I-P_\lambda)\bm a(\delta)=(I-P_\lambda)\left(\mc{C}^\gamma(k_0)-\lambda I\right)\bm a(\delta)=\mc{O}(\delta^{1/2}).
\]
Applying the inverse of $\left(\mc{C}^\gamma(k_0)-\lambda I\right)|_{\operatorname{Ran}(I-P_\lambda)}$ gives $(I-P_\lambda)\bm a(\delta)=\mc{O}(\delta^{1/2})$. Setting $\bm a_\delta:=P_\lambda\bm a(\delta)$, we obtain
\begin{align*}\label{equ: a(delta) a_delta diff}
\mc{C}^\gamma(k_0)\bm a_\delta=\lambda\bm a_\delta,\quad\bm a(\delta)=\bm a_\delta+\mc{O}(\delta^{1/2}).
\end{align*}
Moreover, $\bm a_\delta\ne0$ for sufficiently small $|\delta|$. Indeed, if $\bm a_\delta=0$, then $\bm a(\delta)=\mc{O}(\delta^{1/2})$ and thus (\ref{equ: b(delta) a(delta) relation}) gives $\bm b(\delta)=\mc{O}(\delta^{1/2})$, contradicting (\ref{equ: a(delta) b(delta) bound}). Define $\bm b_\delta:=c^{-1}\mc{S}^{-1}A_\gamma\bm a_\delta$. Then \((\ref{equ: b(delta) a(delta) relation})\) gives 
\[
\bm b(\delta)=\bm b_\delta+\O(\delta^{1/2}).
\]
Moreover, by \((\ref{equ: C^gamma decomposition})\) and \(c^2=\lambda\),
\[
A_\gamma\bm a_\delta=c\mathcal S\bm b_\delta,\qquad
A_\gamma^H\bm b_\delta
=\frac1c\mathcal L\mathcal C^\gamma(k_0)\bm a_\delta
=c\mathcal L\bm a_\delta.
\]
Consequently, \((\ref{equ: D^gamma decomposition})\) yields
\[
\mathcal D^\gamma(k_0)\bm b_\delta
=\mathcal S^{-1}A_\gamma\mathcal L^{-1}A_\gamma^H\bm b_\delta
=c^2\bm b_\delta
=\lambda\bm b_\delta.
\]
The estimates for \(\bm a(\delta)-\bm a_\delta\) and \(\bm b(\delta)-\bm b_\delta\), together with (\ref{equ: a(delta) b(delta) bound}) imply that \(\|\bm a_\delta\|+\|\bm b_\delta\|\) is bounded above and away from zero. Rescaling \(u_\delta,\bm a_\delta,\bm b_\delta\) by the same factor, bounded above and away from zero, we may therefore assume $\|\bm a_\delta\|+\|\bm b_\delta\|=1$ without changing the error orders.

If \(\lambda\) is simple, then $E_\lambda=\operatorname{span}\{\bm a\}$
for a fixed eigenvector \(\bm a\neq0\). Hence, $\bm a_\delta=\eta(\delta)\bm a$ for some \(\eta(\delta)\) bounded above and bounded away from zero. Then
\[
\bm b_\delta
=\frac{1}{c}\mc{S}^{-1}A_\gamma \bm a_\delta=
\eta(\delta)\bm b,
\qquad\text{where }
\bm b:=
\frac1c\mathcal S^{-1}A_\gamma\bm a.
\]
Replacing \(u_\delta\) by \(\eta(\delta)^{-1}u_\delta\), we can take $\bm a_\delta\equiv\bm a,\bm b_\delta\equiv\bm b$ without changing the error orders.

\medskip
\noindent\textbf{Step 3: Reconstruction of the eigenmode.} 
Let \(a_{i,\delta}\) and \(b_{i,\delta}\) denote the components of \(\bm a_\delta\) and \(\bm b_\delta\). Recall (\ref{equ:  W_i^pm def}), we can write $V_i^\pm$ as
\[
V_i^-=(-1)^{\sum_{j=1}^{2i-2}m_j}\delta^{1/2}\begin{pmatrix}
    a_i^-(\delta)\\k_0 b_i^-(\delta)
\end{pmatrix},\quad
V_i^+=(-1)^{\sum_{j=1}^{2i-1}m_j}\delta^{1/2}\begin{pmatrix}
    a_i^+(\delta)\\k_0 b_i^+(\delta)
\end{pmatrix}.
\]
On the $i$\textsuperscript{th} resonator,  recalling (\ref{equ: V(x) propagation}) and $k(\delta)=k_0+\mc{O}(\delta^{1/2})$, propagation from $V_i^-$ gives
\begin{align*}
u_\delta(x)&=(-1)^{\sum_{j=1}^{2i-2}m_j}\delta^{1/2}\begin{pmatrix}
\cos(r_ik(\delta)(x-x_i^-))&\frac{\delta^{1/2}}{r_ik(\delta)}\sin(r_ik(\delta)(x-x_i^-))    
\end{pmatrix}\begin{pmatrix}
a_i^-(\delta)\\k_0b_i^-(\delta)    
\end{pmatrix}\\
&=(-1)^{\sum_{j=1}^{2i-2}m_j}\delta^{1/2}a_i^-(\delta)\cos(r_ik(\delta)(x-x_i^-))+\mc{O}(\delta)\\
&=(-1)^{\sum_{j=1}^{2i-2}m_j}\delta^{1/2}a_{i,\delta}\cos(r_ik_0(x-x_i^-))+\mc{O}(\delta).
\end{align*}
Using the rescaled coordinate $Y(x)$ defined in (\ref{equ: Y(x) def}), we finally obtain
\[
u_\delta(x)=\delta^{1/2}a_{i,\delta}\cos(k_0Y(x))+\mc{O}(\delta),\quad x\in[x_i^-,x_i^+],1\leq i\leq N.
\]
Similarly, on the $i$\textsuperscript{th} spacing,
\begin{align*}
u_\delta(x)&=(-1)^{\sum_{j=1}^{2i-1}m_j}\delta^{1/2}\begin{pmatrix}
\cos(k(\delta)(x-x_i^+))&\frac{1}{k(\delta)\delta^{1/2}}\sin(k(\delta)(x-x_i^+))    
\end{pmatrix}\begin{pmatrix}
a_i^+(\delta)\\k_0b_i^+(\delta)    
\end{pmatrix}\\
&=(-1)^{\sum_{j=1}^{2i-1}m_j}b_i^+(\delta)\sin(k(\delta)(x-x_i^+))+\mc{O}(\delta^{1/2})\\
&=(-1)^{\sum_{j=1}^{2i-1}m_j}b_{i,\delta}\sin(k_0(x-x_i^+))+\mc{O}(\delta^{1/2}).
\end{align*}
Using the rescaled coordinate $Y(x)$ again, we obtain
\[
u_\delta(x)=b_{i,\delta}\sin(k_0Y(x))+\mc{O}(\delta^{1/2}),\quad x\in[x_i^+,x_{i+1}^-],1\leq i\leq N.
\]

\medskip
\noindent\textbf{Step 4: The case $\lambda=0$.} By Lemma \ref{lem: spec of capacitance matrix}, $\lambda=0$ occurs if and only if $\cos(\gamma L)=1$, equivalently, $\e^{\i\gamma L}=1$. By Theorem \ref{thm: k(delta) expand}, the two corresponding resonances are exactly given by $k(\delta)\equiv k_0$. At $k=k_0$, the $T_i^{\text{Res}}$ and $T_i^\text{Sp}$ in (\ref{equ: T_j Res Sp def}) reduce to
\begin{align}\label{equ: T_i Res Sp I}
T_{i}^\text{Res}=(-1)^{m_{2i-1}}I,\quad T_i^{\text{Sp}}=(-1)^{m_{2i}}I.
\end{align}
Set $\sigma=(-1)^{\sum_{j=1}^{2N}m_j}$, then the $A(k_0;\delta)$ in (\ref{equ: A(k;delta) def}) has the form $A(k_0;\delta)=\sigma I$. Since $\e^{\i\gamma L}=1$, (\ref{equ: gamma def}) gives $\e^{\i\alpha L}=\sigma$. Consequently, for any initial vector \(V_1^-\in\mathbb C^2\setminus\{0\}\), the corresponding solution satisfies
$V_{N+1}^-
=
A(k_0;\delta)V_1^-
=
\sigma V_1^-
=
\e^{\mathrm i\alpha L}V_1^-$.
Thus, it satisfies the Bloch condition exactly and therefore corresponds to an eigenmode associated with \(k_0\). To obtain the asymptotic expression stated in the theorem, for any fixed $(A,B)\ne(0,0)$, choose the initial data of $u_\delta$ at $x_1^-$ as $V_1^-=\delta^{1/2}(A,k_0B)^\top$. Then (\ref{equ: T_i Res Sp I}) gives
\[
V_i^-=\delta^{1/2}(-1)^{\sum_{j=1}^{2i-2}m_j}\begin{pmatrix}
    A\\k_0B
\end{pmatrix},\quad
V_i^+=\delta^{1/2}(-1)^{\sum_{j=1}^{2i-1}m_j}\begin{pmatrix}
    A\\k_0B
\end{pmatrix}
\]
Similar to the procedures in Step 3, on the $i$\textsuperscript{th} resonator,
\begin{align*}
u_\delta(x)&=(-1)^{\sum_{j=1}^{2i-2}m_j}\delta^{1/2}\begin{pmatrix}
\cos(r_ik_0(x-x_i^-))&\frac{\delta^{1/2}}{r_ik_0}\sin(r_ik_0(x-x_i^-))    
\end{pmatrix}\begin{pmatrix}
A\\k_0B    
\end{pmatrix}\\
&=(-1)^{\sum_{j=1}^{2i-2}m_j}\delta^{1/2}A\cos(r_ik_0(x-x_i^-))+\mc{O}(\delta)=\delta^{1/2}A\cos(k_0Y(x))+\mc{O}(\delta).
\end{align*}
On the $i$\textsuperscript{th} spacing,
\begin{align*}
u_\delta(x)&=(-1)^{\sum_{j=1}^{2i-1}m_j}\delta^{1/2}\begin{pmatrix}
\cos(k_0(x-x_i^+))&\frac{1}{k_0\delta^{1/2}}\sin(k_0(x-x_i^+))    
\end{pmatrix}\begin{pmatrix}
A\\k_0B    
\end{pmatrix}\\
&=(-1)^{\sum_{j=1}^{2i-1}m_j}B\sin(k_0(x-x_i^+))+\mc{O}(\delta^{1/2})=B\sin(k_0Y(x))+\mc{O}(\delta^{1/2}).
\end{align*}
This completes the proof.
\end{proof}
% Theorem~\ref{thm: eigenmode n=2N} concerns nonzero reference
% wavenumbers \(k_0\) satisfying \(n(k_0)=2N\).
% Although \(n(0)=2N\) as well, the eigenmode asymptotics
% at \(k_0=0\) differ substantially from those at nonzero \(k_0\).
% For completeness and to facilitate the discussion of the
% Zak phase in Section~\ref{sec: Zak phase}, we recall the
% zero-frequency eigenmode asymptotics established in
% \cite[Lemma 4.10]{ammari2023edge}.
% For simplicity, we restrict to \(N=2\) and simple eigenvalues
% \(\lambda\) of \(\mc C^\alpha(0)\).
% In this case, \eqref{equ: gamma def} gives \(\gamma=\alpha\).
% The full statement is given in Proposition~\ref{prop:lowmodes}.

\section{Exact Dirac points}\label{sec:dirac}
We first characterize exact Dirac points for arbitrary \(N\), and then derive a complete local classification for \(N=2\). We now define the notion of a Dirac point. The Brillouin zone \(Y^*=(-\pi/L,\pi/L]\) is identified with the circle \(\mathbb R/(2\pi/L)\mathbb Z\), so that the endpoints \(\pm\pi/L\) represent the same Bloch phase, and local expansions near them are understood in a local coordinate modulo \(2\pi/L\). We call the space consisting of the eigenmodes of (\ref{equ: Bloch problem}) the eigenspace. By Proposition \ref{thm: k real}, the resonances \(k\) can be restricted to be real whenever $\delta>0$. This leads to the following definition.

\begin{definition}
Fix \(\delta>0\). A point \((k_*,\alpha_*)\in\mathbb R\times Y^*\) is called a \textit{Dirac point} if the eigenspace of (\ref{equ: Bloch problem}) at \((k_*,\alpha_*)\) has dimension two, and the two branches of resonances crossing at \((k_*,\alpha_*)\) satisfy
\begin{align}\label{equ: linear crossing}
k_\pm(\alpha)
=
k_*\pm v_*(\alpha-\alpha_*)
+
\mathcal O\bigl((\alpha-\alpha_*)^2\bigr),
\qquad\text{for some }v_*>0.
\end{align}
We call a \(k_*\in\mathbb R\) a \textit{Dirac wavenumber} if there exists some \(\alpha_*\in Y^*\) such that \((k_*,\alpha_*)\) is a Dirac point.
\end{definition}
This definition concerns the exact problem (\ref{equ: Bloch problem}) at fixed \(\delta\). The coincidence of the leading terms in the asymptotic expansions of two resonances \(k(\delta)\) as \(\delta\to0\) does not, by itself, imply the existence of an exact Dirac point. As for an exact Dirac point, recalling the matrix $A(k;\delta)$ in (\ref{equ: A(k;delta) def}), we have the following theorem.
\begin{theorem}\label{thm: Dirac point equiv}
Let $\delta>0$. Then $(k_*,\alpha_*)$ is a Dirac point if and only if
\begin{align}\label{equ: A(k,delta) Dirac}
A(k_*;\delta)=\sigma I,\quad \e^{\i\alpha_* L}=\sigma,\qquad\text{for some }\sigma\in\{1,-1\}.
\end{align}
\end{theorem}
\begin{proof}
Fix $\delta>0$. We first consider the case \(k_*=0\). Example \(\ref{ex: k=0}\) shows that \(A(0;\delta)\ne\pm I\) and that \(k_*=0\) is a resonance only when \(\mathrm e^{\mathrm i\alpha _*L}=1\). In this case, the eigenspace consists precisely of constant functions and hence has dimension one. Thus, neither condition holds at \(k_*=0\). We henceforth assume $k_*\in \R\setminus\{0\}$. Recall the continuous vector $V(x)$ introduced in (\ref{equ: V(x) def}), it is a solution to the system (\ref{equ: V(x) diff equ}). For each \(k\), let \(F(x,k)\) be the fundamental matrix of the system (\ref{equ: V(x) diff equ}), \emph{i.e.} $\partial_xF(x,k)=K(x,k)F(x,k),$ and $F(x_1^-,k)=I$.
More explicitly, using the matrices $T_j^{\text{Res}}(a)$ and $T^{\text{Sp}}(a)$ in Section \ref{sec:transfer}, $F$ is obtained by successive propagation through the resonators and spacings:
\[
F(x,k)=\begin{cases}
    T_j^{\text{Res}}(x-x_j^-)F(x_j^-,k),&x\in[x_j^-,x_{j}^+],\\
    T^{\text{Sp}}(x-x_j^+)F(x_j^+,k),&x\in[x_j^+,x_{j+1}^-].
\end{cases}
\]
Since these propagation matrices have determinant one, recall the definition  of $A(k;\delta)$ in (\ref{equ: A(k;delta) def}),
\[
\det F(x,k)\equiv 1,\qquad F(x_{N+1}^-,k)=A(k;\delta).
\]
In particular, \(F(x,k)\) is invertible. Every solution of system (\ref{equ: V(x) diff equ}) with initial value \(V_1^-=\bm c\in\mathbb C^2\) is given uniquely by $V(x)=F(x,k)\bm c$. In particular, we have $V_{N+1}^-=F(x_{N+1}^-,k)c=A(k;\delta)c$. Write the first row of \(F\) as \((f_1,f_2)\). The corresponding scalar solution is $u(x)=f_1(x,k)c_1+f_2(x,k)c_2$. The Bloch condition is $V_{N+1}^-=\e^{\i\alpha L}V_1^-$, or equivalently $\left(A(k;\delta)-\e^{\i\alpha L}I\right)c=0$. Consequently, the map
\[
\operatorname{Ker}\left(A(k;\delta)-\e^{\i\alpha L}I\right)\to\text{ the eigenspace of (\ref{equ: Bloch problem}) at  } (k,\alpha),\qquad c\to f_1c_1+f_2c_2
\]
is linear and surjective. It is also injective since $u\equiv0$ will give $V_1^-=0$. Thus, the eigenspace of (\ref{equ: Bloch problem}) and the above kernel have the same dimension.

If $(k_*,\alpha_*)$ is a Dirac point, it follows that
\[
\dim\operatorname{Ker}\left(A(k_*;\delta)-\e^{\i\alpha_* L}I\right)=2\Longrightarrow A(k_*;\delta)=\e^{\i\alpha_*L}I.
\]
Taking determinants gives $\e^{2\i\alpha_* L}=1$, so $\e^{\i\alpha_* L}=\sigma\in\{1,-1\}$. This proves necessity.

Conversely, suppose that (\ref{equ: A(k,delta) Dirac}) holds. Then  $\operatorname{Ker}\bigl(A(k_*,\delta)-\e^{\i\alpha_*L}I\bigr)=\mathbb C^2$, so the eigenspace of (\ref{equ: Bloch problem}) has dimension two. It remains to establish the linear crossing (\ref{equ: linear crossing}). Differentiating $\partial_x F=KF$ with respect to $k$ gives
\[
\partial_x(\partial_k F)=(\partial_k K)F+K\partial_kF\Longrightarrow\partial_x(F^{-1}\partial_kF)
=F^{-1}(\partial_kK)F.
\]
Since \(F(x_1^-,k)=I\) is independent of \(k\), integration over the cell yields
\[
A^{-1}\partial_kA
=
\int_{x_1^-}^{x_{N+1}^-}
F^{-1}(\partial_kK)F\,\dd x.
\]
Recall the explicit formula of $K(x,k)$ in (\ref{equ: V(x) diff equ}), and using $\det F\equiv1$, we compute
\begin{align}\label{equ: A^-1 partial_k A}
F^{-1}(\partial_k  K)F=2kb(x)\begin{pmatrix}
    f_1f_2&f_2^2\\-f_1^2&-f_1f_2
\end{pmatrix}\Longrightarrow A^{-1}\partial_kA=2k  JG(k),
\end{align}
where
\[
J=\begin{pmatrix}
    0&1\\-1&0
\end{pmatrix},\quad 
G(k)=\int_{x_1^-}^{x_{N+1}^-}b(x)\begin{pmatrix}
    f_1^2&f_1f_2\\f_1f_2&f_2^2
\end{pmatrix}\dd x.
\]
For $k\in\R$, the matrix $F$ is real, so $G(k)$ is real symmetric. Moreover, for $\bm c=(c_1,c_2)^\top\in\R^2$,
\[
\bm  c^\top G(k)\bm c=\int_{x_1^-}^{x_{N+1}^-}b(x)(c_1f_1+c_2f_2)^2\dd x\geq0.
\]
If this integral vanishes, then the positivity of \(b\) and the continuity of \(c_1f_1+c_2f_2\) imply that the corresponding scalar solution vanishes identically. The injectivity proved above then gives \(\bm c=0\). Hence, \(G(k)\) is positive definite. Set $B:=A(k_*;\delta)^{-1}\partial_k A(k_*;\delta)$. Then (\ref{equ: A^-1 partial_k A}) yields
\[
B=2k_* J G(k_*),\quad \operatorname{tr}B=0,\quad\det B=4k_*^2\det G(k_*)>0.
\]
Since $A(k_*;\delta)=\sigma I$, we can expand $A(k_*+z;\delta)$  as 
\[
A(k_*+z;\delta)=\sigma(I+zB+z^2C+\mc{O}(z^3))\quad\text{with some matrix }C.
\]
Since $\det A(k;\delta)\equiv 1$,
\[
1=1+z\operatorname{tr}B+z^2(\operatorname{tr}C+\det B)+\mc{O}(z^3),
\]
which gives $\operatorname{tr}C=-\det B$. Therefore, with $D(k;\delta)$ defined in (\ref{equ: A(k;delta) def}),
\begin{align}\label{equ: D(k;delta) expand}
D(k_*+z;\delta)=2\sigma-\sigma\det B\cdot z^2+\mc{O}(z^3).
\end{align}
In particular, $\partial_k D(k_*;\delta)=0,\partial_{kk}
D(k_*;\delta)=-2\sigma\det B\ne0$. Define
\[
H(z)=\frac{2-\sigma D(k_*+z;\delta)}{z^2},\quad H(0)=\det B.
\]
By (\ref{equ: D(k;delta) expand}), \(H\) is analytic near zero and positive for a sufficiently small real \(z\). For $\alpha$ near $\alpha_*$, the relation $\e^{\i\alpha_*L}=\sigma$ gives $\cos(\alpha L)=\sigma\cos((\alpha-\alpha_*)L)$. Hence, the dispersion relation $D(k;\delta)=2\cos(\alpha L)$ given in Theorem \ref{thm: resonant conition original} is equivalent to
\begin{align}\label{equ:  z^2H(z)}
z^2H(z)=2-2\cos((\alpha-\alpha_*)L)=4\sin^2\frac{(\alpha-\alpha_*)L}{2}\Longleftrightarrow z\sqrt{H(z)}=\pm 2\sin\frac{(\alpha-\alpha_*)L}{2}. 
\end{align}
The function \(z\sqrt{H(z)}\) has derivative \(\sqrt{\det B}>0\) at zero. Its derivative is therefore bounded below by a positive constant on a sufficiently small interval, on which it is strictly increasing. Its image contains a neighborhood of zero, so each sign in (\ref{equ:  z^2H(z)}) determines a unique local solution \(z_\pm(\alpha)\). The same derivative bound gives \(z_\pm(\alpha)=\mc{O}(\alpha-\alpha_*)\). Substituting this estimate into
\[
\sqrt{\det B}\,z_\pm(\alpha)+\mc{O}(z_\pm(\alpha)^2)
=
\pm (\alpha-\alpha_*)L+\mc{O}\left((\alpha-\alpha_*)^3\right)
\]
yields
\[
z_\pm(\alpha)=\pm\frac{L}{\sqrt{\det B}}(\alpha-\alpha_*)+\mc{O}\left((\alpha-\alpha_*)^2\right).
\]
Thus, the two local dispersion branches satisfy
\[
k_\pm(\alpha)
=
k_*\pm v_*(\alpha-\alpha_*)
+\mc{O}\bigl((\alpha-\alpha_*)^2\bigr),
\qquad
v_*=\frac{L}{2|k_*|\sqrt{\det G(k_*)}}>0.
\]
This proves that \((k_*,\alpha_*)\) is a Dirac point.
\end{proof}
\begin{remark}
At \(k_*=0\), Example~\ref{ex: k=0} shows that a resonance
occurs only when \(\e^{\i\alpha L}=1\), in which case the
eigenspace consists precisely of constant functions.
Thus, no Dirac wavenumber occurs at zero, although
\(n(0)=2N\).

In contrast, let \(k_*\ne0\) satisfy \(n(k_*)=2N\).
By \eqref{equ:defmj},
\(A(k_*;\delta)=(-1)^{\sum_{j=1}^{2N}m_j}I\).
Theorem~\ref{thm: Dirac point equiv} therefore gives a
Dirac point \((k_*,\alpha_*)\), where
\(\e^{\i\alpha_*L}=(-1)^{\sum_{j=1}^{2N}m_j}\),
or equivalently \(\e^{\i\gamma L}=1\). Retaining the terms absorbed into the remainders in
Step~4 of the proof of
Theorem~\ref{thm: eigenmode n=2N} gives the exact
representation for initial data
\(V_1^-=\sqrt\delta(A,k_*B)^\top\):
\begin{equation}\label{eq:exact-central-modes}
u_\delta(x)=
\begin{cases}
\sqrt\delta A\cos(k_*Y(x))
+\dfrac{\delta}{r_i}B\sin(k_*Y(x)),
&x\in[x_i^-,x_i^+],\\[4pt]
\sqrt\delta A\cos(k_*Y(x))
+B\sin(k_*Y(x)),
&x\in[x_i^+,x_{i+1}^-].
\end{cases}
\end{equation}
The choices \((A,B)=(\delta^{-1/2},0)\) and \((0,1)\)
give linearly independent initial vectors and hence,
by uniqueness of the initial-value problem, two exact
eigenmodes spanning the eigenspace.
The first is \(\cos(k_*Y(x))\) throughout the cell;
the second has spacing amplitude one and resonator
amplitude \(\delta/r_i\), as illustrated in
Fig.~\ref{fig:modes}.

Thus, although \(n(k_*)=2N\) in both cases, the
eigenspace is one-dimensional at zero and
two-dimensional at a nonzero Dirac wavenumber.
\end{remark}
\begin{figure}[htbp]
\centering
\includegraphics[width=\textwidth]{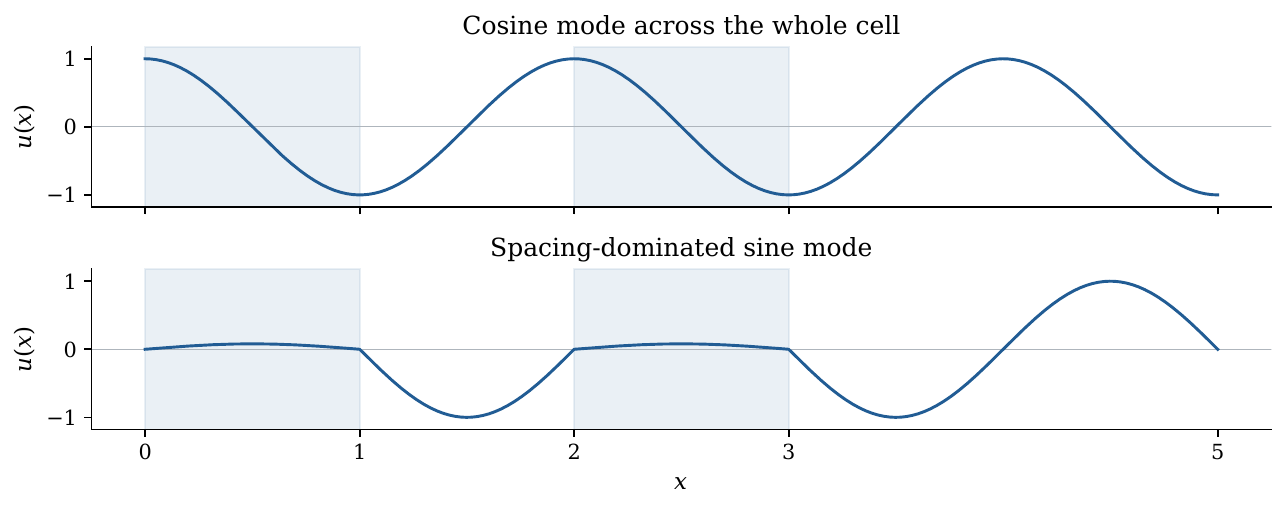}
\caption{Two exact independent central modes from
\eqref{eq:exact-central-modes}, for $N=2$, $r_i=1$,
$\ell_i=1$, $(s_1,s_2)=(1,2)$, $k_0=\pi$, and
$\delta=0.08$. Shading marks the resonators. The first mode is
$\cos(k_0Y)$ throughout the cell; the second has amplitude
$\delta$ in resonators and amplitude one in spacings.
The total phase is $5\pi$, so the Bloch multiplier is $-1$.}
\label{fig:modes}
\end{figure}
We now specialize to the case \(N=2\) and use Theorem \ref{thm: Dirac point equiv} to derive explicit conditions for Dirac points in terms of the geometric and material parameters.

\begin{corollary}[conditions for Dirac points when $N=2$]\label{thm: explict conditions N=2} Let $k\in\R\setminus\{0\},\delta>0$. Define
\[
h_j=\frac{1}{2}\left(\frac{r_j}{\delta}-\frac{\delta}{r_j}\right)\sin(r_j\ell_jk),\quad Z_j=\cos(r_j\ell_jk)+\frac{\i}{2}\left(\frac{r_j}{\delta}+\frac{\delta}{r_j}\right)\sin(r_j\ell_jk),\qquad j=1,2.
\]
Then $(k,\alpha)$ is a Dirac point if and only if, for some $\sigma\in\{1,-1\}$,  $\e^{\i\alpha  L}=\sigma$, and
\begin{align}\label{equ: hZ condition}
h_2\e^{-\i s_2k}=-\sigma  h_1\e^{-\i s_1k},\quad
Z_2\e^{\i s_2k}=\sigma\overline{Z_1\e^{\i s_1k}}.
\end{align}
If $0<\delta<\min\{r_1,r_2\}$, condition (\ref{equ: hZ condition}) admits the following exhaustive alternatives:
\begin{itemize}
    \item $h_1=h_2=0$, and there exist $m_1,m_3,p\in\Z$ such that $\sigma=(-1)^{m_1+m_3+p}$, and
    \begin{align}\label{equ: alternative1}
    r_1\ell_1k=m_1\pi,\quad r_2\ell_2k=m_3\pi,\quad (s_1+s_2)k=p\pi.
    \end{align}
    \item $h_1h_2\ne0$, and there exist $p\in\Z$ and $\eta\in\{1,-1\}$ such that $\sigma=(-1)^p\eta$, and
    \begin{align}\label{equ: alternative2}
    (s_2-s_1)k=p\pi,\quad h_2=-\eta h_1,\quad Z_2\e^{2\i s_1k}=\eta\overline{Z_1}.
    \end{align}
\end{itemize}
The case in which exactly one of $h_1,h_2$ vanishes is impossible.
\end{corollary}
\begin{proof}
By Theorem \(\ref{thm: Dirac point equiv}\), \((k,\alpha)\) is a Dirac point if and only if
\[
\mathrm e^{\mathrm i\alpha L}=\sigma,\qquad
T_2^\text{Sp}T_2^{\text{Res}}T_1^\text{Sp}T_1^\text{Res}=\sigma I
\]
for some \(\sigma\in\{1,-1\}\). The matrix identity is equivalent to
\begin{align}\label{equ: T_2=T_1 inverse}
T_2^\text{Sp}T_2^{\text{Res}}=\sigma \left(T_1^\text{Res}\right)^{-1}\left(T_1^\text{Sp}\right)^{-1}
\end{align}
Set \(Q=\operatorname{diag}(1,k\sqrt\delta)\). Let $\widehat{T_j}^\text{Res}:=Q^{-1}T_j^\text{Res}Q,\widehat{T_j}^\text{Sp}:=Q^{-1}T_j^\text{Sp}Q$, we obtain
\[
\widehat{T_j}^\text{Res}=\begin{pmatrix}
\re Z_j&\im Z_j-h_j\\
-\im Z_j-h_j&\re Z_j
\end{pmatrix},\quad
\widehat{T_j}^\text{Sp}=
\begin{pmatrix}
\cos(s_jk)&\sin(s_jk)\\
-\sin(s_jk)&\cos(s_jk)
\end{pmatrix}.
\]
Let $M_j=\widehat{T_j}^\text{Sp}\widehat{T_j}^\text{Res}$. Then (\ref{equ: T_2=T_1 inverse}) gives $M_2=\sigma M_1^{-1}$.  A direct multiplication gives
\[
M_j=\begin{pmatrix}
\re(Z_j\mathrm e^{\mathrm i s_jk})-h_j\sin(s_jk)
&
\im(Z_j\mathrm e^{\mathrm i s_jk})-h_j\cos(s_jk)
\\
-\im(Z_j\mathrm e^{\mathrm i s_jk})-h_j\cos(s_jk)
&
\re(Z_j\mathrm e^{\mathrm i s_jk})+h_j\sin(s_jk)
\end{pmatrix}.
\]
Since \(\det M_j=1\), comparing the entries of $M_2$ and $\sigma M_1^{-1}$ yields
\[
\begin{aligned}
\re(Z_2\mathrm e^{\mathrm i s_2k})
=\sigma\re(Z_1\mathrm e^{\mathrm i s_1k}),
\im(Z_2\mathrm e^{\mathrm i s_2k})
=-\sigma\im(Z_1\mathrm e^{\mathrm i s_1k}),\\
h_2\sin(s_2k)=-\sigma h_1\sin(s_1k),
h_2\cos(s_2k)=-\sigma h_1\cos(s_1k).
\end{aligned}
\]
These four equalities are equivalent to (\ref{equ: hZ condition}), proving the first assertion.

Now suppose that \(0<\delta<\min\{r_1,r_2\}\). Then $h_j=0$ is equivalent to $r_j\ell_jk\in\pi\Z$. If \(h_1=h_2=0\), write \(r_j\ell_jk=m_{2j-1}\pi,j=1,2\). Since \(Z_j=(-1)^{m_{2j-1}}\), the remaining condition in (\ref{equ: hZ condition}) becomes $\mathrm e^{\mathrm i(s_1+s_2)k}
=\sigma(-1)^{m_1+m_3}$.
This is equivalent to $(s_1+s_2)k=p\pi,
\sigma=(-1)^{m_1+m_3+p}$ for some \(p\in\mathbb Z\).

If \(h_1h_2\ne0\), the first condition in (\ref{equ: hZ condition}) gives $\mathrm e^{\mathrm i(s_2-s_1)k}
=-\sigma h_2/h_1\in\mathbb R$. Its modulus is one, so $(s_2-s_1)k=p\pi$ for some \(p\in\mathbb Z\). Setting \(\eta=(-1)^p\sigma\), the two conditions in (\ref{equ: hZ condition}) reduce to
\[
h_2=-\eta h_1,\qquad
Z_2\mathrm e^{2\mathrm i s_1k}=\eta\overline{Z_1}.
\]
Conversely, these imply the original conditions, so all steps are reversible. Finally, the first condition in (\ref{equ: hZ condition}) forces \(h_1=0\Longleftrightarrow h_2=0\), so exactly one of $h_1,h_2$  vanishing is impossible.
\end{proof}
Assuming that $N=2$, we now use these explicit conditions to give a complete classification of Dirac points near a fixed reference wavenumber \(k_0\in E\) for sufficiently small \(\delta>0\). The classification depends on the resonance set \(I(k_0)\) defined in (\ref{equ: I def}) and the geometric and material parameters. We introduce the completely symmetric configuration, defined by
\[
r_1=r_2=r,\quad \ell_1=\ell_2=\ell,\quad s_1=s_2=s.
\tag{S}
\]
Whenever (S) holds, define
\begin{align}\label{Phi(k,delta) def}
\Phi(k,\delta)=2\delta\cos(sk)\cos(r\ell k)-\left(r+\frac{\delta^2}{r}\right)\sin(sk)\sin(r\ell k).
\end{align}
We are now ready to state the complete local classification of Dirac points for \(N=2\).
\begin{theorem}[Complete local classification of Dirac points for $N=2$]\label{thm: complete classification}
Fix the geometric and material parameters and let \(k_0\in E\). For each $j\in I(k_0)$, let $m_j\in\Z$ be defined by (\ref{equ:defmj}). Then there exist \(\varepsilon,\delta_0>0\), independent of \(\delta\), such that the following statements describe all Dirac points $(k_*,\alpha_*)$ satisfying \(|k_*-k_0|<\varepsilon,\alpha_*\in Y^*\) for every \(0<\delta<\delta_0\).
For \(k_0\ne0\), the classification is as follows; the case \(k_0=0\) is treated separately in part (iv). 
\begin{center}
\begin{tabular}{lcccc}
\toprule
\(I(k_0)\)
& \(\{1,3\}\)
& \(\{2,4\}\)
& \(\{1,2,3,4\}\)
& Other configurations \\
\midrule
Dirac points near \(k_0\)
& See (i)
& See (ii)
& See (iii)
& None \\
\bottomrule
\end{tabular}
\end{center}
\begin{enumerate}
\item[(i)] Suppose that \(I(k_0)=\{1,3\}\). A Dirac point exists near \(k_0\) if and only if
\begin{align}\label{equ: set(1) condition} 
(s_1+s_2)k_0=p\pi\quad\text{for some }p\in\Z
\end{align}
or (S) holds. When it exists, it is unique.
If (\ref{equ: set(1) condition}) holds, its wavenumber $k_*$ is exactly \(k_0\), and its Bloch multiplier is $\e^{\i\alpha_*L}=(-1)^{m_1+m_3+p}$.
If (S) holds, its Bloch multiplier is $\e^{\i\alpha_*L}=-1$ and its wavenumber $k_*$ is the unique root of \(\Phi(k,\delta)=0\) near \(k_0\), with asymptotic form
\[ k_*(\delta) = k_0+\frac{2\cot(sk_0)}{r^2\ell}\delta+\mc{O}(\delta^2). \]
If both conditions hold, this root is exactly \(k_0\).
\item[(ii)] Suppose that \(I(k_0)=\{2,4\}\). A Dirac point exists near \(k_0\) if and only if
\begin{align}\label{equ: set(2) condition}
r_1=r_2=r,\qquad r(\ell_1+\ell_2)k_0=p\pi\quad\text{for some }p\in\Z 
\end{align}
or (S) holds. When it exists, it is unique.
If (\ref{equ: set(2) condition}) holds, its wavenumber $k_*$ is exactly \(k_0\), and its Bloch multiplier is $\e^{\i\alpha_*L}=(-1)^{m_2+m_4+p}$. If (S) holds, its Bloch multiplier is $\e^{\i\alpha_*L}=-1$ and its wavenumber is the unique root of \(\Phi(k,\delta)=0\) near \(k_0\), with asymptotic form
\[ k_*(\delta) = k_0+\frac{2\cot(r\ell k_0)}{rs}\delta+\mc{O}(\delta^2)\]
If both conditions hold, this root is exactly \(k_0\).
\item[(iii)] Suppose that \(k_0\ne0\) and \(I(k_0)=\{1,2,3,4\}\). Set $\sigma=(-1)^{m_1+m_2+m_3+m_4}$. There is always a Dirac point at $k_*=k_0,\e^{\i\alpha_* L}=\sigma.$
If (S) fails, this is the only nearby Dirac point. If (S) holds, then \(\sigma=1\), and there are exactly three nearby Dirac points, given by
\[ k_*=k_0,\e^{\i\alpha_*L}=1,\quad k_*=k_0-\kappa(\delta),\e^{\i\alpha_*L}=-1,\quad k_*=k_0+\kappa(\delta),\e^{\i\alpha_*L}=-1, \]
where \(\kappa(\delta)\) is the unique sufficiently small positive root of \(\Phi(k,\delta)=0\), satisfying
\begin{align}\label{equ: kappa(delta) expand}
\kappa(\delta) = \sqrt{\frac{2\delta}{r^2\ell s}} +\mc{O}(\delta^{3/2}).
\end{align}
\item[(iv)] Suppose that \(k_0=0\). Dirac points near zero exist if and only if (S) holds. In that case, there are exactly two, given by $k_*=\pm\kappa(\delta),\e^{\i\alpha_* L}=-1$, where \(\kappa(\delta)\) is the unique sufficiently small positive root of \(\Phi(k,\delta)=0\) and satisfies (\ref{equ: kappa(delta) expand}). The point $k_0=0$ itself is not a Dirac wavenumber.
\end{enumerate}
\end{theorem}
\begin{proof}
Since $A(-k;\delta)=A(k;\delta),I(-k_0)=I(k_0)$, the characterization in Theorem \(\ref{thm: Dirac point equiv}\) implies that \((k_*,\alpha_*)\) is a Dirac point if and only if \((-k_*,\alpha_*)\) is a Dirac point. Therefore, it suffices to consider \(k_0\ge0\). We take $0<\delta<\min\{r_1,r_2\}$ throughout the proof.

By Corollary \ref{thm: explict conditions N=2}, every Dirac point satisfies one of the two alternatives corresponding to \(h_1=h_2=0\) and \(h_1h_2\ne0\). In either case, \(|h_1|=|h_2|\). To establish local nonexistence, we need to show that there exist \(\varepsilon,\delta_0>0\) such that no Dirac point \((k,\alpha)\) satisfies \(|k-k_0|<\varepsilon\) for any \(0<\delta<\delta_0\). If this were false, taking \(\varepsilon=\delta_0=1/n\), we could choose Dirac points \((k_n,\alpha_n)\) at contrasts \(\delta_n\) that satisfy $|k_n-k_0|<1/n$ and $0<\delta_n<1/n$. Thus, \(k_n\to k_0\) and \(\delta_n\to0\). It therefore suffices to exclude the existence of such sequences.

For $k_0>0$, we first exclude the configurations other than those listed in (i)--(iii). If exactly one of \(1,3\) belongs to \(I(k_0)\), then for $j=1,2$, exactly one of $2\delta h_j
=\left(r_j-\delta^2/r_j\right)\sin(r_j\ell_jk)$ has a nonzero limit as \((k,\delta)\to(k_0,0)\). This contradicts \(|h_1|=|h_2|\). Suppose next that neither \(1\) nor \(3\) belongs to \(I(k_0)\). Then $\sin(r_1\ell_1k_0)\ne0,\sin(r_2\ell_2k_0)\ne0$.
By continuity, both sine factors remain nonzero in a sufficiently small neighborhood of \(k_0\). Since \(0<\delta<\min\{r_1,r_2\}\), we have \(h_1h_2\ne0\) throughout this neighborhood. Thus, only the latter alternative in Corollary \ref{thm: explict conditions N=2} needs to be considered. For the calculations below, write
\[
p_j(k)=r_j\sin(r_j\ell_jk),\quad v_j(k)=\frac{1}{r_j}\sin(r_j\ell_jk),\quad c_j(k)=\cos(r_j\ell_jk)\qquad j=1,2.
\]
Then
\[
h_j=\frac{p_j-\delta^2v_j}{2\delta},
\qquad
Z_j=c_j+\frac{\i}{2\delta}(p_j+\delta^2v_j).
\]
The equation for $h_j$ in (\ref{equ: alternative2}) gives 
\begin{align}\label{equ: p2p1=v2v1}
p_2+\eta p_1=\delta^2(v_2+\eta v_1).
\end{align}
Multiplying the equation for \(Z_j\) in (\ref{equ: alternative2}) by \(\e^{-\i s_1k}\), taking real parts, and using (\ref{equ: p2p1=v2v1}), we obtain
\begin{align}\label{equ:  alternative 2 necessary condition}
\delta(c_2-\eta c_1)\cos(s_1k)
=(p_2-\eta\delta^2v_1)\sin(s_1k).
\end{align}
Since \(p_2(k_0)\ne0\), the equation above implies $\sin(s_1k)=\mc{O}(\delta)$. Moreover, \((s_2-s_1)k\in\pi\mathbb Z\) gives $|\sin(s_2k)|=|\sin(s_1k)|$. Passing to the limit shows that \(2,4\in I(k_0)\), leaving only \(I(k_0)=\{2,4\}\).

Finally, suppose that \(1,3\in I(k_0)\), but exactly one of \(2,4\) belongs to \(I(k_0)\). In a sufficiently small neighborhood, the first two conditions in (\ref{equ: alternative1}) force \(k=k_0\), whereas $(s_1+s_2)k_0\notin\pi\mathbb Z$.
Thus, the former alternative is impossible. Also, $(s_2-s_1)k_0\notin\pi\mathbb Z$, which excludes (\ref{equ: alternative2}) throughout a sufficiently small neighborhood. This proves the nonexistence assertions in the table.

\medskip
\noindent\textbf{Proof of (i).} Suppose that \(I(k_0)=\{1,3\}\). Then (\ref{equ: alternative1}) is possible locally only at \(k=k_0\) and occurs precisely when $(s_1+s_2)k_0=p\pi$ for some \(p\in\mathbb Z\). In this case, a direct computation yields $A(k_0;\delta)=(-1)^{m_1+m_3+p}I$.
By Theorem \(\ref{thm: Dirac point equiv}\), \((k_0,\alpha_*)\) is therefore a Dirac point with Bloch multiplier $\e^{\i\alpha_*L}=(-1)^{m_1+m_3+p}$.

It remains to consider \(k\ne k_0\). After shrinking the neighborhood, both resonator sine factors are nonzero, so (\ref{equ: alternative2}) applies. If \(s_1\ne s_2\), its first phase condition restricts \(k\) to the discrete set $\left\{q\pi/(s_2-s_1):q\in\mathbb Z\right\}$.
A sufficiently small neighborhood of \(k_0\) contains no element of this set other than possibly \(k_0\). Thus, a shifted Dirac point requires \(s_1=s_2=s\). In this case \(p=0\) and \(\eta=\sigma:=\e^{\i\alpha L}\). Put \(x=k-k_0\). Since \(r_j\ell_jk_0=m_{2j-1}\pi\), division of \(h_2=-\sigma h_1\) by \(x\ne0\) gives
\[
\begin{aligned}
(-1)^{m_3}(r_2^2\ell_2-\delta^2\ell_2)
\operatorname{sinc}(r_2\ell_2x)
=-\sigma(-1)^{m_1}(r_1^2\ell_1-\delta^2\ell_1)
\operatorname{sinc}(r_1\ell_1x),
\end{aligned}
\]
where \(\operatorname{sinc}x=\sin x/x\), with value \(1\) at zero. All factors apart from the displayed signs are positive for sufficiently small \(|x|,\delta\). Hence, $\sigma=-(-1)^{m_1+m_3}$ and
\begin{align}\label{equ: sinc equ}
(r_2^2\ell_2-\delta^2\ell_2)\operatorname{sinc}(r_2\ell_2x)
=
(r_1^2\ell_1-\delta^2\ell_1)\operatorname{sinc}(r_1\ell_1x).
\end{align}
If \(r_1^2\ell_1\ne r_2^2\ell_2\), continuity excludes (\ref{equ: sinc equ}) near \((x,\delta)=(0,0)\). Otherwise, write \(r_1^2\ell_1=r_2^2\ell_2=M\), so that \(r_j\ell_j=\sqrt{M\ell_j}\), and define
\[
f(\ell)=(M-\delta^2\ell)\operatorname{sinc}\left(\sqrt{M\ell}|x|\right).
\]
For sufficiently small \(\delta>0\) and \(|x|>0\), both factors are positive and strictly decreasing in \(\ell\) on the interval between \(\ell_1\) and \(\ell_2\). Thus, equation (\ref{equ: sinc equ}), namely \(f(\ell_1)=f(\ell_2)\), therefore implies \(\ell_1=\ell_2\), and hence \(r_1=r_2\). This shows that (S) is necessary for a shifted Dirac point. Under (S), recall (\ref{equ: T_j Res Sp def}), we have $T_1^{\text{Res}}=T_2^{\text{Res}}:=T^{\text{Res}},T_1^{\text{Sp}}=T_2^{\text{Sp}}:=T^{\text{Sp}}$. Write \(T=T^{\mathrm{Sp}}T^{\mathrm{Res}}\). Then
\[
A=T^2,\qquad \det T=1,\qquad
T^2-(\operatorname{tr}T)T+I=0.
\]
By Theorem \(\ref{thm: Dirac point equiv}\), a Dirac point occurs precisely when \(A=\pm I\), with the corresponding Bloch multiplier \(\pm1\). We consider these two cases separately.
If \(A=-I\), the above identity gives $(\operatorname{tr}T)T=0$. Since \(T\) is invertible, \(\operatorname{tr}T=0\). Conversely, \(\operatorname{tr}T=0\) implies \(T^2=-I\). Therefore,
\[
A=-I
\quad\Longleftrightarrow\quad
\operatorname{tr}T=0
\quad\Longleftrightarrow\quad
\Phi(k,\delta)=0,
\]
where the last equivalence follows from
\[
\operatorname{tr}T
=
2\cos(sk)\cos(r\ell k)
-\left(\frac r\delta+\frac\delta r\right)
\sin(sk)\sin(r\ell k)
=\frac{\Phi(k,\delta)}{\delta}.
\]
If \(A=I\), then $(\operatorname{tr}T)T=2I$, so \(T\) is a scalar matrix. Since \(\det T=1\), this implies \(T=\pm I\). Conversely, \(T=\pm I\) clearly implies \(A=I\). Writing \(T=\tau I\), where \(\tau\in\{1,-1\}\), we have $T^{\mathrm{Res}}=\tau(T^{\mathrm{Sp}})^{-1}$. Comparison of the two off-diagonal entries yields
\[
\frac{\delta}{r}\sin(r\ell k)=-\tau\sin(sk),
\qquad
\frac{r}{\delta}\sin(r\ell k)=-\tau\sin(sk).
\]
Since \(0<\delta<r\), we obtain $\sin(r\ell k)=0,\sin(sk)=0$.
Conversely, these equalities imply \(A=I\). However, \(I(k_0)=\{1,3\}\) gives \(\sin(sk_0)\ne0\), so this case is excluded near \(k_0\). Consequently, under (S), the nearby Dirac points are precisely the roots of \(\Phi(k,\delta)=0\), with Bloch multiplier \(\e^{\i\alpha_*L}=-1\). At $(k,\delta)=(k_0,0)$,
\[
\Phi(k_0,0)=0,\quad\partial_k\Phi(k_0,0)=-r^2\ell(-1)^{m_1}\sin(sk_0)\ne0.
\]
By continuity, \(\partial_k\Phi\) has a fixed nonzero sign on a sufficiently small neighborhood. At two fixed endpoints on opposite sides of \(k_0\), the values of \(\Phi(k,0)\) have opposite signs, and these signs persist for sufficiently small \(\delta>0\). Thus, there is exactly one nearby root \(k_*(\delta)\). The derivative bound and \(\Phi(k_0,\delta)=\mc{O}(\delta)\) first give \(k_*(\delta)-k_0=\mc{O}(\delta)\). Taylor expansion then yields
\[
k_*(\delta)-k_0
=-\frac{\partial_\delta\Phi(k_0,0)}{\partial_k\Phi(k_0,0)}\delta
+\mc{O}(\delta^2)
=\frac{2\cot(sk_0)}{r^2\ell}\delta+\mc{O}(\delta^2).
\]
Finally, if (S) and \((s_1+s_2)k_0\in\pi\mathbb Z\) both hold, then $2sk_0\in\pi\mathbb Z,sk_0\notin\pi\mathbb Z$, so \(\cos(sk_0)=0\). Together with \(\sin(r\ell k_0)=0\), this gives \(\Phi(k_0,\delta)=0\) for every \(\delta>0\). The unique root is therefore exactly \(k_0\), which completes the proof of (i).

\medskip
\noindent\textbf{Proof of (ii).} Suppose that \(I(k_0)=\{2,4\}\). Since $\sin(r_1\ell_1k_0)\ne0,\sin(r_2\ell_2k_0)\ne0$,
both sine factors remain nonzero near \(k_0\). Thus, \(h_1h_2\ne0\), and only the second alternative applies. We first consider \(k=k_0\). Since $T_j^{\mathrm{Sp}}(k_0)=(-1)^{m_{2j}}I$, the condition $A(k_0;\delta)=\sigma I$ is equivalent to $T_2^{\mathrm{Res}}(k_0)
=\tau\bigl(T_1^{\mathrm{Res}}(k_0)\bigr)^{-1}$ with $\tau=\sigma(-1)^{m_2+m_4}$. Comparison of the off-diagonal entries gives
\[
\frac{\sin(r_2\ell_2k_0)}{r_2}
=-\tau\frac{\sin(r_1\ell_1k_0)}{r_1},
\qquad
r_2\sin(r_2\ell_2k_0)
=-\tau r_1\sin(r_1\ell_1k_0).
\]
Since both sine factors are nonzero, these imply \(r_1=r_2=r\). The diagonal and off-diagonal conditions then reduce to $\cos(r\ell_2k_0)=\tau\cos(r\ell_1k_0)$ and $
\sin(r\ell_2k_0)=-\tau\sin(r\ell_1k_0)$.
Equivalently, \(\e^{\i r(\ell_1+\ell_2)k_0}=\tau\), so \(r(\ell_1+\ell_2)k_0=p\pi\) for some \(p\in\mathbb Z\). In this case \(A(k_0;\delta)=(-1)^{m_2+m_4+p}I\), and Theorem \(\ref{thm: Dirac point equiv}\) gives the stated Dirac point.

We next consider \(k\ne k_0\). The same discrete-set argument as in (i) shows that a shifted Dirac point requires \(s_1=s_2=s\). Then \(p=0\) and \(\eta=\sigma\). (\ref{equ:  alternative 2 necessary condition}) now becomes
\begin{align}\label{equ: alternative2 necessary condition2}
\tan(sk)=\delta\frac{c_2-\sigma c_1}{p_2-\sigma\delta^2v_1}.
\end{align}
The division is valid when $(k,\delta)$ is near $(k_0,0)$, since $\cos(sk_0)=\pm1,p_2(k_0)\ne0$. We show first that \(r_1=r_2\) is necessary for a shifted Dirac point. Suppose that \(r_1\ne r_2\) and there is a sequence of Dirac points with $k_n\to k_0, \delta_n\to0_+$. Passing to a subsequence fixes $\sigma$. Taking the limit in (\ref{equ: p2p1=v2v1}) gives
\begin{align}\label{equ: Q def}
Q:=r_2\sin(r_2\ell_2k_0)
=-\sigma r_1\sin(r_1\ell_1k_0)\ne0.
\end{align}
Equation (\ref{equ: alternative2 necessary condition2}) first gives $k_n-k_0=\mc{O}(\delta_n)$, and then
\begin{align}\label{equ: kn asymptotic}
k_n-k_0=\mu\delta_n+\mc{O}(\delta_n^2),
\qquad
\mu=
\frac{\cos(r_2\ell_2k_0)-\sigma\cos(r_1\ell_1k_0)}{sQ}.
\end{align}
If \(\mu=0\), equality of the squared cosines implies equality of the squared sines. Together with (\ref{equ: Q def}), this would give \(r_1=r_2\), a contradiction. Thus, \(\mu\ne0\). Set $F_0(k)=p_2(k)+\sigma p_1(k)$ and $F_2(k)=v_2(k)+\sigma v_1(k)$. Equation (\ref{equ: p2p1=v2v1}) reads $F_0(k_n)=\delta_n^2F_2(k_n)$. Since $F_0(k_0)=0$, Taylor expansion in (\ref{equ: kn asymptotic}), successively at order \(\delta_n\) and then at order \(\delta_n^2\), gives
\begin{align}\label{equ: F0 F2 compare}
F_0'(k_0)=0,\qquad
\frac{\mu^2}{2}F_0''(k_0)=F_2(k_0).
\end{align}
The first identity in (\ref{equ: F0 F2 compare}) is $r_2^2\ell_2\cos(r_2\ell_2k_0)=-\sigma r_1^2\ell_1\cos(r_1\ell_1k_0)$. Squaring and using (\ref{equ: Q def}) yields
\begin{align}\label{r_i  Q relation}
(r_1^2\ell_1^2-r_2^2\ell_2^2)(r_1^2-Q^2)=r_2^2\ell_2^2(r_2^2-r_1^2).
\end{align}
Here \(r_1^2-Q^2>0\). Indeed, equality would force \(\cos(r_1\ell_1k_0)=0\), and then \(F_0'(k_0)=0\) would give \(\cos(r_2\ell_2k_0)=0\), then equation (\ref{equ: Q def}) would again imply \(r_1=r_2\). Hence, (\ref{r_i  Q relation}) gives
\[
\operatorname{sign}(r_1^2\ell_1^2-r_2^2\ell_2^2)=\operatorname{sign}(r_2^2-r_1^2).
\]
However, 
\[
F''_0(k_0)=Q(r_1^2\ell_1^2-r_2^2\ell_2^2),\qquad F_2(k_0)=Q\left(\frac{1}{r_2^2}-\frac{1}{r_1^2}\right),
\]
so $F''_0(k_0)F_2(k_0)<0$, contradicting the second equality in (\ref{equ: F0 F2 compare}). This proves local nonexistence when \(r_1\ne r_2\). We may now assume \(r_1=r_2=r\). The equation \(h_2=-\sigma h_1\) reduces to
\begin{align}\label{equ: sin ell_1+sin ell_2}
\sin(r\ell_2k)+\sigma\sin(r\ell_1k)=0.
\end{align}
Denote the left-hand side by \(f_\sigma(k)\). If \(\ell_1\ne\ell_2\), then \(f_\sigma\) is analytic and not identically zero, since $f_\sigma'(0)=r(\ell_2+\sigma\ell_1)\ne0$. Its zeros are therefore isolated. Since \(f_\sigma\) is independent of \(\delta\), we may choose a sufficiently small neighborhood of \(k_0\), independent of \(\delta\) and common to \(\sigma=\pm1\), in which the only possible zero is \(k_0\). Hence, a shifted Dirac point requires \(\ell_1=\ell_2\), and thus (S) holds. Under (S), equation (\ref{equ: sin ell_1+sin ell_2}) and the nonvanishing of \(\sin(r\ell k)\) near \(k_0\) force \(\sigma=-1\). As established in the proof of (i), the nearby Dirac wavenumbers are exactly the roots of \(\Phi(k,\delta)=0\), with Bloch multiplier \(-1\).
At \((k,\delta)=(k_0,0)\),
\[
\Phi(k_0,0)=0,\qquad
\partial_k\Phi(k_0,0)
=-rs\cos(sk_0)\sin(r\ell k_0)\ne0.
\]
The same argument as in (i) gives a unique nearby root satisfying
\[
k_*(\delta)-k_0
=-\frac{\partial_\delta\Phi(k_0,0)}{\partial_k\Phi(k_0,0)}\delta
+\mc{O}(\delta^2)
=\frac{2\cot(r\ell k_0)}{rs}\delta+\mc{O}(\delta^2).
\]
Finally, if (S) and the condition in \((\ref{equ: set(2) condition})\) both hold, then $2r\ell k_0\in\pi\mathbb Z,r\ell k_0\notin\pi\mathbb Z$.
Thus \(\cos(r\ell k_0)=0\). Together with \(\sin(sk_0)=0\), this gives \(\Phi(k_0,\delta)=0\) for every \(\delta>0\). The unique root is therefore exactly \(k_0\), completing the proof of (ii).

\medskip
\noindent\textbf{Proof of (iv).} By Example \ref{ex: k=0} and Theorem \ref{thm: Dirac point equiv}, $k_0=0$ itself is not a Dirac wavenumber. We only need to classify the positive Dirac wavenumber near zero. For sufficiently small $k>0$, $r_j\ell_jk,s_jk\in(0,\pi/2)$ thus $h_j>0$ and only the second alternative applies. In particular $(s_2-s_1)k=p\pi$ for some $p\in\Z$. Taking $k$ sufficiently small forces \(p=0\), hence $s_1=s_2=s$. The remaining conditions then give $h_2=-\sigma h_1,\sigma=\e^{\i\alpha L}\in\{1,-1\}$. Positivity of $h_1,h_2$ forces $\sigma=-1,h_1=h_2$. Dividing  $h_1=h_2$ by $k>0$ gives (\ref{equ: sinc equ}) with $x$ replaced by $k$. The same continuity and monotonicity argument as in (i) yields \(r_1=r_2\) and \(\ell_1=\ell_2\). Thus, (S) is necessary. 

Under (S), the proof of (i) gives $A=-I$ if and only if $\Phi(k,\delta)=0$,  whereas $A=I$ requires $\sin(r\ell k)=\sin(sk)=0$, which is impossible for sufficiently small $k>0$. Therefore, the positive Dirac points near zero are precisely the positive roots of $\Phi(k,\delta)=0$, with Bloch multiplier $-1$. Since $\Phi$ is analytic and  even in $k$, write $\Phi(k,\delta)=\Psi(k^2,\delta)$. Expanding gives
\[
\Psi(y,\delta)=2\delta-r^2\ell sy+\mc{O}(\delta y+y^2)
\]
In particular, $\Psi(0,\delta)=2\delta,\partial_y\Psi(0,0)=-r^2\ell s<0$. Choose $\varepsilon>0$ sufficiently small. By continuity, for all sufficiently small $\delta>0$,
\[
\partial_y\Psi(y,\delta)\leq-\frac{1}{2}r^2\ell s,\quad 0\leq y\leq\varepsilon^2
\quad\Longrightarrow\quad
\Psi(\varepsilon^2,\delta)
\le 2\delta-\frac{r^2\ell s}{2}\varepsilon^2<0.
\]
Together with \(\Psi(0,\delta)>0\), this proves that \(\Phi(k,\delta)=0\) has exactly one root \(\kappa(\delta)\in(0,\varepsilon)\).
The same derivative bound gives $\kappa(\delta)^2\le4\delta/(r^2\ell s)$. Substitution into \(\Phi(\kappa(\delta),\delta)=0\) therefore yields
\[
0=2\delta-r^2\ell s\,\kappa(\delta)^2+\mc{O}(\delta^2),
\qquad
\kappa(\delta)=\sqrt{\frac{2\delta}{r^2\ell s}}+\mc{O}(\delta^{3/2}).
\]
By evenness of \(A\), the only negative Dirac point near zero is \(-\kappa(\delta)\), with the same Bloch multiplier \(-1\). This proves (iv).

\medskip
\noindent\textbf{Proof of (iii).} Suppose that \(k_0\ne0\) and \(I(k_0)=\{1,2,3,4\}\). Then $A(k_0;\delta)=\sigma I$ with $\sigma=(-1)^{m_1+m_2+m_3+m_4}$. Theorem \(\ref{thm: Dirac point equiv}\) gives the central Dirac point with \(\e^{\i\alpha_*L}=\sigma\). To classify the remaining nearby points, write \(k=k_0+x\), where $0<|x|<|k_0|/2$ and define $D_x=\diag\{1,1+k_0/x\}$. We use $T_j^{\text{Res}}(k)$ and $T_j^{\text{Sp}}(k)$ to denote the $T_j^{\text{Res}}$ and $T_j^{\text{Sp}}$ in (\ref{equ: T_j Res Sp def}) to emphasize their dependence on $k$. Direct substitution gives
\[
T_j^{\mathrm{Res}}(k_0+x)
=
(-1)^{m_{2j-1}}
D_xT_j^{\mathrm{Res}}(x)D_x^{-1},\quad
T_j^{\mathrm{Sp}}(k_0+x)
=
(-1)^{m_{2j}}
D_xT_j^{\mathrm{Sp}}(x)D_x^{-1}.
\]
Multiplying these identities yields $A(k_0+x;\delta)=\sigma D_xA(x;\delta)D_x^{-1}$.
Since \(D_x\) is invertible, for either \(\tau\in\{1,-1\}\), $A(k_0+x;\delta)=\tau I$ if and only if $A(x;\delta)=\sigma\tau I$. If (S) fails, part (iv) excludes all sufficiently small nonzero Dirac wavenumbers \(x\). Thus, the central point is the only Dirac point. If (S) holds, then $m_1=m_3,m_2=m_4$ so $\sigma=1$. Part (iv) gives exactly two nonzero Dirac wavenumbers near zero, \(x=\pm\kappa(\delta)\), both satisfying \(A(x;\delta)=-I\). Therefore, the two additional Dirac wavenumbers are \(k_0\pm\kappa(\delta)\), both with Bloch multiplier \(-1\). Together with the central point, whose Bloch multiplier is \(1\), these are precisely the three stated points. The expansion of \(\kappa(\delta)\) follows from part (iv). This proves (iii).
\end{proof}
\begin{remark}
The existence part of Theorem \ref{thm: complete classification} can be summarized as follows. For a fixed $k_0\in E\setminus\{0\}$, a Dirac point exists near $k_0$ for every sufficiently small $\delta>0$ if and only if at least one of the following conditions holds:
\begin{itemize}
    \item $r_1\ell_1k_0,r_2\ell_2k_0,(s_1+s_2)k_0\in\pi \Z$;
    \item $r_1=r_2=r,\quad s_1k_0,s_2k_0,r(\ell_1+\ell_2)k_0\in\pi\Z$; 
    \item (S) holds, and $r\ell k_0\in\pi\Z$ or $sk_0\in\pi\Z$.
\end{itemize}
The first two conditions characterize the reference wavenumber $k_0$ itself as being a Dirac wavenumber for all sufficiently small $\delta$. The last condition additionally accounts for the shifted Dirac points.
\end{remark}
For \(N\ge3\), the condition \(A(k;\delta)=\pm I\) involves a product of more propagation matrices, allowing configurations that are absent when \(N=2\). Consequently, the same reduction does not yield explicit conditions analogous
to those in Corollary \ref{thm: explict conditions N=2}, and the classification in Theorem~\ref{thm: complete classification} does not extend
directly to general arrays. For \(N=2\), symmetry (S) is necessary for Dirac wavenumbers to shift away from the reference wavenumber as \(\delta\to0\), and it reduces the propagation matrix to \(A=T^2\), where \(T=T^{\mathrm{Sp}}T^{\mathrm{Res}}\). More generally, for arbitrary $N$, if \(r_j=r\), \(\ell_j=\ell\), and \(s_j=s\) for all \(j\), then \(A=T^N\). This reduction allows us to characterize all Dirac points for arbitrary \(N\).
\begin{theorem}
\label{thm: complete classification symmetric general N}
Let $N\geq 2$ and assume that \(r_j=r\), \(\ell_j=\ell\), and
\(s_j=s\) for \(1\leq j\leq N\). Fix the geometric and material parameters, let \(k_0\in E\), \(\Phi(k,\delta)\) be defined in (\ref{Phi(k,delta) def}) and consider the equations
\begin{align}\label{Phi(k,c)=2delta c_j}
\Phi(k,\delta)=2\delta c_j,\qquad\text{where }c_j=\cos\frac{j\pi}{N}\text{ for }
1\leq j\leq N-1.
\end{align}
Then there exist \(\varepsilon,\delta_0>0\), independent of
\(\delta\), with \(\delta_0<r\), such that the following statements
describe all Dirac points \((k_*,\alpha_*)\) satisfying $|k_*-k_0|<\varepsilon,\alpha_*\in Y^*$ for every \(0<\delta<\delta_0\).

\begin{enumerate}
\item[(i)]
Suppose that $r\ell k_0=p\pi$ for some $p\in\Z$ and $sk_0\notin\pi\Z$.
There are exactly \(N-1\) nearby Dirac points. Their wavenumbers
\(k_j(\delta)\) are the unique roots near \(k_0\) of (\ref{Phi(k,c)=2delta c_j}) and satisfy
\[
k_j(\delta)
=
k_0+
\frac{2[\cos(sk_0)-(-1)^pc_j]}
{r^2\ell\sin(sk_0)}\delta
+\mc{O}(\delta^2),\qquad 1\leq j\leq N-1.
\]
The corresponding Bloch multipliers are $\e^{\i\alpha_jL}=(-1)^j$.
\item[(ii)]
Suppose that $r\ell k_0\notin\pi\Z$ and $sk_0=q\pi$ for some $q\in\Z$.
There are exactly \(N-1\) nearby Dirac points. Their wavenumbers
\(k_j(\delta)\) are the unique roots near \(k_0\) of (\ref{Phi(k,c)=2delta c_j}) and satisfy
\[
k_j(\delta)
=
k_0+
\frac{2[\cos(r\ell k_0)-(-1)^qc_j]}
{rs\sin(r\ell k_0)}\delta
+\mc{O}(\delta^2),\qquad 1\leq j\leq N-1.
\]
The corresponding Bloch multipliers are $\e^{\i\alpha_jL}=(-1)^j$.
\item[(iii)]
Suppose that \(k_0\ne0\) and $r\ell k_0=p\pi,sk_0=q\pi$ for some $p,q\in\Z$.
There are exactly \(2N-1\) nearby Dirac points, given by $k_*=k_0,\e^{\i\alpha_*L}=(-1)^{N(p+q)}$ and
\[
k_*=k_0\pm\kappa_j(\delta),
\qquad
\e^{\i\alpha_*L}=(-1)^{N(p+q)+j},
\qquad j=1,\ldots,N-1.
\]
where \(\kappa_j(\delta)\) is the unique sufficiently small positive root of (\ref{Phi(k,c)=2delta c_j}), satisfying
\begin{align}\label{equ: kappa general N expand}
\kappa_j(\delta)
=\frac{2}{r}\sin\frac{j\pi}{2N}\sqrt{\frac{\delta}{\ell s}}
+\mc{O}(\delta^{3/2}),
\qquad j=1,\ldots,N-1.
\end{align}

\item[(iv)]
Suppose that \(k_0=0\).
There are exactly \(2(N-1)\) nearby Dirac points, given by
\[
k_*=\pm\kappa_j(\delta),
\qquad
\e^{\i\alpha_*L}=(-1)^j,
\qquad j=1,\ldots,N-1,
\]
where \(\kappa_j(\delta)\) is the unique sufficiently small positive root of (\ref{Phi(k,c)=2delta c_j}) and satisfies
(\ref{equ: kappa general N expand}).
The point \(k_0=0\) itself is not a Dirac wavenumber.
\end{enumerate}
\end{theorem}
\begin{proof}
Under the symmetric configuration,  recall (\ref{equ: T_j Res Sp def}), we have $T_1^{\text{Res}}=\cdots=T_N^{\text{Res}}:=T^{\text{Res}}$ and $T_1^{\text{Sp}}=\cdots=T_N^{\text{Sp}}:=T^{\text{Sp}}$. Write \(T(k;\delta)=T^{\mathrm{Sp}}T^{\mathrm{Res}}\). Then
\begin{align}\label{equ: T(k;delta) properties}
A(k;\delta)=T(k;\delta)^N,\qquad
\det T(k;\delta)=1,\qquad
\operatorname{tr} T(k;\delta)=\frac{\Phi(k,\delta)}{\delta}.
\end{align}

We first establish the exact conditions used in the classification.
By Theorem \ref{thm: Dirac point equiv}, \(k\in\R\) is a Dirac wavenumber precisely
when \(T(k;\delta)^N=\sigma I\) for some \(\sigma\in\{1,-1\}\). If \(T(k;\delta)\ne\pm I\) and \(T(k;\delta)^N=\sigma I\),
then \(T(k;\delta)\) is diagonalizable, since the polynomial
\(z^N-\sigma\) has no repeated roots. Therefore, its eigenvalues are
\(\e^{\pm\i j\pi/N}\) for some \(1\le j\le N-1\). Consequently, (\ref{equ: T(k;delta) properties}) gives $\Phi(k,\delta)=2\delta c_j$ together with $\sigma=(-1)^j$.
Conversely, this equation gives two distinct eigenvalues
\(\e^{\pm\i j\pi/N}\), and hence $T(k;\delta)^N=(-1)^jI$.

For the remaining possibility, assume \(T(k;\delta)=\sigma I\) for some \(\sigma\in\{1,-1\}\).
Then
\(T^{\mathrm{Sp}}=\sigma(T^{\mathrm{Res}})^{-1}\). When $k\ne0$, comparing the off-diagonal entries and multiplying by \(k\) gives
\[
\delta\sin(sk)
=-\sigma r\sin(r\ell k),
\qquad
\frac{\sin(sk)}{\delta}
=-\frac{\sigma}{r}\sin(r\ell k).
\]
Since \(0<\delta<r\), these imply
\(\sin(r\ell k)=\sin(sk)=0\). Conversely, if \(r\ell k=p\pi\) and \(sk=q\pi\), with \(k\ne0\) and \(p,q\in\Z\), then $T^{\mathrm{Res}}=(-1)^pI,T^{\mathrm{Sp}}=(-1)^qI$ and hence
\[
T(k;\delta)=(-1)^{p+q}I,\qquad
A(k;\delta)=(-1)^{N(p+q)}I.
\]
Thus, a nonzero wavenumber \(k\) is a Dirac wavenumber if and only if
\begin{align}\label{equ: complete symmetric Dirac condition}
\Phi(k,\delta)=2\delta c_j
\quad\text{for some }1\le j\le N-1,
\qquad\text{or}\qquad
r\ell k,\;sk\in\pi\Z.
\end{align}
In the first case, the corresponding Bloch multiplier is
\(\e^{\i\alpha L}=(-1)^j\); in the second case, writing
\(r\ell k=p\pi\) and \(sk=q\pi\), it is
\(\e^{\i\alpha L}=(-1)^{N(p+q)}\).

\medskip
\noindent
(i) Suppose that \(r\ell k_0=p\pi\) and \(sk_0\notin\pi\Z\).
At \((k_0,0)\),
\[
\Phi(k_0,0)=0,\quad
\partial_k\Phi(k_0,0)
=-(-1)^p r^2\ell\sin(sk_0)\ne0,\quad \partial_\delta\Phi(k_0,0)=2(-1)^p\cos(sk_0).
\]
The implicit function theorem therefore gives, for each
\(1\le j\le N-1\), a unique root \(k_j(\delta)\) near \(k_0\) of
\(\Phi(k,\delta)=2\delta c_j\). Taylor expansion then yields
\[
k_j(\delta)-k_0
=-\frac{\partial_\delta\Phi(k_0,0)-2c_j}{\partial_k\Phi(k_0,0)}\delta
+\mc{O}(\delta^2)
=\frac{2[\cos(sk_0)-(-1)^p c_j]}
{r^2\ell\sin(sk_0)}\delta+\mc{O}(\delta^2).
\]
The Bloch multipliers are \((-1)^j\). The roots are distinct because the \(c_j\)'s are distinct. By shrinking the neighborhood, \(sk\notin\pi\Z\) throughout it, so the second condition in (\ref{equ: complete symmetric Dirac condition}) contributes no additional point.

\medskip
\noindent
(ii) Suppose that \(sk_0=q\pi\) and \(r\ell k_0\notin\pi\Z\).
Now
\[
\Phi(k_0,0)=0,
\quad
\partial_k\Phi(k_0,0)
=-(-1)^q rs\sin(r\ell k_0)\ne0,
\quad
\partial_\delta\Phi(k_0,0)
=2(-1)^q\cos(r\ell k_0).
\]
The same argument gives exactly \(N-1\) distinct roots, with
\[
k_j(\delta)-k_0
=-\frac{\partial_\delta\Phi(k_0,0)-2c_j}{\partial_k\Phi(k_0,0)}\delta
+\mc{O}(\delta^2)=
\frac{2[\cos(r\ell k_0)-(-1)^q c_j]}
{rs\sin(r\ell k_0)}\delta+\mc{O}(\delta^2),
\]
and Bloch multipliers \((-1)^j\).

\medskip
\noindent
(iv) Next, we consider zero. By Example \ref{ex: k=0} and Theorem \ref{thm: Dirac point equiv}, $k_0=0$ itself is not a Dirac wavenumber. Thus, we only need to consider nonzero wavenumbers near zero. Since \(\Phi(k,\delta)\) is even in \(k\),
\(\Phi(\sqrt{\mr{z}},\delta)\) extends analytically to a neighborhood
of \((\mr{z},\delta)=(0,0)\). Its expansion is
\[
\Phi(\sqrt{\mr{z}},\delta)-2\delta c_j
=
2(1-c_j)\delta-r^2\ell s\,\mr{z}
+\mc{O}(\delta \mr{z}+\mr{z}^2).
\]
In particular,
\[
\left.
\partial_z\bigl[\Phi(\sqrt z,\delta)-2\delta c_j\bigr]
\right|_{(0,0)}
=-r^2\ell s\ne0.
\]
The implicit function theorem gives a unique root \(z_j(\delta)\)
near zero, satisfying
\[
z_j(\delta)
=
\frac{2(1-c_j)}{r^2\ell s}\delta+\mc{O}(\delta^2)>0.
\]
Thus, the only roots near zero are
\(k=\pm\kappa_j(\delta)\), where
\[
\begin{aligned}
\kappa_j(\delta)
=\sqrt{z_j(\delta)}=
\frac{2}{r}\sin\frac{j\pi}{2N}
\sqrt{\frac{\delta}{\ell s}}
+\mc{O}(\delta^{3/2}).
\end{aligned}
\]
They are distinct and have Bloch multipliers \((-1)^j\).

\medskip
\noindent
(iii) Finally, suppose that \(k_0\ne0\),
\(r\ell k_0=p\pi\), and \(sk_0=q\pi\).
Then
\[
T(k_0;\delta)=(-1)^{p+q}I_2,\qquad
A(k_0;\delta)=(-1)^{N(p+q)}I_2.
\]
Thus, \(k_0\) is a Dirac wavenumber with the stated multiplier. Moreover, for \(x\) near zero,
\[
\Phi(k_0+x,\delta)=(-1)^{p+q}\Phi(x,\delta).
\]
If \(p+q\) is even, the roots of
\(\Phi(k_0+x,\delta)=2\delta c_j\) are
\(x=\pm\kappa_j(\delta)\).
If \(p+q\) is odd, use \(c_{N-j}=-c_j\); the same collection of
roots is obtained after replacing \(j\) by \(N-j\).
Therefore, the additional wavenumbers are precisely
\(k_0\pm\kappa_j(\delta)\).
Their Bloch multipliers are
\[
\begin{cases}
(-1)^j,&p+q\text{ even},\\
(-1)^{N-j},&p+q\text{ odd},
\end{cases}
=
(-1)^{N(p+q)+j}.
\]
After shrinking the neighborhood, \(k_0\) is the only wavenumber satisfying \(r\ell k,sk\in\pi\mathbb Z\). Hence, there are exactly \(1+2(N-1)=2N-1\) Dirac points.
\end{proof}

\begin{remark}[Primitive period and band folding]\label{rem:folding}
In Theorem~\ref{thm: complete classification symmetric general N},
the medium has the smaller period $d=\ell+s$, whereas the cell
used in the theorem has length $L=Nd$. If
$\operatorname{tr}T=2\cos\beta$, the two primitive-cell
multipliers are $\e^{\pm \i\beta}$. They become the same
$L$-cell multiplier at $\beta=j\pi/N$.
Consequently, the crossings determined by
$\Phi=2\delta\cos(j\pi/N)$ are exact crossings of folded bands.
They need not survive a perturbation that breaks translation by
$d$ while preserving translation by $L$. The scalar-monodromy
criterion remains valid, but the smaller-period symmetry that
enforced it is then absent. This observation is separate from
the nonzero common-resonance crossings at $T=\pm I$.
\end{remark}

\begin{figure}[htbp]
\centering
\includegraphics[width=\textwidth]{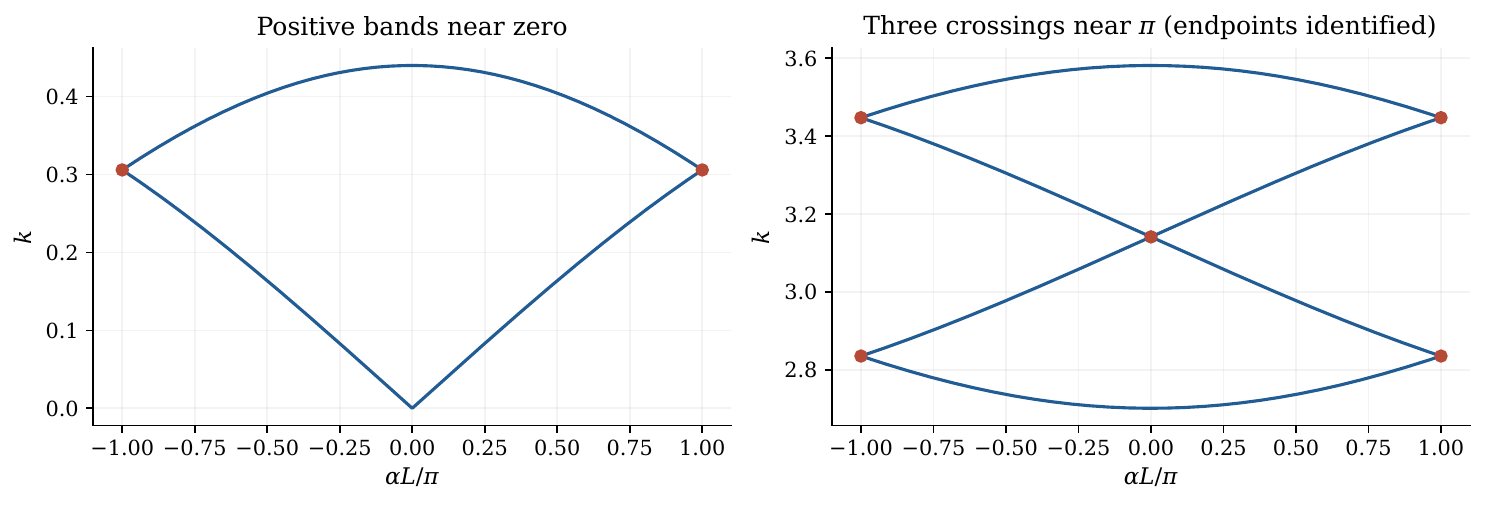}
\caption{Exact dispersion for the completely symmetric two-resonator
cell, with $r=\ell=s=1$, $L=4$, and $\delta=0.05$.
Left: the two positive subwavelength bands meet at the folded-zone
edge. Right: four bands near $k_0=\pi$ have the central crossing
at $\alpha=0$ and the two shifted crossings at the identified
zone endpoints. Markers show the exact scalar-monodromy points.
Negative-frequency subwavelength branches are not plotted.}
\label{fig:dirac}
\end{figure}

\section{Zak phase}\label{sec: Zak phase}
The Zak phase describes the holonomy of a simple isolated Bloch
band over the Brillouin zone~\cite{zak1989,Lin2022Topological}.
Here we define it using the quasiperiodic Bloch eigenmode itself,
rather than its periodic part. Fix \(\delta>0\). Let \(p_\delta\) and \(w_\delta\) be the
\(L\)-periodic extensions to \(\R\) of \(a^{-1}\) and \(b\),
respectively, defined on \([x_1^-,x_{N+1}^-]\) in the proof
of Proposition~\ref{thm: k real}. The equation becomes
\[
-(p_\delta u')'=k^2w_\delta u,
\]
with continuity of \(u\) and \(p_\delta u'\) at interfaces.
In this section, \(Y\) denotes a period cell of length \(L\),
not necessarily \([x_1^-,x_{N+1}^-]\).
Eigenmodes are understood through their quasiperiodic extensions
to \(\R\). For the periodic part \(v_\alpha(x)=\e^{-\i\alpha x}u_\alpha(x)\),
the fiber operator on \(L^2(Y,w_\delta\,\dd x)\) is
\begin{equation}\label{eq:zak-fiber}
H_\delta(\alpha)
=-w_\delta^{-1}(\partial_x+\i\alpha)
p_\delta(\partial_x+\i\alpha).
\end{equation}
It can equivalently be defined by its closed form on
\(H^1_{\mathrm{per}}(Y)\), whose domain is independent of
\(\alpha\). Its compact resolvent and analytic dependence on
\(\alpha\) allow a smooth normalized eigenfunction along any
simple isolated band of \(k^2\); see~\cite{kuchment2016overview}.

We use the weighted inner product
\[
\langle f,g\rangle_w
:=\int_Yw_\delta(x)\overline{f(x)}g(x)\,\dd x,
\]
which is conjugate-linear in its first argument.
Choose a smooth gauge satisfying
\(\langle u_\alpha,u_\alpha\rangle_w=1\) and
\(u_{\pi/L}=u_{-\pi/L}\).
The Zak phase is defined by
\begin{equation}\label{eq:zak-definition}
\varphi_{\rm Zak}
=\i\int_{-\pi/L}^{\pi/L}
\langle u_\alpha,\partial_\alpha u_\alpha\rangle_w
\,\dd\alpha
\pmod{2\pi}.
\end{equation}
An overall positive constant multiplying \(w_\delta\)
has no effect after normalization. For comparison, using the same spatial origin, the periodic-part
convention is
\begin{equation}\label{eq:zak-periodic-definition}
\varphi_{\rm per}
=\i\int_{-\pi/L}^{\pi/L}
\langle v_\alpha,\partial_\alpha v_\alpha\rangle_w
\,\dd\alpha
\pmod{2\pi},
\end{equation}
with endpoint identification
\(v_{\pi/L}(x)=\e^{-2\pi\i x/L}v_{-\pi/L}(x)\).

\begin{remark}
Conventions (\ref{eq:zak-definition}) and (\ref{eq:zak-periodic-definition}) are both gauge invariant modulo \(2\pi\):
an endpoint-compatible gauge transformation changes either
integral by
\(-\chi(\pi/L)+\chi(-\pi/L)\in2\pi\Z\).
At a fixed spatial origin, changing the integration cell leaves
\(\varphi_{\rm per}\) unchanged, whereas
\(\varphi_{\rm Zak}\) generally changes.
Conversely, changing coordinates to \(x'=x-x_0\), while retaining
the same physical cell, leaves \(\varphi_{\rm Zak}\) unchanged
and shifts \(\varphi_{\rm per}\) by \(-2\pi x_0/L\).
Thus, the quasiperiodic convention requires specifying the
period cell, but not the spatial origin, whereas the periodic-part
convention requires specifying the spatial origin, but not the
period cell. For \(v_\alpha(x)=\e^{-\i\alpha x}u_\alpha(x)\), the two Zak-phase conventions satisfy
\begin{equation*}\label{eq:zak-conventions-difference}
\varphi_{\rm per}-\varphi_{\rm Zak}
=
\int_{-\pi/L}^{\pi/L}\int_Y
xw_\delta|u_\alpha|^2\,\dd x\,\dd\alpha
\pmod{2\pi}.
\end{equation*}
We adopt the quasiperiodic convention to apply the eigenmode
asymptotics directly, specifying \(Y\) for each parameter
configuration.
\end{remark}

Time-reversal symmetry alone does not quantize the Zak phase.
Under reflection symmetry, the quasiperiodic Zak phase is
quantized to \(0\) or \(\pi\) modulo \(2\pi\)~\cite{CoutantLombard2024} and agrees with
the periodic-part Zak phase whose spatial origin is the
reflection center, provided both are computed on the same
cell centered there.

\begin{proposition}[Reflection symmetry]\label{prop:zak}
Suppose that \(p_\delta,w_\delta\) are reflection symmetric
about \(x_c\), and consider a simple band isolated over the
whole Brillouin zone. Let \(\xi_0,\xi_{\pi/L}\in\{1,-1\}\) be
the reflection eigenvalues about \(x_c\) of the Bloch modes
at \(\alpha=0,\pi/L\), respectively.
\begin{enumerate}
\item[(i)]
Define \(\varphi_{\rm Zak}\) using
\(Y=[x_c-L/2,x_c+L/2]\), and define \(\varphi_{\rm per}\)
with spatial origin at \(x_c\). Then
\[
\varphi_{\rm per}=\varphi_{\rm Zak}\pmod{2\pi},
\qquad
\e^{\i\varphi_{\rm Zak}}=\xi_0\xi_{\pi/L}.
\]
In particular, both phases are quantized to \(0\) or \(\pi\)
modulo \(2\pi\).
\item[(ii)]
The point \(x_c+L/2\) is also a reflection center.
Replacing the period cell \([x_c-L/2,x_c+L/2]\) by
\([x_c,x_c+L]\) changes \(\varphi_{\rm Zak}\) by \(\pi\)
modulo \(2\pi\). Likewise, moving the spatial origin from
\(x_c\) to \(x_c+L/2\) changes \(\varphi_{\rm per}\) by
\(\pi\) modulo \(2\pi\).
\end{enumerate}
\end{proposition}

\begin{proof}
As for (i), consider \(Y=[x_c-L/2,x_c+L/2]\) and let
\(\mathcal Rf(x)=f(2x_c-x)\).
It is unitary for the weighted inner product and maps Bloch
eigenmodes at \(\alpha\) to those at \(-\alpha\). At \(\alpha=0,\pi/L\), the Bloch multipliers for
\(\alpha\) and \(-\alpha\) coincide, so \(\mathcal R\)
preserves the corresponding one-dimensional eigenspace.
Hence \(\mathcal Ru_0=\xi_0u_0\) and
\(\mathcal Ru_{\pi/L}=\xi_{\pi/L} u_{\pi/L}\), with
\(\xi_0,\xi_{\pi/L}\in\{1,-1\}\) because \(\mathcal R^2=I\). Simplicity and normalization give a smooth real phase
\(\chi(\alpha)\) such that
\[
u_{-\alpha}=\e^{\i\chi(\alpha)}\mathcal Ru_\alpha,
\qquad 0\leq\alpha\leq\pi/L.
\]
If \(\mathcal A(\alpha)
=\i\langle u_\alpha,\partial_\alpha u_\alpha\rangle_w\),
differentiation, with the minus sign from the variable
\(-\alpha\), yields
\[
\mathcal A(-\alpha)=\chi'(\alpha)-\mathcal A(\alpha).
\]
Consequently,
\(\varphi_{\rm Zak}=\chi(\pi/L)-\chi(0)\pmod{2\pi}\).
At zero, \(\mathcal Ru_0=\xi_0u_0\) gives
\(\e^{\i\chi(0)}=\xi_0\).
At the zone edge, \(u_{-\pi/L}=u_{\pi/L}\) and
\(\mathcal Ru_{\pi/L}=\xi_{\pi/L} u_{\pi/L}\) give
\(\e^{\i\chi(\pi/L)}=\xi_{\pi/L}\).
Thus, \(\e^{\i\varphi_{\rm Zak}}=\xi_0\xi_{\pi/L}\).

With spatial origin at \(x_c\), the periodic part is
\(v_\alpha(x)=\e^{-\i\alpha(x-x_c)}u_\alpha(x)\).
Using the same cell \(Y\), we obtain
\[
\varphi_{\rm per}-\varphi_{\rm Zak}
=
\int_{-\pi/L}^{\pi/L}\int_Y
(x-x_c)w_\delta|u_\alpha|^2\,\dd x\,\dd\alpha
\pmod{2\pi}.
\]
The reflection relation and the substitution \(x\mapsto2x_c-x\)
give
\[
\int_Y(x-x_c)w_\delta|u_{-\alpha}|^2\,\dd x
=
-\int_Y(x-x_c)w_\delta|u_\alpha|^2\,\dd x.
\]
Hence, the preceding integral vanishes, proving
\(\varphi_{\rm per}=\varphi_{\rm Zak}\pmod{2\pi}\).

As for (ii), periodicity and reflection symmetry imply
\(p_\delta(2x_c+L-x)=p_\delta(x)\), and likewise for
\(w_\delta\), so \(x_c+L/2\) is also a reflection center.
Moreover, quasiperiodicity gives
\[
u_\alpha(2x_c+L-x)
=\e^{\i\alpha L}u_\alpha(2x_c-x).
\]
Thus, the reflection eigenvalues about \(x_c+L/2\) are
\(\xi_0\) and \(-\xi_{\pi/L}\) at zero and the zone edge,
respectively. Applying (i) to \([x_c,x_c+L]\) shows that
the quasiperiodic Zak phase changes by \(\pi\) modulo \(2\pi\).
By (i), the periodic-part Zak phases with origins at the
respective reflection centers are equal to these quasiperiodic
phases, and therefore also differ by \(\pi\) modulo \(2\pi\).
\end{proof}

When $\delta\to0$, the endpoint reflection eigenvalues can be determined from
the parity of the leading eigenmode profiles, provided that
the asymptotic remainders converge uniformly.

\begin{lemma}\label{lem:parity asymptotics}
Let $k_\delta(\alpha)$ be a simple isolated band for all
sufficiently small $\delta>0$, and assume the reflection
symmetry hypotheses of Proposition~\ref{prop:zak}.
At the endpoints $\alpha_*=0,\pi/L$, suppose that the
corresponding eigenmodes can be normalized so that
\[
 u_\delta^{\alpha_*}
 =u_0^{\alpha_*}+o(1)
 \qquad\text{uniformly on }[x_1^-,x_{N+1}^-].
\]
where $u_0^{\alpha_*}\neq0$ and $\mathcal R u_0^{\alpha_*}
 =\xi_{\alpha_*}u_0^{\alpha_*}$ for some $\xi_{\alpha_*}\in\{1,-1\}$. Then, for all sufficiently small $\delta>0$,
\[
 \mathcal R u_\delta^{\alpha_*}
 =\xi_{\alpha_*}u_\delta^{\alpha_*},
 \qquad \alpha_*=0,\pi/L,
\]
and the Zak phase $\theta_\delta$ satisfies $\e^{\i\theta_\delta}=\xi_0\xi_{\pi/L}$.
Consequently, $\theta_\delta=0$ modulo $2\pi$ if the two
leading profiles have the same parity, and
$\theta_\delta=\pi$ modulo $2\pi$ if they have opposite
parities.
\end{lemma}

\begin{proof}
Fix $\alpha_*\in\{0,\pi/L\}$.
By reflection symmetry, $\mathcal R u_\delta^{\alpha_*}$
belongs to the same eigenspace as $u_\delta^{\alpha_*}$.
Since this eigenspace is one-dimensional and
$\mathcal R^2=I$, there exists
$\xi_{\delta,\alpha_*}\in\{1,-1\}$ such that $\mathcal R u_\delta^{\alpha_*}
 =\xi_{\delta,\alpha_*}u_\delta^{\alpha_*}$. If $\xi_{\delta,\alpha_*}=-\xi_{\alpha_*}$, then
\[
 (\mathcal R+\xi_{\alpha_*}I)
 u_\delta^{\alpha_*}=0,
 \qquad
 (\mathcal R+\xi_{\alpha_*}I)
 u_0^{\alpha_*}=2\xi_{\alpha_*}u_0^{\alpha_*}.
\]
Using that reflection
preserves the uniform norm, we obtain
\[
 \begin{aligned}
 2\|u_0^{\alpha_*}\|_{C[x_1^-,x_{N+1}^-]}
 =
 \bigl\|(\mathcal R+\xi_{\alpha_*}I)
       (u_0^{\alpha_*}-u_\delta^{\alpha_*})\bigr\|_{C[x_1^-,x_{N+1}^-]}
\leq
 2\|u_0^{\alpha_*}-u_\delta^{\alpha_*}\|_{C[x_1^-,x_{N+1}^-]}
 =o(1).
 \end{aligned}
\]
Since $u_0^{\alpha_*}\neq0$, this is impossible for
sufficiently small $\delta$.
Hence $\xi_{\delta,\alpha_*}=\xi_{\alpha_*}$ at both
endpoints. Proposition~\ref{prop:zak} therefore gives $\e^{\i\theta_\delta}=\xi_0\xi_{\pi/L}$, which proves the assertion.
\end{proof}

We now specialize to $N=2$ and identify the geometric
conditions under which the reflection symmetry holds.

\begin{proposition}
\label{prop:reflection N=2}
Let \(N=2\) and \(0<\delta<\min\{1,r_1,r_2\}\). The material coefficients
\(p_\delta,w_\delta\) are reflection symmetric if and only if
at least one of the following conditions holds:
\begin{enumerate}
\item[(i)]
\(r_1=r_2\) and \(\ell_1=\ell_2\). The corresponding
reflection centers are
\(x_c=x_1^++s_1/2\) modulo \(L/2\);

\item[(ii)]
\(s_1=s_2\). The corresponding reflection centers are
\(x_c=x_1^-+\ell_1/2\) modulo \(L/2\).
\end{enumerate}
These give all reflection centers; if both conditions hold,
that is, if the completely symmetric configuration (S) holds,
both families of centers occur.
\end{proposition}
\begin{proof}
By \(L\)-periodicity, it suffices to classify reflection
centers in \([x_1^-,x_3^-)\).
Since \(\delta<\min\{1,r_1,r_2\}\), the values of \(p_\delta\) and $w_\delta$ on resonators and spacings are distinct. A reflection preserving
\(p_\delta,w_\delta\) must therefore preserve the two types
of intervals. Its center cannot be an interface, since
reflection would exchange these types locally.
Thus, the center lies inside a resonator or a spacing.
The interval containing it is mapped to itself, so the
center is its midpoint.

If the center lies in a spacing, reflection exchanges its
two adjacent resonators. Equality of their lengths and
material coefficients gives \(\ell_1=\ell_2\) and
\(r_1=r_2\). The possible centers in \([x_1^-,x_3^-)\) are $x_1^++s_1/2$ or $x_2^++s_2/2$. Their difference is
\(\ell_2+(s_1+s_2)/2=L/2\), yielding the family in (i)
after periodic extension.

If the center lies in a resonator, reflection exchanges its
two adjacent spacings, giving \(s_1=s_2\).
The possible centers in \([x_1^-,x_3^-)\) are $x_1^-+\ell_1/{2}$ or $x_2^-+\ell_2/{2}$.
Their difference is
\((\ell_1+\ell_2)/2+s_1=L/2\), yielding the family in (ii)
after periodic extension. This proves necessity and
exhausts all possible centers.

Conversely, under (i), reflection about \(x_1^++s_1/2\)
preserves each spacing modulo \(L\) and exchanges the two
resonators, whose lengths and material coefficients agree.
Under (ii), reflection about \(x_1^-+\ell_1/2\) preserves
each resonator modulo \(L\) and exchanges the equal-length
spacings. Hence, both reflections preserve
\(p_\delta,w_\delta\) and $x_c +L/2$ is also a reflection center. 
Therefore, the stated families give all reflection centers.
If both conditions hold, or equivalently, if (S) holds, both
families occur.
\end{proof}
\begin{remark}
Let \(k_0\in E\), so that \(I(k_0)\ne\varnothing\).
Under condition (i) of Proposition~\ref{prop:reflection N=2},
the resonator phases coincide, hence
\(1\in I(k_0)\) if and only if \(3\in I(k_0)\).
Thus, configurations containing exactly one of \(1,3\)
are excluded, and
\[
I(k_0)\in
\bigl\{\{2\},\{4\},\{2,4\},\{1,3\},
\{1,2,3\},\{1,3,4\},\{1,2,3,4\}\bigr\}.
\]
Under condition (ii), the spacing phases coincide, hence
\(2\in I(k_0)\) if and only if \(4\in I(k_0)\).
Thus, configurations containing exactly one of \(2,4\)
are excluded, and
\[
I(k_0)\in
\bigl\{\{1\},\{3\},\{1,3\},\{2,4\},
\{1,2,4\},\{2,3,4\},\{1,2,3,4\}\bigr\}.
\]
In particular, under (S), only
\(I(k_0)=\{1,3\}\), \(\{2,4\}\), or \(\{1,2,3,4\}\)
can occur.
\end{remark}

Combining the configurations described above with the
eigenmode expansions in Appendix~\ref{app:asymptotics N2} determines
the parity of the endpoint profiles and yields the following
characterization of the Zak phase.

\begin{theorem}[Zak phases for \(N=2\)]
\label{thm: Zak phases N2}
Let \(N=2\), fix the geometric and material parameters,
and let \(k_0\in E\). Assume that at least one of the
following reflection symmetry conditions holds:
\begin{enumerate}
\item[(i)]
\(r_1=r_2=r\) and \(\ell_1=\ell_2=\ell\), with
\(x_c=x_1^++s_1/2\);
\item[(ii)]
\(s_1=s_2=s\), with \(x_c=x_1^-+\ell_1/2\).
\end{enumerate}
Let \(Y=[x_c-L/2,x_c+L/2]\), and let
\(\varphi_{\rm Zak}\) denote the quasiperiodic Zak phase
defined on \(Y\).
For the local Bloch bands in Appendix~\ref{app:asymptotics N2}
that are simple and isolated throughout the Brillouin zone,
the following statements hold for sufficiently small
\(\delta>0\), with all comparisons strict and all Zak phases
taken modulo \(2\pi\).

\begin{enumerate}
\item[(a)]
Suppose \(k_0\ne0\) and
\(I(k_0)\ne\{1,2,3,4\}\).
For the two branches
\(k_\pm(\alpha;\delta)
=k_0+\lambda_\pm(\alpha)\delta+\mc O(\delta^2)\),
the terms ``lower'' and ``upper'' refer to the ordering
\(\lambda_-<\lambda_+\).

Under symmetry (i), the Zak phases are given by
\begin{center}
\begin{tabular}{lll}
\toprule
\(I(k_0)\) & Condition & Zak phases \\
\midrule
\(\{2\}\) & --- & \(0\) \\
\(\{4\}\) & --- & \(\pi\) \\
\(\{1,3\}\)
& \(|\sin(s_1k_0)|<|\sin(s_2k_0)|\)
& Both \(0\) \\
\(\{1,3\}\)
& \(|\sin(s_1k_0)|>|\sin(s_2k_0)|\)
& Both \(\pi\) \\
\(\{2,4\}\)
& \(\cot(r\ell k_0)/s_1<\cot(r\ell k_0)/s_2\)
& Lower \(0\), upper \(\pi\) \\
\(\{2,4\}\)
& \(\cot(r\ell k_0)/s_1>\cot(r\ell k_0)/s_2\)
& Lower \(\pi\), upper \(0\) \\
\(\{1,2,3\}\) & --- & All three \(0\) \\
\(\{1,3,4\}\) & --- & All three \(\pi\) \\
\bottomrule
\end{tabular}
\end{center}

Under symmetry (ii), they are given by
\begin{center}
\begin{tabular}{lll}
\toprule
\(I(k_0)\) & Condition & Zak phases \\
\midrule
\(\{1\}\) & --- & \(0\) \\
\(\{3\}\) & --- & \(\pi\) \\
\(\{1,3\}\)
& \(\cot(sk_0)/(r_1^2\ell_1)
   <\cot(sk_0)/(r_2^2\ell_2)\)
& Lower \(0\), upper \(\pi\) \\
\(\{1,3\}\)
& \(\cot(sk_0)/(r_1^2\ell_1)
   >\cot(sk_0)/(r_2^2\ell_2)\)
& Lower \(\pi\), upper \(0\) \\
\(\{2,4\}\)
& \(r_1|\sin(r_1\ell_1k_0)|
   <r_2|\sin(r_2\ell_2k_0)|\)
& Both \(0\) \\
\(\{2,4\}\)
& \(r_1|\sin(r_1\ell_1k_0)|
   >r_2|\sin(r_2\ell_2k_0)|\)
& Both \(\pi\) \\
\(\{1,2,4\}\) & --- & All three \(0\) \\
\(\{2,3,4\}\) & --- & All three \(\pi\) \\
\bottomrule
\end{tabular}
\end{center}

\item[(b)]
Suppose \(k_0\ne0\) and \(I(k_0)=\{1,2,3,4\}\).
The two inner branches meet at the exact Dirac
wavenumber \(k_0\) and are excluded from the single-band
Zak phase considered here.
The two outer branches, associated with the larger
effective eigenvalue, have the same Zak phase:
\[
\begin{array}{c|cc}
& \varphi_{\rm Zak}=0 & \varphi_{\rm Zak}=\pi \\ \hline
\text{Symmetry (i)} & s_1<s_2 & s_1>s_2 \\[2pt]
\text{Symmetry (ii)}
& r_1^2\ell_1<r_2^2\ell_2
& r_1^2\ell_1>r_2^2\ell_2
\end{array}
\]

\item[(c)]
Suppose \(k_0=0\), and order the two subwavelength
bands by their eigenvalues \(k^2\).
Under symmetry (i), both Zak phases are \(0\) if
\(s_1<s_2\), and both are \(\pi\) if \(s_1>s_2\).
Under symmetry (ii), the lower and upper bands have
Zak phases \(0,\pi\), respectively, if
\(r_1^2\ell_1>r_2^2\ell_2\), and \(\pi,0\), respectively,
if \(r_1^2\ell_1<r_2^2\ell_2\).

\item[(d)]
The same results hold for the periodic-part convention
with spatial origin at the specified \(x_c\).
Replacing the quasiperiodic integration cell by the cell
centered at \(x_c+L/2\), or moving the periodic-part
spatial origin to \(x_c+L/2\), adds \(\pi\) to every
listed phase modulo \(2\pi\).
\end{enumerate}
\end{theorem}

\begin{proof}
By Proposition~\ref{prop:zak}, the Zak
phase of a simple isolated band satisfies
$\e^{\i\varphi_{\rm Zak}}=\xi_0\xi_{\pi/L}$,
where $\xi_0$ and $\xi_{\pi/L}$ are its endpoint
reflection eigenvalues.

The leading eigenmode profiles are given in
Propositions~\ref{prop:app N2 single},
\ref{prop:app N2 opposite},
\ref{prop:app N2 triple}, and~\ref{prop:lowmodes},
and Theorem~\ref{thm: eigenmode n=2N}.
At $\alpha=0,\pi/L$, their parities under $\mathcal R$
are obtained using
$\cos(m\pi-z)=(-1)^m\cos z$,
$\sin(m\pi-z)=(-1)^{m+1}\sin z,m\in\Z$, and
$\sigma=\e^{\i\alpha L}\in\{1,-1\}$ whenever reflection
maps a profile to a neighboring cell.
For profiles involving effective eigenvectors,
the stated strict comparisons exclude repeated
effective eigenvalues at the endpoints.
The corresponding effective matrices therefore
determine the eigenvectors up to a nonzero scalar
factor, as well as their eigenvalue ordering.

These calculations give endpoint parity product $1$
for each entry listed with Zak phase $0$ and $-1$
for each entry listed with Zak phase $\pi$.
Since the leading profiles are nonzero and the
remainders converge uniformly,
Lemma~\ref{lem:parity asymptotics} shows that
their parities equal the exact reflection eigenvalues
for sufficiently small $\delta$.
The endpoint formula therefore proves (a)--(c).
The inner branches in (b) meet at the exact Dirac
wavenumber $k_0$ and are excluded from the
single-band Zak phase. Statement (d) follows directly from
Proposition~\ref{prop:zak}.
\end{proof}

\section{Concluding remarks}\label{sec:conclusion}
In this paper, we have shown that the periodic one-dimensional
problem admits a common description of subwavelength and
nonzero Fabry--P\'erot Bloch resonances through an exact
transfer matrix and a frequency-dependent capacitance
matrix. In fact, the cyclic matching structure identifies
the leading square-root splittings, including their
multiplicities. It also shows why the leading coefficients
lose their quasimomentum dependence when the resonant cycle
is broken by a nonresonant interval. However, exact bands
can acquire dispersion at higher orders. The local resonance analysis also extends to small complex
contrast by analytic continuation.

We have illustrated how the one-dimensional exterior intervals play a key role. We have shown that its bounded spacing components
contribute Dirichlet frequencies to the limiting set
\(E\), and their coincidence with interior Neumann
frequencies produces a singular coupling. This mechanism
complements the regular frequency-dependent reduction used
away from the exterior poles. Uniform spectral exclusion
on compact sets outside the limiting set \(E\) yields
a direct bandgap statement, whereas the eigenmode formulas
at full resonance exhibit the different amplitudes in
resonators and spacings.

A complete classification of all Dirac points near each
limiting wavenumber has been obtained for \(N=2\), including
both points fixed at the reference wavenumber and
contrast-dependent crossings. This classification is based on an exact Dirac criterion, which distinguishes exact crossings from coincidences of leading asymptotic coefficients.
For completely symmetric arrays, a complete classification
is also given for arbitrary \(N\), with band folding
clarifying the role of a smaller primitive period.
The zero-frequency constant mode does not constitute
a Dirac point. The weighted Zak phase has also been constructed for simple
isolated bands, and its quantization under reflection
symmetry has been proved. Explicit resonance and eigenmode
expansions for all nonempty resonant configurations when
\(N=2\) are provided in Appendix~\ref{app:asymptotics N2},
under the stated simplicity assumptions.

Further analysis is needed to describe eigenmodes for general \(N\) and arbitrary proper resonant subsets,
to compute higher-order dispersion of the \(\O(\delta)\)
branches, and to make the reductions uniform near changes
in band multiplicity. Such estimates would also provide
a basis for comparing continuum and reduced Zak phases,
studying interfaces, and quantifying the opening of gaps
when the smaller-period symmetry is broken.
A separate high-frequency analysis would be required to
control the constants as the reference wavenumber tends
to infinity. These interesting problems will be the subject
of future work.

\section*{Acknowledgments}
This work was partially supported by the National Key R\&D Program of China grant number 2024YFA1016000 and the Fundamental Research Funds for the Central Universities grant number 226-2025-00192.

\section*{Data Availability Statement}
The figure-generation and verification script is supplied with the LaTeX source. No experimental data were used.

\section*{Conflict of interest} 
The authors have no conflicts of interest to declare. 

\appendix 

\section{Resonance and eigenmode asymptotics for
\texorpdfstring{$N=2$}{N=2}}
\label{app:asymptotics N2}
For general \(N\), Theorem~\ref{thm: k(delta) expand}
determines the numbers of local resonances at the
\(\delta^{1/2}\) and \(\delta\) scales, counted with
multiplicity, and uses \(\mathcal C^\gamma(k_0)\) to describe
the former, but does not give explicit expansions for
the latter.
Likewise, Theorem~\ref{thm: eigenmode n=2N} gives eigenmode
asymptotics only when \(n(k_0)=2N\).
For \(N=2\), the smaller system allows explicit calculations
for each configuration of \(I(k_0)\).
To complement the general results
and facilitate the analysis of Dirac points and Zak phases
in Sections~\ref{sec:dirac} and~\ref{sec: Zak phase}, we give explicit expansions for the remaining
cases. Completeness throughout follows from
Theorem~\ref{thm: k(delta) expand}. The cases considered here
cover all nonempty configurations of \(I(k_0)\), subject to
the stated simplicity assumptions. For a simple isolated band,
Proposition~\ref{prop:zak} determines the Zak phase from the
endpoint reflection eigenvalues, which can be identified from
nonzero leading profiles of definite parity with controlled
remainders.

 We use $I(k_0)$ and $m_j$ as defined in
(\ref{equ: I def}) and (\ref{equ:defmj}), respectively. Throughout, the geometric and material parameters and the real Bloch
parameter $\alpha$ are fixed, and $\delta\to0$. Put
\[
\e^{\i\gamma L}=(-1)^{\sum_{j\in I(k_0)}m_j}\e^{\i\alpha L},
\qquad
\theta_1=\frac{1}{r_1^2\ell_1s_1},\quad
\theta_2=\frac{1}{r_2^2\ell_2s_1},\quad
\theta_3=\frac{1}{r_2^2\ell_2s_2},\quad
\theta_4=\frac{1}{r_1^2\ell_1s_2}.
\]
The eigenmodes are suitably normalized, with uniform field
remainders on the closed period cell.
Expansions involving effective eigenvectors are restricted
to simple effective eigenvalues. Repeated effective eigenvalues
require higher-order analysis and are not considered here.

\subsection{Cyclic relabelling and exact reconstruction}
\label{app:N2 reconstruction}

It suffices to give representatives of the configurations under a
cyclic exchange of the two resonators. This exchange sends
$(1,2,3,4)$ to $(3,4,1,2)$ and interchanges
$(r_1,\ell_1,s_1)$ with $(r_2,\ell_2,s_2)$.
To apply a representative formula, start the cell at $x_2^-$ and
extend the resulting field by $u(x+L)=\e^{\i\alpha L}u(x)$.
This rule specifies the omitted configurations without changing the
Bloch multiplier.

For completeness, any exact eigenmode can be reconstructed from its
initial data, including at exceptional parameter values. Let $a_j=u(x_j^-)$ and $b_j=u'_-(x_j^-)$, where the derivative is
on the spacing side. On the $j$\textsuperscript{th} resonator,
\begin{equation}\label{eq:app N2 exact resonator}
u(x)=a_j\cos(r_jk(x-x_j^-))
+\frac{\delta b_j}{r_jk}\sin(r_jk(x-x_j^-)).
\end{equation}
Write $c_j=u(x_j^+)$ and $d_j=u'_+(x_j^+)$. Then
\[
c_j=a_j\cos(r_j\ell_jk)
+\frac{\delta b_j}{r_jk}\sin(r_j\ell_jk),\qquad
d_j=-\frac{r_jk}{\delta}a_j\sin(r_j\ell_jk)
+b_j\cos(r_j\ell_jk).
\]
On the following spacing,
\begin{equation}\label{eq:app N2 exact spacing}
u(x)=c_j\cos(k(x-x_j^+))
+\frac{d_j}{k}\sin(k(x-x_j^+)),
\end{equation}
and
\[
a_{j+1}=c_j\cos(ks_j)+\frac{d_j}{k}\sin(ks_j),\qquad
b_{j+1}=-kc_j\sin(ks_j)+d_j\cos(ks_j).
\]
The Bloch condition is $(a_3,b_3)^\top=\e^{\i\alpha L}(a_1,b_1)^\top$.
At $k=0$, the formulas are interpreted by continuity.
These relations follow directly from the interval equations and
transmission conditions and also fix all phase factors in the
expansions below. The endpoint notation in this subsection is local and is
independent of the effective eigenvector notation used below.

\subsection{A single resonant phase}

\begin{proposition}\label{prop:app N2 single}
Suppose \(k_0\ne0\) and \(|I(k_0)|=1\).
By Theorem~\ref{thm: k(delta) expand}, there is exactly
one local resonance at the \(\delta\) scale.
Its resonance and eigenmode expansions are as follows.
\begin{enumerate}
\item[(i)] If $I(k_0)=\{1\}$, then
\[
k(\delta)=k_0+ \underbrace{
\frac{\cot(s_1k_0)+\cot(s_2k_0)}{r_1^2\ell_1}}_{:=\mu_1} \delta
+\mc O(\delta^2),
\]
and an associated eigenmode satisfies
\[
u_\delta(x)=
\begin{cases}
\cos(r_1k_0(x-x_1^-)),&x\in[x_1^-,x_1^+],\\
(-1)^{m_1}\dfrac{\sin(k_0(x_2^--x))}{\sin(s_1k_0)},
&x\in[x_1^+,x_2^-],\\
0,&x\in[x_2^-,x_2^+],\\
\e^{\i\alpha L}\dfrac{\sin(k_0(x-x_2^+))}{\sin(s_2k_0)},
&x\in[x_2^+,x_3^-]
\end{cases}
+\mc O(\delta).
\]
\item[(ii)] If $I(k_0)=\{4\}$, then
\[
k(\delta)=k_0+
\frac{1}{s_2}\left[
\frac{\cot(r_1\ell_1k_0)}{r_1}
+\frac{\cot(r_2\ell_2k_0)}{r_2}\right]\delta
+\mc O(\delta^2),
\]
and
\[
u_\delta(x)=
\begin{cases}
\sin(k_0(x-x_2^+)),&x\in[x_2^+,x_3^-],\\
0,&\text{otherwise}
\end{cases}
+\mc O(\delta).
\]
The cases $I(k_0)=\{3\}$ and $\{2\}$ follow from (i) and (ii),
respectively, by cyclic relabelling.
\end{enumerate}
\end{proposition}
\begin{proof}
Substituting \(k=k_0+\delta z\) into the rescaled dispersion
equation gives, at \(\delta=0\), the linear equations
\[
r_1^2\ell_1z=\cot(s_1k_0)+\cot(s_2k_0),
\qquad
s_2z=\frac{\cot(r_1\ell_1k_0)}{r_1}
+\frac{\cot(r_2\ell_2k_0)}{r_2},
\]
for (i) and (ii), respectively.
Their roots are simple, so the implicit function theorem
gives the stated resonance expansions.

As for (i), we normalize \(u_\delta(x_1^-)=1\).
Since all other phase sines are nonzero, the nonresonant
endpoint equations are invertible, giving bounded exterior
derivatives and \(u_\delta=\mc O(\delta)\) on the second
resonator. Exact propagation across the first resonator gives
\[
u_\delta(x_1^+)=(-1)^{m_1}+\mc O(\delta^2),
\qquad
(-1)^{m_1}u'_{\delta,+}(x_1^+)
-u'_{\delta,-}(x_1^-)
=-r_1^2\ell_1k_0\mu_1+\mc O(\delta).
\]
The spacing solutions with these endpoint values give the
displayed sine interpolants, while the first resonator
has the displayed cosine profile, all with uniform
\(\mc O(\delta)\) errors.

As for (ii), we normalize \(u'_{\delta,+}(x_2^+)=k\).
Invertibility of the nonresonant endpoint equations gives
fields of order \(\delta\) on the other three intervals.
The resonator propagation and Bloch conditions yield
\[
u_\delta(x_2^+)
=-\frac{\delta}{r_2}\cot(r_2\ell_2k_0)+\mc O(\delta^2),
\qquad
u_\delta(x_3^-)
=(-1)^{m_4}\frac{\delta}{r_1}
\cot(r_1\ell_1k_0)+\mc O(\delta^2).
\]
The exact solution on the second spacing is
\[
u_\delta(x)=u_\delta(x_2^+)\cos(k(x-x_2^+))
+\sin(k(x-x_2^+)),
\]
which gives the stated sine profile with a uniform
\(\mc O(\delta)\) remainder.
\end{proof}

\subsection{Two adjacent resonant phases}

\begin{proposition}\label{prop:app N2 adjacent}
Suppose \(k_0\ne0\) and \(I(k_0)\) consists of two
adjacent indices.
By Theorem~\ref{thm: k(delta) expand}, there are exactly
two local resonances at the \(\delta^{1/2}\) scale
and none at the \(\delta\) scale.
If $I(k_0)=\{1,2\}$, they satisfy
\[
k^\pm(\delta)=k_0\pm\sqrt{\theta_1\delta}+\mc O(\delta),
\]
whereas for $I(k_0)=\{2,3\}$,
\[
k^\pm(\delta)=k_0\pm\sqrt{\theta_2\delta}+\mc O(\delta).
\]
In both cases, each branch has an eigenmode of the form
\[
u_\delta^\pm(x)=
\begin{cases}
\sin(k_0(x-x_1^+)),&x\in[x_1^+,x_2^-],\\
0,&\text{otherwise}
\end{cases}
+\mc O(\delta^{1/2}).
\]
The cases $\{3,4\}$ and $\{4,1\}$ follow by cyclic relabelling.
\end{proposition}

\begin{proof}
The nonzero effective eigenvalue is $\theta_1$ for $\{1,2\}$
and $\theta_2$ for $\{2,3\}$, as follows from
\[
\begin{pmatrix}\theta_1&-\theta_1\\0&0\end{pmatrix},
\qquad
\begin{pmatrix}0&0\\-\theta_2&\theta_2\end{pmatrix},
\]
respectively. Theorem~\ref{thm: k(delta) expand} gives the two stated
branches and excludes additional branches at the
\(\delta\) scale. For these branches,
(\ref{eq:app N2 exact resonator})--(\ref{eq:app N2 exact spacing}),
with $k-k_0=\mc O(\delta^{1/2})$, give a nonzero leading sine
profile on the resonant spacing and fields of order
$\delta^{1/2}$ elsewhere.
\end{proof}

\subsection{Two nonadjacent resonant phases}

We consider \(I(k_0)=\{1,3\}\) and \(\{2,4\}\), for which
the resonance shifts are of order \(\delta\).
The following expansions are stated for simple effective
eigenvalues.

\begin{proposition}\label{prop:app N2 opposite}
Let \(k_0\ne0\).
For each simple eigenvalue \(\lambda\) of the effective
matrix specified below, there exist \(\varepsilon,\delta_0>0\)
such that, for every \(0<\delta<\delta_0\), there is exactly
one resonance satisfying
\[
\left|\frac{k-k_0}{\delta}-\lambda\right|<\varepsilon.
\]
This resonance defines a branch \(k(\delta)\) with $k(\delta)=k_0+\lambda\delta+\mc O(\delta^2)$.
After suitable normalization, its eigenmode has the following
expansion, with a uniform remainder on the period cell.
\begin{enumerate}
\item[(i)]
Suppose \(I(k_0)=\{1,3\}\). Define $\mc L=\operatorname{diag}(r_1^2\ell_1,r_2^2\ell_2)$ and
\[
C=
\begin{pmatrix}
\cot(s_1k_0)+\cot(s_2k_0)
&
-\csc(s_1k_0)-\e^{-\i\gamma L}\csc(s_2k_0)
\\
-\csc(s_1k_0)-\e^{\i\gamma L}\csc(s_2k_0)
&
\cot(s_1k_0)+\cot(s_2k_0)
\end{pmatrix}.
\]
If \((\lambda,(a_1,a_3)^\top)\) is a simple eigenpair
of \(\mc L^{-1}C\), then
\[
u_\delta(x)=
\begin{cases}
a_1\cos(r_1k_0(x-x_1^-)),
&x\in[x_1^-,x_1^+],\\
(-1)^{m_1}
\dfrac{a_1\sin(k_0(x_2^--x))
+a_3\sin(k_0(x-x_1^+))}{\sin(s_1k_0)},
&x\in[x_1^+,x_2^-],\\
(-1)^{m_1}a_3\cos(r_2k_0(x-x_2^-)),
&x\in[x_2^-,x_2^+],\\
\e^{\i\alpha L}
\dfrac{a_3\e^{-\i\gamma L}\sin(k_0(x_3^--x))
+a_1\sin(k_0(x-x_2^+))}{\sin(s_2k_0)},
&x\in[x_2^+,x_3^-]
\end{cases}
+\mc O(\delta),
\]
where
\(\e^{\i\gamma L}=(-1)^{m_1+m_3}\e^{\i\alpha L}\).

\item[(ii)]
Suppose \(I(k_0)=\{2,4\}\). Define $\mc S=\operatorname{diag}(s_1,s_2)$ and
\[
D=
\begin{pmatrix}
\dfrac{\cot(r_1\ell_1k_0)}{r_1}
+\dfrac{\cot(r_2\ell_2k_0)}{r_2}
&
-\e^{-\i\gamma L}\dfrac{\csc(r_1\ell_1k_0)}{r_1}
-\dfrac{\csc(r_2\ell_2k_0)}{r_2}
\\[6pt]
-\e^{\i\gamma L}\dfrac{\csc(r_1\ell_1k_0)}{r_1}
-\dfrac{\csc(r_2\ell_2k_0)}{r_2}
&
\dfrac{\cot(r_1\ell_1k_0)}{r_1}
+\dfrac{\cot(r_2\ell_2k_0)}{r_2}
\end{pmatrix}.
\]
If \((\lambda,(b_2,b_4)^\top)\) is a simple eigenpair
of \(\mc S^{-1}D\), then
\[
u_\delta(x)=
\begin{cases}
b_2\sin(k_0(x-x_1^+)),
&x\in[x_1^+,x_2^-],\\
(-1)^{m_2}b_4\sin(k_0(x-x_2^+)),
&x\in[x_2^+,x_3^-],\\
0,&\text{on the resonators}
\end{cases}
+\mc O(\delta),
\]
where
\(\e^{\i\gamma L}=(-1)^{m_2+m_4}\e^{\i\alpha L}\).
\end{enumerate}
Both effective matrices have real eigenvalues. If their two
eigenvalues are distinct, the corresponding branches exhaust
the resonances near \(k_0\).
\end{proposition}

\begin{proof}
Put \(k=k_0+\delta z\).
For (i), the leading resonator endpoint values are
\(a_1\) and \((-1)^{m_1}a_3\).
Solving the nonresonant spacings gives the displayed sine
interpolants. Matching their derivatives with the resonator
propagation relations gives
\[
C(a_1,a_3)^\top=z\mc L(a_1,a_3)^\top.
\]
For (ii), the leading spacing sine amplitudes are
\(b_2\) and \((-1)^{m_2}b_4\).
Solving the nonresonant resonators in terms of their endpoint
derivatives and expanding the spacing phases gives
\[
D(b_2,b_4)^\top=z\mc S(b_2,b_4)^\top.
\]
Both effective matrices are similar to Hermitian matrices,
through \(\mc L^{1/2}\) and \(\mc S^{1/2}\), respectively,
so their eigenvalues are real.

The rescaled exact dispersion equation reduces at
\(\delta=0\), up to a nonzero factor, to the corresponding
effective characteristic equation.
At a simple root \(\lambda\), the implicit function theorem
gives a unique \(z(\delta)=\lambda+\mc O(\delta)\).
Fixing a nonzero eigenvector component gives its
\(\mc O(\delta)\) perturbation, and the exact reconstruction
yields the uniform mode expansion.
\end{proof}

\begin{remark}
Repeated effective eigenvalues require higher-order
eigenmode analysis. The results on the Zak phase based on these
expansions therefore exclude effective degeneracies at
\(\alpha=0,\pi/L\).
\end{remark}

\subsection{Three resonant phases}

\begin{proposition}\label{prop:app N2 triple}
Suppose \(k_0\ne0\) and \(|I(k_0)|=3\).
By Theorem~\ref{thm: k(delta) expand}, there are exactly
two local resonances at the \(\delta^{1/2}\) scale
and one at the \(\delta\) scale.
Their expansions are as follows.
\begin{enumerate}
\item[(i)] For $I(k_0)=\{1,2,3\}$, the two former branches are
\[
k^\pm(\delta)=k_0\pm\sqrt{(\theta_1+\theta_2)\delta}
+\mc O(\delta),\quad
u_\delta^\pm(x)=
\begin{cases}
\sin(k_0(x-x_1^+)),&x\in[x_1^+,x_2^-],\\
0,&\text{otherwise}
\end{cases}
+\mc O(\delta^{1/2}).
\]
The remaining branch satisfies
\[
k(\delta)=k_0+
\frac{2[\cos(s_2k_0)-\cos(\gamma L)]}
{(r_1^2\ell_1+r_2^2\ell_2)\sin(s_2k_0)}\delta
+\mc O(\delta^2).
\]
Its eigenmode is
\[
u_\delta(x)=
\begin{cases}
\cos(r_1k_0(x-x_1^-)),&x\in[x_1^-,x_1^+],\\
(-1)^{m_1}[\cos(k_0(x-x_1^+))+h\sin(k_0(x-x_1^+))],
&x\in[x_1^+,x_2^-],\\
(-1)^{m_1+m_2}\cos(r_2k_0(x-x_2^-)),&x\in[x_2^-,x_2^+],\\
\dfrac{(-1)^{m_1+m_2+m_3}\sin(k_0(x_3^--x))
+\e^{\i\alpha L}\sin(k_0(x-x_2^+))}{\sin(s_2k_0)},
&x\in[x_2^+,x_3^-]
\end{cases}
+\mc O(\delta),
\]
where
\[
h=\frac{r_2^2\ell_2-r_1^2\ell_1}{r_2^2\ell_2+r_1^2\ell_1}
\frac{\cos(s_2k_0)-\cos(\gamma L)}{\sin(s_2k_0)}
+\i\frac{\sin(\gamma L)}{\sin(s_2k_0)}.
\]

\item[(ii)] For $I(k_0)=\{2,3,4\}$, the two former branches are
\[
k^\pm(\delta)=k_0\pm\sqrt{(\theta_2+\theta_3)\delta}
+\mc O(\delta).
\]
Using the rescaled coordinate $Y(x)$ defined in
(\ref{equ: Y(x) def}), their eigenmodes satisfy
\[
u_\delta^\pm(x)=
\begin{cases}
s_2\sin(k_0(Y(x)-Y(x_1^+))),&x\in[x_1^+,x_2^-],\\
-s_1\sin(k_0(Y(x)-Y(x_1^+))),&x\in[x_2^+,x_3^-],\\
0,&\text{on the resonators}
\end{cases}
+\mc O(\delta^{1/2}).
\]
The remaining branch satisfies
\[
k(\delta)=k_0+
\frac{2[\cos(r_1\ell_1k_0)-\cos(\gamma L)]}
{r_1(s_1+s_2)\sin(r_1\ell_1k_0)}\delta
+\mc O(\delta^2),
\]
and has an eigenmode
\[
u_\delta(x)=
\begin{cases}
\sin(k_0(Y(x)-Y(x_1^+))),
&x\in[x_1^+,x_2^-]\cup[x_2^+,x_3^-],\\
0,&\text{on the resonators}
\end{cases}
+\mc O(\delta).
\]
The cases $\{1,3,4\}$ and $\{1,2,4\}$ follow from (i) and (ii),
respectively, by cyclic relabelling.
\end{enumerate}
\end{proposition}
\begin{proof}
The nonzero effective eigenvalues are
\(\theta_1+\theta_2\) in (i) and
\(\theta_2+\theta_3\) in (ii), giving the stated
\(\delta^{1/2}\)-scale branches by
Theorem~\ref{thm: k(delta) expand}. 

For (i), exact reconstruction gives the leading sine profile
on the first spacing for these two branches, with fields
of order \(\delta^{1/2}\) elsewhere.
For the \(\delta\)-scale branch, put \(k=k_0+\delta z\)
and normalize \(u_\delta(x_1^-)=1\).
Propagation gives
\(u_\delta(x_2^-)=(-1)^{m_1+m_2}+\mc O(\delta)\).
Derivative matching across both resonators gives, at
leading order,
\[
(r_1^2\ell_1+r_2^2\ell_2)z
=\frac{2[\cos(s_2k_0)-\cos(\gamma L)]}{\sin(s_2k_0)}.
\]
The second spacing and the Bloch condition determine the
left exterior derivative of the first resonator.
Propagation across that resonator then yields
\[
\frac{(-1)^{m_1}u'_{\delta,+}(x_1^+)}{k_0}
=
\frac{\cos(s_2k_0)-\e^{-\i\gamma L}}{\sin(s_2k_0)}
-r_1^2\ell_1z+\mc O(\delta).
\]
Substituting the leading root gives \(h\), and hence the
displayed mode.

For (ii), incorporate the phase signs into
\(\sin(k_0(Y(x)-Y(x_1^+)))\), and denote its two spacing
amplitudes by \(b_1,b_2\).
Propagation across the two
spacings and the second resonator gives the endpoint
matching relation $(k-k_0)(s_1b_1+s_2b_2)
=\mc O((k-k_0)^2+\delta)$.
Since the first resonator is nonresonant, its endpoint
values are \(\mc O(\delta)\).
On the \(\delta^{1/2}\)-scale branches this implies
\(s_1b_1+s_2b_2=0\) at leading order.
The normalization \(b_1=s_2,\ b_2=-s_1\) gives the stated
profile.

On the \(\delta\)-scale branch, the second resonator field
and its phase-adjusted exterior derivative change are
\(\mc O(\delta)\), so \(b_1=b_2\) at leading order.
Normalize their common value to one.
The remaining matching condition gives
\[
(s_1+s_2)z
=\frac{2[\cos(r_1\ell_1k_0)-\cos(\gamma L)]}
{r_1\sin(r_1\ell_1k_0)},
\]
and exact reconstruction gives the stated mode. The simple roots of the rescaled equations give the
resonance remainders by the implicit function theorem.
Exact reconstruction gives the uniform field remainders.
\end{proof}
\subsection{Four resonant phases}

When \(I(k_0)=\{1,2,3,4\}\) and \(k_0\ne0\), the resonance
and eigenmode expansions have already been given in
Theorems~\ref{thm: k(delta) expand} and
\ref{thm: eigenmode n=2N}, respectively.
Although \(I(0)=\{1,2,3,4\}\) as well, the eigenmode
asymptotics at zero are different.
Here \(\gamma=\alpha\), and
\[
\mc C^\alpha(0)=
\begin{pmatrix}
\dfrac{s_1^{-1}+s_2^{-1}}{r_1^2\ell_1}
&
-\dfrac{s_1^{-1}+\e^{-\i\alpha L}s_2^{-1}}{r_1^2\ell_1}
\\[6pt]
-\dfrac{s_1^{-1}+\e^{\i\alpha L}s_2^{-1}}{r_2^2\ell_2}
&
\dfrac{s_1^{-1}+s_2^{-1}}{r_2^2\ell_2}
\end{pmatrix}.
\]
For completeness, we give the corresponding mode expansion
for a simple eigenvalue. For the proof, see \cite[Lemma 4.10]{ammari2023edge}.

\begin{proposition}\label{prop:lowmodes}
Let \(\lambda\) be a simple eigenvalue of
\(\mc C^\alpha(0)\), with eigenvector
\(\bm a=(a_1,a_2)^\top\).
For \(\lambda>0\), let
\(k(\delta)=\pm\sqrt{\lambda\delta}+\mc O(\delta)\)
be the corresponding resonance given by
Theorem~\ref{thm: k(delta) expand}, and let \(u_\delta\)
be its eigenmode.
Then, after suitable normalization,
\[
u_\delta(x)=a_1V_1(x)+a_2V_2(x)+\mc O(\delta^{1/2}),
\]
uniformly on \([x_1^-,x_3^-]\), where
\[
V_1(x)=
\begin{cases}
1,&x_1^-\leq x\leq x_1^+,\\
\dfrac{x_2^--x}{s_1},&x_1^+\leq x\leq x_2^-,\\
\e^{\i\alpha L}\dfrac{x-x_2^+}{s_2},
&x_2^+\leq x\leq x_3^-,\\
0,&\text{otherwise},
\end{cases}
\qquad
V_2(x)=
\begin{cases}
\dfrac{x-x_1^+}{s_1},&x_1^+\leq x\leq x_2^-,\\
1,&x_2^-\leq x\leq x_2^+,\\
\dfrac{x_3^--x}{s_2},&x_2^+\leq x\leq x_3^-,\\
0,&\text{otherwise}.
\end{cases}
\]
For \(\lambda=0\), the same conclusion holds for the
exact constant mode at \(k=0\).
\end{proposition}

% \begin{proof}
% For $\lambda>0$, normalize the interface state $(u,u')$
% at one exterior interface to have norm one.
% Since $k^2=\O(\delta)$, the resonator transfer matrix
% in the exterior-flux coordinates $(u,u'/\delta)$ has
% uniformly bounded entries, as does the spacing matrix
% in $(u,u')$. Thus, propagation bounds all endpoint
% values and exterior derivatives uniformly.
% The interface conditions give $u'=\O(\delta)$ inside
% both resonators. Write $a_j(\delta)=u(x_j^-)$,
% $j=1,2$; integration of the interior equation gives
% \[
%  u'_+(x_j^+)-u'_-(x_j^-)
%  =-\frac{k^2}{\delta}\,r_j^2\ell_j a_j(\delta)
%  +\O(\delta),\qquad j=1,2.
% \]
% The derivatives on both spacings equal their linear
% interpolation slopes up to $\O(\delta)$. Thus,
% \[
%  (\mathcal C^\alpha(0)-\lambda I)\bm a(\delta)
%  =\O(\sqrt\delta),
%  \qquad
%  \bm a(\delta)=(a_1(\delta),a_2(\delta))^\top.
% \]
% The vector $\bm a(\delta)$ is bounded away from zero,
% since the spacing derivatives are determined by its
% endpoint values up to $\O(\delta)$ and the original
% state was normalized. Projection onto the simple
% eigenspace, followed by a bounded scalar normalization,
% gives $\bm a(\delta)=\bm a+\O(\sqrt\delta)$.
% Interior integration and spacing interpolation give
% the stated uniform estimate.
% For $\lambda=0$, necessarily $\alpha=0$ modulo
% $2\pi/L$; the exact mode is constant and the assertion
% follows directly.
% \end{proof}

\bibliographystyle{plain}
\bibliography{highcon}

\end{document}